\documentclass[12pt,reqno]{amsart}
\usepackage[left=3cm, right=3cm, top=2cm, bottom=2cm]{geometry}
\usepackage{mathrsfs,amsmath,graphicx,color}
\usepackage{amsmath,mathrsfs,amsfonts,amssymb}
\usepackage{cases,pgfplots}
\usepackage{mathtools}
\usepackage{wrapfig}
\usepackage{stmaryrd}
\usepackage{color}
\usepackage{soul}
\usepackage{hyperref}
\numberwithin{equation}{section}
\newcommand{\picdis}[1]{#1}
\newcommand{\blue}[1]{{{#1}}} 

\newtheorem{theorem}{{\bf Theorem}}[section]

\newtheorem{coro}[theorem]{{\bf Corollary}}
\newtheorem{lemma}[theorem]{{\bf Lemma}}
\newtheorem{prop}{{\bf Proposition}}[section]

\renewcommand{\leq}{\leqslant}
 \renewcommand{\geq}{\geqslant}
\newcommand{\e}{\varepsilon }
\newcommand{\p}{\partial}
\renewcommand{\O}{\Omega}
\renewcommand{\div}{\operatorname{div}}
\newcommand{\rot}{\operatorname{rot}}

\newcommand{\tr}{\operatorname{tr}}

\newcommand{\HH}{\mathbf{H}}

\newcommand{\1}{\mathbf{1}}

 \newcommand{\btau}{\boldsymbol{\tau}}
 \newcommand{\bxi}{\boldsymbol{\xi }}
  \newcommand{\bPsi}{\boldsymbol{\Psi }}
  \newcommand{\bphi}{\boldsymbol{\varphi }}
  \newcommand{\bnu}{\boldsymbol{\nu}}
 
\newcommand{\R}{\mathbb{R}}
\newcommand{\nn}{\mathbf{n}}
\newcommand{\dd}{\mathrm{d}}
\newcommand{\uu}{\mathbf{u}}
\newcommand{\vv}{\mathbf{v}}
\newcommand{\ww}{\mathbf{w}}

\renewcommand{\tilde}{\widetilde}

\def\({\left(}
\def\){\right)}
 \def\[{\begin{equation}}
 \def\]{\end{equation}}

\begin{document}

\title{\small{Phase transition    of  an anisotropic  \\Ginzburg--Landau  equation}}

\author{Yuning Liu}
\address{NYU Shanghai, 567 Yangsi W road, Pudong, Shanghai 200126, China,
and NYU-ECNU Institute of Mathematical Sciences at NYU Shanghai, 3663 Zhongshan Road North, Shanghai, 200062, China}
\email{yl67@nyu.edu}

\begin{abstract}  
\normalsize
We study the effective  
geometric motions  of an anisotropic  Ginzburg--Landau  equation   with   a small parameter $\varepsilon>0$ which characterizes    the width of the    transition layer.  For well-prepared initial datum, we show that  as $\varepsilon$ tends to zero 
  the   solutions  will develop a sharp interface limit   which   evolves   under    mean curvature flow. The bulk limits of the solutions  correspond to   a  vector field     $\mathbf{u}(x,t)$ which is   of unit length on  one side of  the interface,  and  is  zero on the other side.  The proof  combines the modulated energy method and weak convergence methods. In particular,  by  a (boundary) blow-up argument we show     that   $\mathbf{u}$   must be   tangent to  the sharp interface.        Moreover, it  solves    a geometric evolution equation for  the   Oseen--Frank model in liquid crystals.

\medskip
\noindent \textbf{Keywords:} modulated energy methods, weak convergence methods, blow-up analysis, mean curvature flow, phase-transition.

\medskip
\noindent \textbf{Mathematical Subject Classification}: 
	53E10, 
		35R35, 
35K58  
	35K57. 


 \end{abstract}
 \maketitle

\tableofcontents



\section{Introduction}
 In the study of liquid crystals one often encounters elastic energies with anisotropy, i.e. energies with distinct coefficients multiplying the square of the divergence and
 the curl of the order parameters. Typical examples involve  the Oseen--Frank model \cite{HardtKinderlehrerLin1986}, Ericksen's model \cite{MR1294333,lin2020isotropic}  and the Landau--De Gennes model \cite{ball2017mathematics}. From a microscopic point of view, the anisotropy of  these   models can be interpreted as  excluded volume potential of   molecular interaction, cf.  \cite{HLWZZ}. Anisotropic models also arise  in the theory of  superconductivity, cf. \cite{MR1313011}.
 The anisotropy brings various  new challenges to  the studies of   both   variational problems and their gradient flows of the aforementioned models. 
   In contrast to  the convergence analysis  of  isotropic models, i.e. the (scalar) Allen--Cahn equations (cf.  \cite{MR1177477,MR1237490,MR1205984,MR1674799,MR2040901,MR2253464,MR2440879,MR3495430}),  the powerful analytic tools such as     maximum principle and   monotonicity formula are not readily established  for anisotropic ones.

The  attempt  of
 this work  is  to study an  anisotropic system modeling  the isotropic-nematic  phase transition of a  liquid crystal droplet.    Let $d\in \{2,3\}$ be the dimension of the physical  domain $\Omega$ with $C^3$ boundary $\p\Omega$.
We consider the   anisotropic   Ginzburg--Landau type  energy
\begin{equation}\label{GL energy}
A_\e (\uu)=\int_{\O} \(\frac{\e }2 \mu |\div \uu|^2+\frac\e 2 |\nabla \uu|^2+\frac{1}\e  F(\uu)  \)\, dx.
\end{equation}
 Here $\uu=(u_1,u_2,u_3):\O\subset \R^d\mapsto  {\R^3}$ is the  order parameter describing the state of the system.  The function $F(\uu)$ is a    double equal-well  potential  which permits the isotropic-nematic  phase transition.  More precisely, it  attains  its 
global minimum value $0$ at $\{ 0\}\cup \mathbb{S}^2$.  
An example  of $F$ is the Chern--Simons-Higgs model 
$F(\uu)= |\uu|^2(1-|\uu|^2)^2$. See for instance   \cite{MR1050529,MR1050530} for the physics and 
  \cite{MR1324400,MR4076075,MR3910590} for the mathematical analysis  of  related   variational problems.     
The parameter  $\e>0 $ denotes the relative intensity of elastic and  bulk energy, which is  usually quite small. The parameter  $\mu>0$ is  material dependent which measures the degree of anisotropy.

The energy \eqref{GL energy} is a simplified case of the full Landau--De Gennes energy (cf. \cite{IyerXuZarnescu2015,MR4272911}).
The variational investigations  of 
 the isotropic-nematic phase transition  involving  \eqref{GL energy} were first done   by Golovaty, Novack, Sternberg and Venkatraman  \cite{MR4076075,MR3910590} in the static case in 2D.  The present  paper is concerned with  the    $L^2$-gradient flow of \eqref{GL energy}, i.e.  the following system. 
\begin{subequations} \label{GL system}
\begin{align}
\p_t  \uu_\e -\mu \nabla (\div \uu_\e) &= \Delta \uu_\e   - \frac 1{\e ^2} D F (\uu_\e  )&&~\text{in}~ \Omega\times (0,T),\label{Ginzburg-Landau}\\
\uu_\e  (x,0)&=\uu_\e ^{in}(x)&&~\text{in}~\Omega,\\
\uu_\e(x,t)&= 0 &&~\text{on }\p \Omega\times (0,T),\label{bc of omega}
\end{align}
 \end{subequations}
 where   $ D F (\uu  )$ is the gradient  of $F(\uu)$ with respect to $\uu$.
 We shall study the small $\e$-asymptotics   of this system with well-prepared  initial datum  $\uu_\e ^{in}$ that undergoes  a sharp transition across  a co-dimensional one   interface  $I_0\subset \R^d$. We shall show  that the energy density  $\frac\e 2 |\nabla \uu_\e|^2+\frac{1}\e  F(\uu_\e)$ will be concentrated on a mean curvature flow $I:=\bigcup_{t\geq 0} I_t\times \{t\}$ starting from $I_0$, namely
\[
\lim_{\e\to 0} \int_{\O} \( \frac\e 2 |\nabla \uu_\e|^2+\frac{1}\e  F(\uu_\e)  \)\, dx=\sigma\mathcal{H}^{d-1}( I_t),\label{intro energy conv}
\]
where $\mathcal{H}^{d-1}$ is the $(d-1)$ dimensional  Hausdorff measure, and $\sigma$ is a positive constant depending on $F$.
Moreover, we shall derive   bulk limit  $\uu:=\lim_{\e\to 0}\uu_\e$  away from $I_t$  and its   boundary condition on $I_t$.

\begin{figure}
 \begin{tikzpicture}[scale = 1.1]
\begin{axis}[
axis equal,
 axis lines=none,
]
\addplot[samples=100, domain=0:2*pi,    red] 
	({8+ 0.7*cos(deg(x))+0.1*cos(5* deg(x))}, {0.1+ 0.7*sin(deg(x))}) node[midway,yshift=4mm]{\tiny$I_t$};
 	\addplot[samples=100, domain=0:2*pi, black] 
	({8+1.5*cos(deg(x))+0.05*cos(5* deg(x))}, {1.5*sin(deg(x))}) 
	node[left,xshift=-20mm]{\tiny${\color{blue}\O_t^+}$}
	node[left,xshift=-1.5mm]{\tiny${\color{blue}\O_t^-}$}
	node[midway,xshift=8mm,yshift=-20mm]{\tiny$ \p\O$};
\end{axis} 
\end{tikzpicture}
\caption{$I_t$ is the interface,  $\O_t^+$ is the nematic  phase and $\O_t^-$ is the isotropic phase.}
\end{figure}

System \eqref{Ginzburg-Landau} is a vectorial and anisotropic generalization of the scalar   parabolic Allen--Cahn  equation.  In the scalar case,  there have been many   developments  on its  co-dimensional one    limit   to the (two-phase) mean curvature flow during the last two decades. Here we   mention two classes of results  and postpone  the  discussions of  some  others in the sequel.  One is  the convergence to a Brakke's flow by Ilmanen \cite{MR1237490} using a version of Huisken's monotonicity formula \cite{MR1030675} and  tools from geometric measure theory. See also \cite{MR1425577,MR1803974,MR2040901,MR2253464,MR3495430,MR2440879} and the references therein for further renovations. Despite of its energetic nature,  a major difficulty of such an approach is   the control of the  {so-called} {\it  discrepancy measure}, and   almost all  existing literatures using this approach    rely   crucially  on a version of   Modica's maximum principle \cite{MR803255}. 
   However, it is not clear  whether  Modica's maximum principle holds for  elliptic/parabolic  systems. 
  Another approach, which relies more  on   parabolic comparison principle,  is  the global in time   convergence  towards the  viscosity solution built  by    Chen--Giga-Goto \cite{MR1100211} and independently by Evans--Spruck   \cite{MR1100206}.  Such an approach has been implemented by Evans--Soner--Souganidis \cite{MR1177477}. One can also refer to   \cite{MR1205984,MR1674799} and   the references therein for further discussions.  These two approaches both give global in time (weak) convergences  to   weakly defined  solutions of   the mean curvature flow  (up to their    extinction times). However, as  their technics  involve  parabolic maximum principle  in one way or another,  it is not clear how to use them to attack   vectorial models   in general. It is worth mentioning that for radially symmetric initial datum, Bronsard--Stoth \cite{MR1443865}   have  obtained  global in time convergence to the mean curvature flow of planar circles.

 To the best of our  knowledge, there are mainly two approaches   to rigorously justify   the convergence of  the vectorial Allen--Cahn equations, both  assuming  {that}  the limiting interface motion   has a (local in time)  classical solution. Compared with the aforementioned methods,  which lead to global in time (weak) convergence, they have quite different natures.
 The first  approach  is the   asymptotic  expansion technics  developed by De Mottoni--Schatzman \cite{MR1672406} and  by Alikakos--Bates--Chen \cite{MR1308851}.  It  has been used   recently  by Fei--Wang--Zhang--Zhang   \cite{MR4059996} to study the isotropic-nematic phase transition in  liquid crystals, and by Fei--Lin--Wang--Zhang \cite{Fei2023aa}  {to study} matrix-valued   Allen--Cahn equations. 
The second  approach, which also assumes a classical  solution of the limiting interface motion (but not   the limiting   flows in the bulk regions),   is the  {modulated energy}  method developed  by Fischer--Laux--Simon  \cite{fischer2020convergence}.   Such a method is  motivated by Jerrard--Smets \cite{MR3353807} and Fischer--Hensel \cite{MR4072686}, and has been    generalized to  a  matrix--valued  model  by Laux--Liu  \cite{MR4284534}. 

In the present  work, we shall   use  the methods  {employed} in  \cite{fischer2020convergence,MR4284534} to    derive the energy convergence \eqref{intro energy conv} and  the bulk limit  $\uu=\lim_{\e_k\to 0}\uu_{\e_k}$ by establishing  two  modulated energy inequalities.  Moreover, the derivation of  the anchoring boundary condition of $\uu$ (see \eqref{anchoring bc} below) uses a blow-up argument, which  is  inspired by a recent work of Lin--Wang \cite{lin2020isotropic}.
There the authors  have studied   isotropic-nematic phase transitions in the static case  based  on an anisotropic Ericksen's model.

 To state the main result, we assume that 
\begin{equation}\label{interface}
 \begin{split}
 I=\bigcup_{t\in [0,T]} I_t \times \{t\}~
 \text{ is a smoothly evolving}  \\
  (d-1)\text{-dimensional submanifold       in}~\O,
 \end{split}
\end{equation}
 starting from a   $(d-1)$-dimensional   submanifold    $I_0\subset  \O$. 
 Here a $(d-1)$-submanifold refers to an embedded closed  smooth surface when $d=3$ and   curve when $d=2$.
 
 Let $\O^+_t$ be the domain enclosed by $I_t$, and   $d_I(x,t)$ be   the signed-distance  from $x$ to $I_t$ which takes negative  values in  $\O^-_t$, and positive  values  in $\O^+_t=\O\backslash \overline{\O^-_t}$. Equivalently,  
 \begin{equation}\label{def:omegapm}
\O^\pm_t:= \{x\in\Omega\mid d_I(x,t)\gtrless0\}.
\end{equation}
For   $\delta>0$,  the (open)  $\delta$-neighborhood of $I_t$ is denoted by
\begin{equation}
B_\delta(I_t):= \{x\in\Omega\mid   | d_I(x,t)|<\delta\}.
\end{equation}
Let  $\delta_0\in (0,1)$  be  {a  sufficiently small number} so that  the nearest point projection $$P_{I}(\cdot,t): B_{4\delta_0}(I_t) \rightarrow I_t$$ is smooth for any $t\in [0,T]$, and  that  the interface \eqref{interface} stays at least $4\delta_0$  {distant}  away  from the physical boundary  $\p\O$. A further  description of the geometry can be found in  Subsection \ref{subsectionmodulated} or  {in} \cite{MR2754215}.

The first step to  study  the singular limit of \eqref{GL system} is to construct a  {modulated energy}   which encodes a distance between  the energy in  \eqref{GL energy}  and   an energy  corresponding to   the moving interface $I_t$ in   \eqref{interface}.
Following \cite{MR3353807,MR4072686,fischer2020convergence},  we  define an    extension of     the   {inward}   normal vector $\nn(\cdot,t)$ of $I_t$ by 
$$\bxi (x,t):=\phi \( \frac{d_I(x,t)}{\delta_0}\)\nabla d_I(x,t)\quad \text{ for } x\in \O,$$
where  $\phi\in C_c^2( \R;  [0,1]) $ is  an appropriate cut-off function (see \eqref{phi func control} below for its precise definition).
Now we introduce 
\begin{align}\label{entropy}
E_\e  [\uu_\e   | I](t):= & \int_\O \frac{\e }2 \mu |\div \uu_\e (\cdot,t)|^2\, dx\nonumber\\
&+\int_\O\(\frac{\e }{2}\left|\nabla \uu_\e  (\cdot,t)\right|^2+\frac{1}{\e } {F  (\uu_\e  (\cdot,t))}-   \bxi\cdot\nabla \psi_\e (\cdot,t) \)\, dx,
\end{align}
where  
$\psi_\e$ is defined by 
\begin{align}
\psi_\e (x,t):=  \int_0^{| \uu_\e  (x,t)|} g(s)\, ds .\label{psi}
\end{align}
We shall work with a   class of potentials $F(\uu)$ under standard   assumptions (see e.g. \cite{MR1237490,MR1425577}).  That is,
\begin{align}\label{bulk}
  F(\uu) =f(|\uu|)= g^2(|\uu|)/2,
  \end{align}
 where $f$ is a double equal-well potential, namely,
\begin{subequations}\label{bulk assump2}
\begin{align}
&f\in C^\infty( \R_{\geq 0}),\quad f(s)>0\text{ for } s\in  \R_{\geq 0}\backslash \{0,1\},\\
&g\geq 0\text{ and  is  {locally} Lipschitz continuous}, ~g(0)=g(1)=0.\label{g lip}
\end{align}
\end{subequations}
Moreover, the following structural assumptions on $f$ are made:
 \begin{subequations}\label{bulk assump}
\begin{align}
&\exists s_0\in (0,1) \text{ s.t. } f'(s) >0\text{ on } (0,s_0)~\text{ and }~ f'(s) <0\text{ on } ( s_0,1);\label{f increase}\\
&f'(0)= f'(1)=0,\qquad f''(0), f''(1)>0;\label{bulk stable}\\
&  \exists c_0\in (0,1) \text{ s.t. }      2c_0^2 s^2\leq f(s)\leq 2c_0^{-2}s^2\text{ for any } s\geq 100.\label{growth f}
\end{align}
\end{subequations}
 After an appropriate modification  for large $|s|$, the function  $g(s)=|s||s^2-1|$, which corresponds to the Chern--Simons--Higgs potential, satisfies \eqref{bulk assump}.

 To control the bulk errors, we   need 
 another modulated energy: 
\[ B[\uu_\e  | I](t):= \int_\O   \Big(\sigma\chi-\sigma+ 2(\psi_\e-\sigma)^- \Big)\eta\circ d_I  \,\, dx+\int_\O  \( \psi_\e-\sigma\)^+|\eta\circ d_I| \, dx.\label{gronwall2new}\]
Here $\chi(\cdot,t):=\1_{\O_t^+}-\1_{\O_t^-}$, $h^\pm$ denote the positive/negative parts of a function $h$ respectively, and $\eta$ is a truncation of the identity function defined by  
\[\label{truncation eta}
\eta(z):=\left\{
\begin{array}{rl}
z\qquad \text{when } &z \in [-\delta_0, \delta_0],\\
 \delta_0\qquad \text{when } &z\geq \delta_0,\\
 -\delta_0\qquad \text{when } &z\leq -\delta_0. 
\end{array}
\right. \]
Note that  $( \eta\circ d_I)~\chi\geq 0$   in $\O$ due to our convention on  the signed-distance function, and thus  the two  integrands in \eqref{gronwall2new} are both  non-negative. We refer the readers to     the proof of Theorem \ref{sharp l1} below for more details on the positivity of \eqref{gronwall2new}.

Now we state    the main result of this work:
\begin{theorem}\label{main thm}
Let  $d\in \{2,3\}$, and the assumptions    \eqref{bulk assump2} and   \eqref{bulk assump} be in place.  
Assume that the moving interface $I$ in  \eqref{interface} evolves under   mean curvature  flow,   and   the initial datum of \eqref{GL system}  satisfies the following conditions:
\begin{subequations}\label{initial assupms}
\begin{align}
&  {\uu_\e^{in}\in W^{1,2}_0( \O )},\label{bdd0}\\
&A_\e(\uu_\e^{in})\leq c_1,\label{bdd1}\\
  & E_\e  [\uu_\e^{in}  | I_0]+B  [\uu_\e^{in}  | I_0]\leq c_1\e,\label{initial}
\end{align}
\end{subequations}
where     $c_1>0$ is   independent of $\e$.
  Then  there exists  $C_1>0$ independent of $\e$  {such that}
  \begin{align}\label{initial preserve}
 &\sup_{t\in [0,T]}E_\e  [\uu_\e  | I](t)+\sup_{t\in [0,T]}B  [\uu_\e  | I](t)  \leq C_1\e,\\
 &\sup_{t\in [0,T]}\int_{\O}| \psi_\e-\sigma\1_{\O_t^+}| \, dx \leq C_1\e^{1/4}.\label{con psi}
\end{align}
Moreover,  up to extraction of a subsequence   $\e_k\downarrow 0$,    
\begin{align}
 \uu_{\e_k}\xrightarrow{k\to\infty}    \1_{\O_t^+} \uu~&\text{  in }~   C([0,T];L^2_{loc} {(\O\backslash I_t)}),\label{reg limit}
\end{align}
 where     $\uu$ satisfies    the following properties:
\begin{subequations}
\begin{align} 
&\uu\in   L^\infty(0,T;W^{1,6/5}(\O^+_t;\mathbb{S}^2)), ~ \p_t \uu\in  L^2(0,T;   L^{6/5}(\O^+_t)),\label{reguptobdy1}\\
& \uu(x,t)=0 ~\text{ for every }~ t\in [0,T]\text{ and for a.e. }  ~ x\in \O_t^-,\label{u equal 0 region1}\\
&  (\uu\cdot\nn)(x,t)=0 \text{ for  a.e.  } t\in [0,T] \text{ and  for }  \mathcal{H}^{d-1}\text{-a.e. }x\in I_t.\label{anchoring bc}
\end{align}
\end{subequations}

\end{theorem}
 Among   the conditions in  \eqref{initial assupms}, the  {crucial}  one is \eqref{initial}, which is used to obtain the     inequalities  in Theorem \ref{thm close energy with div} and in Theorem \ref{sharp l1} below. 
  To construct   {an}  initial datum satisfying \eqref{initial assupms}, we need  the following result.
\begin{prop}  \label{prop initial data}
Let  $I_0\subset \O$ be a   $(d-1)$-dimensional submanifold.   For   any vector field
 \[\uu^{in}\in W^{1,2}(\O;\mathbb{S}^2) ~\text{ with } ~ \uu^{in}
 |_{I_0}\cdot \nn_{I_0}=0~a.e.~\text{ on }~I_0,\label{initialvectorfield}
 \] 
  there exists $\uu_\e^{in} \in  W^{1,2}_0(\O)\cap L^\infty(\O)$ such that 
\begin{align} \label{initial out}
\begin{cases}
\uu_\e ^{in}=\uu^{in} & \quad \text{in}~   \O^+_0\backslash B_{2\delta_0}(I_0), \\
 \uu_\e^{in}=0 & \quad\text{in }  \O^-_0\backslash B_{2\delta_0}(I_0),
\end{cases}
\end{align}
and     \eqref{initial assupms} holds for   a constant $c_1>0$ which only depends on $I_0$ and $\|\uu^{in}\|_{W^{1,2}(\O)}$. 
  \end{prop}
   
  We comment on the conditions in  \eqref{initialvectorfield}. When $d=3$, $I_0$ is a smooth closed surface in $\O$. Due to topological obstructions, a vector field satisfying \eqref{initialvectorfield} is usually not smooth. For instance, when $I_0$ is diffeomorphic to a 2-sphere, due to the hairy ball theorem, $\uu^{in}|_{I_0}$ must have (at least) one pole.  
  One example of such a pole,  which is  often encountered  in the theory of liquid crystal,   is given by the hedgehog profile. Locally the tangent vector field near such a pole is $C^1$-equivalent to the mapping  $\mathbf{h}(x)=x/|x|: B_1\cap \R^2\to \mathbb{S}^1$. Note that   $\mathbf{h}\in W^{\frac1 2,2}(B_1\cap \R^2)$ but  $\mathbf{h}\notin W^{1,2}(B_1\cap \R^2)$. 
  When $d=2$, there are fewer constraints  to arrange a vector field $\mathbf{f}:I_0\mapsto \mathbb{S}^2\subset \R^3$ that is orthogonal  to the planar curve $I_0\subset \R^2\times \{0\}$.
   In general, using  the extension lemma of Hardt--Lin (cf.  \cite[Lemma 2.2.10]{LinWang2008}),    any tangent vector field    $\mathbf{f}\in W^{\frac 12, 2}(I_0;\mathbb{S}^2)$  has an   extension  $\uu^{in}$   satisfying \eqref{initialvectorfield}.

An immediate  consequence of Theorem \ref{main thm} is the   convergence in  \eqref{intro energy conv}. Indeed, it follows from \eqref{initial preserve} and \eqref{energy bound0} below that  
$\int_{\O} \frac{\e }2 \mu |\div \uu|^2\,dx\xrightarrow{\e\to 0}0$,
and thus  such an   energy  does not   contribute  to the surface  energy in the limit. However,      it forces  $\uu$ to satisfy  the  boundary condition \eqref{anchoring bc}.
Now   applying    integration by parts to the last term of \eqref{entropy}, and then  using \eqref{con psi} and $\bxi|_{\p\O}=0$, we   find 
\begin{align}\label{earlycal1}
&\lim_{\e\to 0} \int_{\O} \( \frac\e 2 |\nabla \uu_\e|^2+\frac{1}\e  F(\uu_\e)  \)\, dx\nonumber\\
=&\lim_{\e\to 0} -\int_\O (\div\bxi) \psi_\e\,dx=-\sigma\int_{\O_t^+}(\div\bxi) \, dx=\sigma \mathcal{H}^{d-1}( I_t).
\end{align}
Note that the last step is due to the  Green's formula.

Under additional assumptions, we can show   that the limit $\uu$ in \eqref{reg limit} solves a geometric evolution equation in the bulk region  $\O^+:=\bigcup_{t\in [0,T]}\O^+_t\times \{t\}$.
 \begin{theorem}\label{main thm oseen frank}
 Let  $d=2$ and   the   assumptions   of Theorem \ref{main thm} be in place. Assume further  that 
\[\label{bulk2} 
\begin{split}
 &f(s)= s^2  \text{ for }  s\leq 1/4; f(s)= (s-1)^2   \text{ for }     s\geq   3/4; \\
 & f(s)\geq   1/{16}  \text{ for } s\in [1/4, 3/4];\\
 &  \sup_{s\in [1/4, 3/4]} |f'(s)|\leq 4.
  \end{split} \]
 Then there exists a sufficiently small $\mu>0$ (independent of $\e$)  {such  that}   the vector field $\uu$ in \eqref{reg limit}  satisfies 
\begin{align}
&\int_\O \p_t \uu\wedge \uu \cdot \bPsi\, dx+\int_\O   (\nabla \uu\wedge \uu)\cdot\nabla\bPsi\, dx\nonumber\\
&=\mu  \int_\O (\div \uu) \Big((\rot \bPsi)\cdot\uu-(\rot \uu)\cdot \bPsi\Big)\, dx\label{weak OF flow}
\end{align}
for almost every $t\in (0,T)$ and for every $ \bPsi\in C^1_c(\O_t^+;\R^3)$.
 \end{theorem}
In the above equation   $\wedge$  is the   wedge product in $\R^3$ and   
$\rot$ is the curl operator.  
The equation \eqref{weak OF flow} is the weak formulation of an    Oseen--Frank flow, written as   
\[\p_t \uu=\Delta \uu+\mu ( {\mathbb{I}_3}-\uu\otimes \uu) \nabla (\div \uu)+|\nabla \uu|^2\uu,\qquad \text{ for } t\in (0,T], x\in\O_t^+.\label{OF flow}\]
It can be   verified that when $\uu$ is sufficiently regular,  then  \eqref{weak OF flow} implies  \eqref{OF flow}. 
It is worth  mentioning  that      equation of the form \eqref{OF flow} is the $L^2$-gradient flow of the   variational problem \[\inf\int_{U}   \(\mu |\div \uu|^2+  |\nabla \uu|^2\)\, dx,\label{OS variational}\]
where the infimum is taken among mappings $\uu\in W^{1,2}(U;\mathbb{S}^2)$ fulfilling   certain boundary conditions on $\p U$. Note that  \eqref{OS variational} is a special case of the full Oseen--Frank model (cf.   \cite{HardtKinderlehrerLin1986}).



This  work will be organized as follows:  In Section \ref{sec entropy}, we shall adapt  the  {modulated energy}  method of   \cite{fischer2020convergence}  to the vectorial and anisotropic system \eqref{GL system}, and then  derive a differential inequality, i.e. Proposition \ref{gronwallprop}. 
Such an inequality was   previously  derived in \cite{MR4284534}  for a matrix-valued  equation. When applied to \eqref{GL system}, it   includes   a term which does not have an obvious sign due to the additional $\div$ term. This problem will be solved in  Section \ref{sec close} during  the proof of  the    inequality in Theorem \ref{thm close energy with div}. This theorem,  which     leads to   the first part of Theorem \ref{main thm},  is a major  novelty of   the present  work,   and  will be employed   in Section \ref{sec level} (see Theorem \ref{sharp l1}) 
to derive the   $L^1$-estimate of $\psi_\e$  in  \eqref{con psi}.  Such an  estimate will  be used  in  Lemma \ref{sharp est of bdy} to identify    appropriate   level sets of $\psi_\e$ which converge to $I_t$ in certain sense. With this key lemma, we derive in Section \ref{anchoring}   the  anchoring boundary condition \eqref{anchoring bc}, and thus finish the proof of Theorem \ref{main thm}.
 Section \ref{sec har} is devoted to  the  proof of Theorem \ref{main thm oseen frank}.  The proof of Proposition \ref{prop initial data} is quite similar to the construction given in \cite{MR4284534}. We present a proof   in Appendix \ref{app initial} for the convenience of the readers. 



\section{Preliminaries}\label{sec entropy}
\subsection{Notation and Conventions}\label{subsecnotation}
  We shall adopt the following conventions throughout the paper.
Unless specified otherwise, $C>0$ is  a generic constant whose value might change from line to line, and will depend    on the geometry of the interface \eqref{interface} but not on $\e$ or $t\in [0,T]$. 
For two square matrices $A$ and $B$,  their   Frobenius inner product is defined by  $A:B:= \tr A^{\mathsf{T}} B$, which  induces the  norm $|A|:=\sqrt{\tr A^{\mathsf{T}} A}$.
We shall also use the following  notation for a vector-valued function $\uu(x,t)=(u_1(x,t),u_2(x,t),u_3(x,t))$ where 
\[x=\begin{cases}
(x_1,x_2,x_3) &\text{ when }d=3,\\
 (x_1,x_2,0) &\text{ when }d=2.
\end{cases}
\]
\[\label{2ddivdef}
\begin{split}
\p_0&=\p_t, \quad \p_i=\p_{x_i} \quad 1\leq i\leq 3,\\
\nabla u_1&=(\p_1 u_1,\p_2 u_1,\p_3 u_1),\quad \div \uu=\sum_{i=1}^3\p_i u_i,\\
\rot \uu&=(\p_2 u_3-\p_3u_2 ,\p_3 u_1-\p_1u_3,\p_1 u_2-\p_2u_1).
\end{split}\]
To ease computations when $d=2$, $\nabla \uu$ will be    understood as the   matrix 
$\begin{pmatrix}
\p_1 u_1& \p_2 u_1 & 0 \\
\p_1 u_2& \p_2 u_2 & 0 \\
\p_1 u_3& \p_2 u_3 & 0
\end{pmatrix}
$, and 
 \[\begin{split}
 &\text{any  planar vector field is understood as a}\\
 &\text{3D vector field with vanishing 3rd component.}
 \end{split}
\]
 In particular, the latter  applies to   the normal and  the   mean curvature vector fields (cf. \eqref{velocity} and  \eqref{def:H} respectively below). 
For  a function of     $\uu$, like $F(\uu)$, its gradient will be denoted by 
$$ D F=(\p_{u_1} F,\p_{u_2} F,\p_{u_3} F).$$

We end this section by  the following assumptions  regarding   various   constants.
 Theorem \ref{main thm} will be proved for any fixed constant $\mu>0$, while Theorem \ref{main thm oseen frank} is valid for a sufficiently small (fixed) $\mu$. To simplify the presentation, we shall assume without loss of generality  that 
\[\mu\in (0,1)\text{ is a fixed constant}.\label{conventionmu}\]
Finally we can  {normalize} $g$ (cf. \eqref{bulk}) to have 
 \[\sigma:= \int_0^1 g(s)\, ds=1.\label{normalization}\]

As the $L^2$-gradient flow of \eqref{GL energy}, the system
\eqref{GL system} enjoys  the   following  energy dissipation law:
  \begin{equation}\label{dissipation}
A_\e (\uu_\e  (\cdot,\hat{T}))+  \int_0^{\hat{T}} \int_\O \e  |\p_t \uu_\e  |^2 \,d x d t=A_\e (\uu_\e ^{in}(\cdot))
\end{equation}
 for   arbitrarily large time      $\hat{T}$.
Combining this with the   theory of gradient flow and  the regularity theory for elliptic system  (cf. \cite{MR2401600,MR4272911}),  one can construct   a unique solution to system  \eqref{GL system} that satisfies 
 $$\uu_\e\in L^2(0,\hat{T};W^{2,2}(\O)\cap W^{1,2}_0(\O))\text{ and }\p_t \uu_\e\in L^2(\O\times (0,\hat{T})).$$
     So for almost every $\hat{t}\in (0,\hat{T})$, we have 
     $$\uu_\e(\cdot,\hat{t})\in W^{2,2}(\O)\hookrightarrow W^{1,6}(\O)\hookrightarrow C^{0,1/2}(\overline{\O}).$$ Under  the   assumption \eqref{growth f}, the nonlinearity of \eqref{Ginzburg-Landau} has a linear growth.   So considering   the system with initial datum $\uu_\e(\cdot,\hat{t})$, and using the H\"{o}lder estimates for parabolic system (cf. \cite{MR1399194}), we  deduce  that \[ \uu_\e \text{ is a classical solution of }\eqref{Ginzburg-Landau} \text{ in  } \O\times (0, \hat{T}].\label{regularityofsolution}\]
     
For initial datum undergoing   phase transitions  near the initial interface $I_0$, formal asymptotic analysis   suggests      that   $\nabla \uu_\e  $ will be   singular  near    $I_t$. However, the global dissipation    law  \eqref{dissipation} is not sufficient to   {yield}  the (strong) convergence of $\uu_\e$, not even  in the domain away from $I_t$. Following a recent work of  Fisher et al. \cite{fischer2020convergence}, we shall  establish  in this section  a differential  inequality which modulates  the concentration  and   leads to  the compactness of solutions in Sobolev spaces.

\subsection{The modulated energy}\label{subsectionmodulated}

We first  set up the geometry of the moving interface $I$ defined in \eqref{interface}.
Under a local parametrization $\bphi_t(s):U\subset \R^{d-1}\to I_t$, the mean curvature  flow reads
\[\p_t \bphi_t(s) = \kappa  \nn \label{csf}\]
where $ \kappa=\kappa (\bphi_t(s),t)$ is the mean  curvature   and $\nn=\nn(\cdot,t): I_t\mapsto \mathbb{S}^{d-1}$  is   the  inward  normal vector.
 For any  $t\in [0,T]$ we assume  that  the nearest-point projection $P_I(\cdot,t):B_{4\delta_0}(I_t)\mapsto I_t$ is smooth  for some sufficiently small $\delta_0\in (0,1)$ which only depends on the geometry of  $I$. Analytically we have $P_I(x,t)  =x-\nabla d_I(x,t) d_I(x,t)$. So for  each fixed $t\in [0,T]$,  any point $x\in B_{4\delta_0}(I_t)$ corresponds to a unique pair $(r,s)$ with    $r=d_I(x,t)$ and $s\in U$, and     the identity
$$d_I\Big(\bphi_t(s)+r\nn(\bphi_t(s),t), t\Big)= r$$ holds with   independent variables  $(r,s,t)$.  
  Differentiating this identity with respect to $r$ and $t$   leads to  the following identities: 
  \[\label{velocity}
  \begin{split}
  \nabla d_I(x,t)&= \nn(P_I(x,t),t),\\
  -\p_t d_I(x,t)&=\p_t \bphi_t(s)\cdot\nn(\bphi_t(s),t)=: V(s,t).
  \end{split}
  \]
  The significance of  these equations is that they   extend  the   normal vector and the normal velocity from  $I_t$ to  a neighborhood of it.  So we shall also use $\nn$ to denote $\nabla d_I$ when the latter is smooth.
We shall   extend $\nn$  to   the whole computational domain $\O$ by defining 
\[\bxi (x,t):=\phi \( \frac{d_I(x,t)}{\delta_0}\)\nabla d_I(x,t)\label{def:xi}\]
where    $\phi: \R\mapsto \R_+$ is  an     even, smooth function  that    decreases on  $[0,1]$, and satisfies 
\begin{equation}
\begin{cases}
\phi(z)>0~&\text{for}~|z|< 1,  \\
\phi(z)=0~&\text{for}~|z|\geq 1, \\
1-4 z^2\leq \phi(z)\leq 1-\frac 12 z^2~&\text{for}~|z|\leq  1/2.
\end{cases}\label{phi func control}
\end{equation}
To fulfill these requirements, we can simply choose 
$$\phi(z)
=\begin{cases}
e^{\frac 1{z^2-1}+1}~&\text{for}~|z|< 1,\\
0~&\text{for}~|z|\geq 1.
\end{cases}
$$

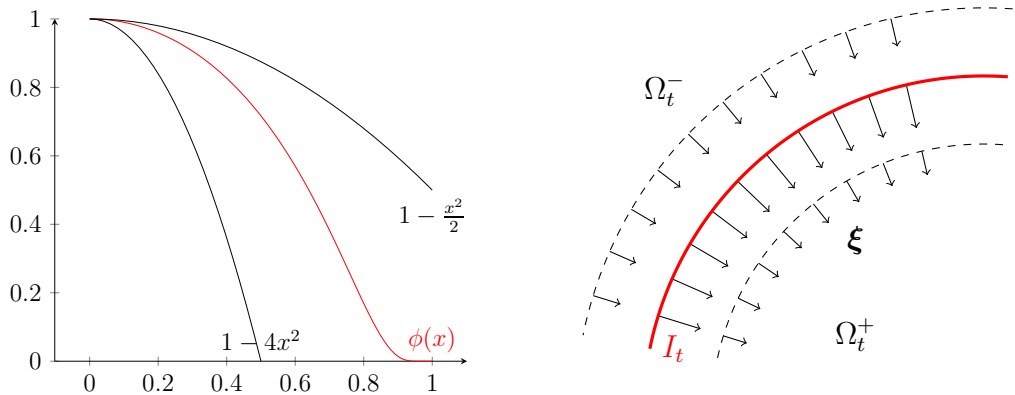
\picdis{\begin{tikzpicture}[scale = 0.8]
\begin{axis}[axis equal,axis lines = left,
]
 
\addplot[domain=0: 0.9999,color=red,samples=100]{exp(x^2/(x^2-1))} node[above] {$\phi(x)$};
\addplot[domain=0:0.5,color=black]{1-4*x^2} node[above] {$1-4x^2$};
\addplot[domain=0:1,color=black]{1-0.5*x^2} node[below] {$1-\frac{x^2}2$};
\end{axis}

\end{tikzpicture}
\qquad 
\begin{tikzpicture}[scale = 1]
\begin{axis}[axis equal,
 axis lines=none,
 xtick=\empty,
 ytick=\empty,
]

  \draw (30, 100) node   {$\O_t^-$};
    
  \addplot[samples=8, domain=1.8:pi-0.3, 
	variable=\t,
	quiver={
		u={-cos(deg(t))},
		v={-sin(deg(t))},
		scale arrows=0.09},
		->,black]
	({cos(deg(t))}, {sin(deg(t))});
	  \addplot[samples=10, domain=1.8:pi-0.3, 
	variable=\t,
	quiver={
		u={-1.5*cos(deg(t))},
		v={-1.5*sin(deg(t))},
		scale arrows=0.07},
		->,black]
	({1.5*cos(deg(t))}, {1.5*sin(deg(t))}); 
	  \addplot[samples=10, domain=1.8:pi-0.3, 
	variable=\t,
	quiver={
		u={-1.25*cos(deg(t))},
		v={-1.25*sin(deg(t))},
		scale arrows=0.13},
		->,black]
	({1.25*cos(deg(t))}, {1.25*sin(deg(t))}) ;
	\addplot[samples=100, domain=1.5:pi-0.2,dashed] 
	({1.5*cos(deg(x))}, {1.5*sin(deg(x))});
  \addplot[samples=100, domain=1.5:pi-0.2, dashed,black] 
	({cos(deg(x))}, {sin(deg(x))}) ;
	\addplot[samples=100, domain=1.5:pi-0.2, very thick,red] 
	({1.25*cos(deg(x))}, {1.25*sin(deg(x))}) node[right]{$I_t$};
	  \draw (100, 45) node   {$\bxi $};
  \draw (100, 10) node   {$\O_t^+$};
\end{axis}
\end{tikzpicture}
 }

     We   proceed with  the  extension of the mean  curvature.  Choosing a cut-off function $\eta_0(x,t)$ such that
  \[\eta_0(\cdot,t)\in C_c^\infty(B_{2\delta_0}(I_t);[0,1])~\text{ and  }\eta_0\equiv 1\text{  in }B_{\delta_0}(I_t),\label{cut-off eta delta}\] 
  we constantly extend the inward    mean  curvature vector   by  {defining}
\[  {\HH(x,t):=\kappa \nabla d_I(x,t)}  \quad\text{with}\quad \kappa(x,t)=-\Delta d_I  (P_I(x,t))\eta_0(x,t).\label{def:H}\]
These combined with \eqref{def:xi}  imply that 
\begin{subequations}
\begin{align}
& (\nn\cdot\nabla )\HH=0\text{ in }B_{\delta_0}(I_t),\label{normal H}\\
 & (\bxi \cdot\nabla )\HH=0 \text{ in }\O,\label{normal H2}\\
\label{bc n and H}
&  \bxi =0  ~\text{and}~\HH=0~\text{on}~\p\O.
\end{align}
\end{subequations}

\begin{lemma}
There exists a    constant $C>0$ depending only  on the geometry of the  interface \eqref{interface}  such  that
the  following properties hold for every $t\in [0,T]$:
 \begin{subequations}\label{xi der}
  \begin{align}
     |\nabla\cdot \bxi +\HH \cdot \bxi|    \leq C &~ | d_I|\quad \text{in}~B_{\delta_0}(I_t),\label{div xi H}\\
  \p_t d_I  +(\HH \cdot\nabla)  d_I  &=0\quad \text{in}~B_{\delta_0}(I_t),\label{mcf}\\
\p_t \bxi +\left(\HH \cdot \nabla\right) \bxi +\left(\nabla \HH\right)^{\mathsf{T}} \bxi &=0\quad \text{in}~B_{\delta_0}(I_t),\label{xi der1}  
\end{align}
  \end{subequations}
where  $\nabla \HH:=\{\p_j H_i\}_{1\leq i,j\leq 3}$ is a matrix with $i$ being the row index. 
\end{lemma}
  \begin{proof}
By introducing    $\phi_0(\tau):=\phi (\frac \tau{\delta_0})$, we can rewrite  \eqref{def:xi} as  $\bxi =\phi_0 \( d_I\)\nabla d_I$. Since $\phi$ is even,  we have   $\phi_0'(0)=0$. This   combined with Taylor's expansion in $d_I$  implies that 
\begin{align*}
\nabla\cdot \bxi &=|\nabla d_I|^2 \phi_0'(d_I)+\phi_0(d_I)\Delta d_I(x,t)
\\&=\qquad \quad O(d_I) +\phi_0 (d_I)\Delta d_I(P_I(x,t),t).
\end{align*}
This and    \eqref{def:H} lead to   \eqref{div xi H}.    
Using  \eqref{velocity} and \eqref{def:H}, we can write \eqref{csf} as the  transport equation \eqref{mcf}, which leads to  the following identities     in $B_{\delta_0}(I_t)$:
\begin{align*}
\p_t \nabla d_I+(\HH \cdot \nabla) \nabla d_I +(\nabla \HH )^{\mathsf{T}} \nabla d_I=0,\\
\p_t \phi_0(d_I )+ (\HH  \cdot\nabla) \phi_0(d_I)=0.
\end{align*}
 These two equations together  imply \eqref{xi der1}. 
  \end{proof}
It will be  convenient to introduce
\[\psi_\e=\dd^F \circ \uu_\e\quad \text{ where }\dd^F    (\vv):=\int_0^{|\vv|} g(s)\, ds.\]
It can be verified using \eqref{g lip} that 
  \[\dd^F    (\vv) \in C^1(\R^3),\quad \text{ and }\quad    D \dd^F  (\vv) =0~\text{ iff }~\vv\in \{0,\mathbb{S}^2\} .\label{df def}\]  
By    \eqref{bulk} we have 
\[|  D \dd^F (\vv)| = \sqrt{2 F  (\vv)},\qquad \forall \vv\in \R^3.\label{Dddfequ}\]
Recalling    \eqref{regularityofsolution},  we have 
 \begin{subequations}
   \begin{align}
  \label{ADM chain rule}
  \p_i\psi_\e (x,t)   &=    \p_i \uu_\e   (x,t)\cdot  D \dd^F  \(\uu_\e   (x,t)\)\quad \text{ for any }(x,t)\in \O\times (0, T],\\
  \nabla\psi_\e(x,t)&= \nabla |\uu_\e(x,t)| ~g(|\uu_\e(x,t)|)\qquad \text{ if } ~\uu_\e(x,t)\neq 0.\label{ADM chain rule1}
  \end{align}
  \end{subequations}
Now we  define    the phase-field analogues of   the normal vector and  the mean curvature  vector  respectively by 
\begin{subequations}
\begin{align}
 \nn_\e (x,t)&:=\begin{cases}
 \frac{\nabla \psi_\e }{|\nabla \psi_\e|}(x,t)&\text{ if } \nabla \psi_\e (x,t)\neq 0,\\
0& \text{otherwise}.
 \end{cases}
\label{normal diff}\\
\HH_\e (x,t)&:=\begin{cases}
-\left(\e  \Delta \uu_\e  -\frac{1}{\e } D F (\uu_\e  ) \right)\cdot\frac{\nabla \uu_\e  }{\left|\nabla \uu_\e  \right|} &\text{ if } \nabla \uu_\e\neq 0,\\
0&\text{otherwise}.
\end{cases}
 \label{mean curvature app}
\end{align}
\end{subequations}
 Note that in \eqref{mean curvature app}, the inner product is made with  the column  vectors of $\nabla \uu_\e=(\p_1 \uu_\e,\p_2 \uu_\e,\p_3 \uu_\e)$. We deduce from \eqref{normal diff} that 
 \[\nabla\psi_\e=|\nabla\psi_\e| \nn_\e\quad \text{ for any }(x,t).\label{normal diff2}\]
Define also the  orthogonal    projection $\Pi_{\uu_\e  }$ by 
\begin{align} \label{projection1}
\Pi_{\uu_\e  }  \p_i \uu_\e  :=
\begin{cases}
\(\p_i \uu_\e  \cdot\frac{\uu_\e}{| \uu_\e|  }\) \frac{\uu_\e}{| \uu_\e|  }&~\text{if}~ \uu_\e  \neq 0,\\
0,&~\text{otherwise}.
\end{cases}
\end{align}
\begin{lemma}
The following equations hold:
\begin{subequations}
\begin{align}
&\label{projectionnorm}
|\nabla \psi_\e | = |\Pi_{\uu_\e  } \nabla \uu_\e  | |  D \dd^F    (\uu_\e  )|\quad \text{ for  any   }(x,t),\\
& \label{projection}
\Pi_{\uu_\e  } \nabla \uu_\e  =\frac{|\nabla\psi_\e |} {|  D \dd^F     (\uu_\e  )|^2}  D \dd^F    (\uu_\e  )\otimes \nn_\e \quad \text{ on  } \{ x\mid |\uu_\e|\notin \{0,1\}\}.
\end{align}

\end{subequations}
\end{lemma}
\begin{proof}
Concerning \eqref{projectionnorm}, it suffices to work with  the set $\{ x\mid |\uu_\e|\notin \{0,1\}\}$ where  $g(|\uu_\e|)>0$ (cf. \eqref{bulk assump2}), for otherwise the equation will  follow from  \eqref{df def} and  \eqref{ADM chain rule}.
On this set  we deduce from \eqref{df def} that   $   D \dd^F (\uu_\e)=\frac{\uu_\e}{|\uu_\e|} g(|\uu_\e|)\neq 0$, and we can rewrite \eqref{ADM chain rule} as 
\[\p_i \psi_\e =\p_i \uu_\e\cdot\frac{D \dd^F (\uu_\e)}{|D \dd^F (\uu_\e)|} |D \dd^F (\uu_\e)|=\p_i \uu_\e\cdot\frac{\uu_\e}{| \uu_\e|  } |D \dd^F (\uu_\e)|.\label{pipsie1}\]
 This combined with \eqref{projection1}  implies \eqref{projectionnorm}.

Now we turn to the proof of \eqref{projection}. On the set $\{ x\mid |\uu_\e|\notin \{0,1\}\}$, we have

\begin{align}
\frac{|\nabla\psi_\e |} {| D \dd^F    (\uu_\e  )|^2} D \dd^F    (\uu_\e  )\otimes \nn_\e\overset{
 \eqref{normal diff2}}=\frac{ D \dd^F    (\uu_\e  )  } {| D \dd^F    (\uu_\e  )|^2}\otimes \nabla\psi_\e \overset{\eqref{ADM chain rule1}}= \frac{\uu_\e}{| \uu_\e|}\otimes \nabla |\uu_\e|,
\end{align}
and  this implies  \eqref{projection} in view of \eqref{projection1}.
\end{proof}

The following lemma establishes   coercivity properties  of  the modulated energy   \eqref{entropy}.
\begin{lemma}\label{lemma:energy bound}
The following estimates hold  for  every $t\in [0,T]$:
\begin{subequations} \label{energy bound}
\begin{align}
 \int_\O \(\frac{\e }{2} \left|\nabla \uu_\e  \right|^2+\frac{1}{\e } F  (\uu_\e  )-|\nabla \psi_\e | \) \, d x & \leq E_\e  [\uu_\e   | I] , \label{energy bound-1}\\
 \e  \int_\O \(  \mu |\div \uu_\e |^2+\left|\nabla \uu_\e  -\Pi_{\uu_\e  }\nabla \uu_\e  \right|^2   \)\, d x & \leq 2   E_\e  [\uu_\e   | I] ,\label{energy bound0}\\
  \int_\O\left(\sqrt{\e }\left|\Pi_{\uu_\e  }\nabla \uu_\e  \right|-\frac1{\sqrt{\e }} \left|  D \dd^F    (\uu_\e  )\right| \right)^{2}\, d x & \leq  2    E_\e  [ \uu_\e   | I] ,\label{energy bound2}\\
 \int_\O\( {\frac{\e }{2}}\left| \nabla \uu_\e  \right|^{2} +\frac{1}{\e } F  (\uu_\e  )+\left|\nabla \psi_\e\right|\)\left(1-\bxi  \cdot\nn_\e\right) \, d x & \leq  {4 E_\e  [ \uu_\e   | I]},\label{energy bound1}
\\
   \int_\O \(\frac{\e }2 \left|\nabla \uu_\e  \right| ^{2} +\frac{1}{\e } F  (\uu_\e  ) +|\nabla\psi_\e |\) \min\(d^2_I,1\)\, d x  & \leq C E_\e  [ \uu_\e   | I]
\label{energy bound3}
 \end{align}
\end{subequations}
where $C=C(\delta_0,\phi)$.
\end{lemma} 
\begin{proof}
The case when $\mu= 0$   has been  done in \cite{MR4284534}, and the proof  carries over  to the present case.  First, it follows from   \eqref{projection1}  that 
 \[\left|\nabla \uu_\e  -\Pi_{\uu_\e  }\nabla \uu_\e  \right|^2+\left| \Pi_{\uu_\e  }\nabla \uu_\e  \right|^2=\left|\nabla \uu_\e    \right|^2.\label{gougudingli}\] 
Combining this  with \eqref{normal diff2},   we  can write
 \begin{align}
&\frac{\e }{2} \left|\nabla \uu_\e  \right|^2+\frac{1}{\e } F  (\uu_\e  )-\bxi\cdot\nabla \psi_\e  \nonumber\\
= &  \frac \e  2\left|  \nabla \uu_\e  \right|^2     +\frac{1}{\e } F  (\uu_\e  )-|\nabla \psi_\e | +   |\nabla\psi_\e | (1-\bxi \cdot\nn_\e )\nonumber\\
= & \frac \e  2  \left|\nabla \uu_\e  -\Pi_{\uu_\e  }\nabla \uu_\e  \right|^2  +    \(\frac \e  2\left| \Pi_{\uu_\e  }\nabla \uu_\e  \right|^2     +\frac{1}{\e } F  (\uu_\e  )-|\nabla \psi_\e |\)\nonumber \\
& +   |\nabla\psi_\e | (1-\bxi \cdot\nn_\e ).\label{E decom1}
\end{align}
 By \eqref{Dddfequ} and \eqref{projectionnorm},   the second term in the last display  is non-negative. Since $|\bxi|\leq 1$, we   also have    \eqref{energy bound-1}, \eqref{energy bound0}, \eqref{energy bound2} and 
\[E_\e  [ \uu_\e   | I]\geq   \int_\O\left(1-\bxi  \cdot\nn_\e\right)\left|\nabla \psi_\e\right|\,dx.\label{diffusivecalibration}\]
 Combining \eqref{diffusivecalibration} with  \eqref{energy bound-1} and the inequality 
 $ 1-\bxi  \cdot\nn_\e\leq 2$, we obtain \eqref{energy bound1}.
Finally, by \eqref{phi func control} and $\delta_0\in (0,1)$ we have 
\[1-\bxi  \cdot\nn_\e \geq 1-\phi\(\frac {d_I}{\delta_0}\)  \geq \min \(\frac {d^2_I}{2\delta_0^2}, 1-\phi(\tfrac 1 2)\)\geq  C \min(d^2_I,1).\label{lowerbdcali}\] 
This together with  \eqref{energy bound1} implies  \eqref{energy bound3}.
\end{proof}


The following result was first proved in \cite{fischer2020convergence} for  the scalar  Allen-Cahn equation, and was generalized to the vectorial case in \cite{MR4284534}.
\begin{prop}\label{gronwallprop}
	There exists a generic  constant $C>0$ depending only  on the geometry of the  interface \eqref{interface}  {such  that}
	\begin{align}
	\frac{d}{d t} E_\e  [ \uu_\e   | I] &+\frac 1{2\e }\int_\O \(\e ^2 \left| \p_t \uu_\e    \right|^2-|\HH_\e |^2\)\,dx+\frac 1{2\e }\int_\O \Big| \e \p_t \uu_\e    -(\nabla\cdot  \bxi  )  D \dd^F    (\uu_\e  )    \Big|^2\,dx\nonumber \\
	&+\frac 1{2\e }\int_\O \Big| \HH_\e -\e   |\nabla \uu_\e  |\HH \Big|^2\,dx  \leq CE_\e  [ \uu_\e   | I]\qquad \text{ for } t\in (0,T]. \label{gronwall}
	\end{align}
\end{prop}
   We present a proof of \eqref{gronwall} in Appendix \ref{appendix} for the convenience of the readers.    

 \section{Uniform estimates of solutions}\label{sec close}
Observe that   the second term on the  left-hand side of \eqref{gronwall} does not have an obvious sign. However, we have the following theorem.
  
\begin{theorem}\label{thm close energy with div}
Under the assumptions of Theorem \ref{main thm}, there exists a  constant $C_0>0$,  which depends only  on the geometry of the  interface \eqref{interface} and $c_1$ (cf.  \eqref{initial}), such that   
 \begin{align}\label{energy bound4}
\sup_{t\in  [0,T]} \frac 1 \e E_\e  [ \uu_\e   | I]+&\int_0^T\int_\Omega \( \Big|\p_t \uu_\e  +(\HH \cdot\nabla) \uu_\e  \Big|^2 +   \Big|\p_t \uu_\e  -\Pi_{\uu_\e  }\p_t \uu_\e  \Big|^2 \)  \, dxdt  \leq C_0.
 \end{align}
  \end{theorem}
   It is worth mentioning   that $C_0$ is independent of $\mu$.
The proof of \eqref{energy bound4} relies on the following lemma.
\begin{lemma}\label{tan est of Q2 lemma}
 For any function  $\eta_1$ with  $\eta_1(\cdot,t)\in C_c(B_{4\delta_0}(I_t);\R_{\geq 0})$, there exists a universal constant $C>0$ which is independent of $t$ and $\e $  such that  
 \begin{align}
   \int_\O \eta_1\Big|   \nabla \uu_\e  \(\mathbb{I}_3-\nn\otimes \nn  \) \Big|^2\, dx\leq C  \e^{-1}
E_\e  [ \uu_\e   | I](t)\qquad \forall t\in [0,T].\label{tan est of Q2}
\end{align}
\end{lemma}
\begin{proof}
 On  the set $\{ x\mid g(|\uu_\e|)>0\}=\{ x\mid |\uu_\e|\notin \{ 0, 1\}\}$   we can use   \eqref{projection} and   \eqref{projectionnorm} to estimate 
\begin{align*}
&\Big| \Pi_{\uu_\e  } \nabla \uu_\e    ( {\mathbb{I}_3}-   \nn_\e\otimes\bxi ) \Big|^2\\
=&\,\left|\frac{|\nabla\psi_\e |} {|  D \dd^F    (\uu_\e  )|^2}  D \dd^F    (\uu_\e  )\otimes (\nn_\e -\bxi )\right|^2 \\
 \leq & \,  {| \nn_\e-\bxi |^2}\left|  \Pi_{\uu_\e  } \nabla \uu_\e  \right|^2\\
\leq  &\,  2(1-\bxi \cdot\nn_\e )\left|   \nabla \uu_\e  \right|^2.
\end{align*}
   On   the set  $\{ x\mid  |\uu_\e |=0\}$  we have  $\Pi_{\uu_\e  } \nabla \uu_\e=0$ by the second case  in   \eqref{projection1}. On the open set  $   \{  x\mid |\uu_\e|>0\}\supset \{  x\mid |\uu_\e|=1\}$  we can write $\Pi_{\uu_\e  } \nabla \uu_\e=  \nabla |\uu_\e| \otimes \frac{\uu_\e}{|\uu_\e|}$ by the first case in  \eqref{projection1}.  This combined with  
 \cite[Theorem 4.4]{MR3409135} implies   that  $\Pi_{\uu_\e  } \nabla \uu_\e=0$ for a.e. $x\in \{  x\mid |\uu_\e|=1\}$. Altogether we have shown that    
 \[\label{tangentestpointwise}
\begin{split}
 \Big| \Pi_{\uu_\e  } \nabla \uu_\e    ( {\mathbb{I}_3}-   \nn_\e\otimes\bxi ) \Big|^2
\leq   \,  2(1-\bxi \cdot\nn_\e )\left|   \nabla \uu_\e  \right|^2\quad \text{ a.e. in }\O.
\end{split}
\]
 This together with   \eqref{energy bound1} implies 
 \[\label{tan est of Q}
   \int_\O \Big| \Pi_{\uu_\e  } \nabla \uu_\e    (\mathbb{I}_3-   \nn_\e\otimes\bxi ) \Big|^2\, dx \leq C \e^{-1}
E_\e  [ \uu_\e   | I].\]
In $B_{4\delta_0}(I_t)$ where $\nn=\nabla d_I$, we have   the decomposition  
 \[\mathbb{I}_3-\nn_\e\otimes \nn =\mathbb{I}_3-  \nn_\e\otimes \bxi+  \nn_\e\otimes  ( \bxi -\nn).\]
Using \eqref{def:xi} and \eqref{phi func control}, we can estimate the last term by 
\begin{align}
&~| \bxi-\nn |^2= |\nn_\e\otimes  ( \bxi -\nn)|^2\nonumber\\
&\leq 2|  \bxi-\nn |=2\(1-\phi(\tfrac{d_I}{\delta_0})\)\leq C \min \(d^2_I ,1\).\label{diff normal d2}
\end{align}
 These inequalities  and   \eqref{energy bound3} lead to  
\[  \int_\O   \eta_1\Big| \Pi_{\uu_\e  } \nabla \uu_\e    (\mathbb{I}_3-   \nn_\e\otimes\nn  )  \Big|^2\, dx\leq C   \e^{-1}
E_\e  [ \uu_\e   | I].\label{tan est of Q1}\]
Now using  \eqref{diff normal d2}, \eqref{energy bound1}  and   \eqref{energy bound3} we find 
 $$ \int_\O   \eta_1 | \nabla \uu_\e |^2 \Big( |\nn_\e-\bxi|^2+|\bxi-\nn|^2\Big)\, dx \leq C   \e^{-1}
E_\e  [ \uu_\e   | I].$$
\blue{The above two estimates together with   the    formula  
$$(\mathbb{I}_3-\nn\otimes \nn)-(\mathbb{I}_3-\nn_\e\otimes \nn)=(\nn_\e-\bxi)\otimes \nn+(\bxi-\nn)\otimes \nn $$
yield   \eqref{tan est of Q2}.}
 \end{proof}
To proceed we need    an $L^3$-estimate  of $\uu_\e$. 

\begin{lemma}\label{L infinity bound}
  Under the assumption \eqref{bdd1},  there exists a constant $C=C(c_1)>0$  such that
\begin{subequations}
\begin{align}
&  \sup_{t\in [0,T]} A_\e(\uu_\e(\cdot,t)) +   \sup_{t\in [0,T]} \|\nabla\psi_\e(\cdot, t)\|_{L^1(\O) } \leq   C,\label{nablapsiest}\\
\label{L infinity bound1}
&  \sup_{t\in [0,T]} \|\uu_\e (\cdot, t) \|_{L^3(\Omega)}\leq ~  C.\end{align}
\end{subequations}
\end{lemma}
\begin{proof} 
It follows from   {\eqref{Dddfequ}}, \eqref{projectionnorm} and the Cauchy--Schwarz inequality   that 
$$A_\e(\uu_\e)\geq \int_\O \(\frac \e 2|\Pi_{\uu_\e} \nabla \uu_\e|^2+\frac 1{2\e} |D \dd^F(\uu_\e)|^2\)\, dx\geq \int_\O |\nabla\psi_\e|\, dx.$$
This and \eqref{dissipation} lead to   \eqref{nablapsiest}.
To prove \eqref{L infinity bound1}, we first note that  if $|\uu_\e|> 2$, then    
$$\psi_\e =\int_0^2g(z) \, dz+\int_2^{|\uu_\e|}g(z) \, dz\overset{\eqref{growth f}}\geq   c_0  (|\uu_\e|^2-4).$$
This combined with Sobolev's embedding   and $\psi_\e|_{\p\O}=0$ (cf.  \eqref{bc of omega}) leads to 
\begin{align*}
\int_\O |\uu_\e|^3\, dx &\leq C+\int_{\{x\in \O\mid  |\uu_\e|> 2\}}  |\uu_\e|^3\, dx\nonumber\\
 &\leq C\(1+ \|\psi_\e\|^{3/2}_{L^{3/2}(\O)}\)\nonumber\\
 &\leq C\(1+ \|\nabla \psi_\e\|^{3/2}_{L^1(\O)}\).
\end{align*}

 \end{proof}

\begin{proof}[Proof of Theorem \ref{thm close energy with div}]
We shall only present  the proof in 3D because   the  2D case  is  analogous under the conventions made in Subsection \ref{subsecnotation}.
We shall   employ    Einstein summation notation  by summing over repeated Latin indices.

We first use   \eqref{gronwall} to get 
\begin{align}\label{energy1}
	\frac 2 \e\frac{d}{d t}  E_\e  [ \uu_\e   | I] &+\frac 1{\e^2 }\int_\O \left[\Big(\e ^2 \left| \p_t \uu_\e    \right|^2-|\HH_\e |^2\Big)+\Big| \HH_\e -\e   |\nabla \uu_\e  |\HH \Big|^2\right]\,dx\nonumber\\
	&+\frac 1{\e^2 }\int_\O \left| \e \p_t \uu_\e    -  D \dd^F  (\uu_\e  )(\nabla\cdot  \bxi  )    \right|^2\,dx    \leq   \frac C \e E_\e  [ \uu_\e   | I].
	\end{align}
 Observe  that   the   orthogonal projection \eqref{projection1} is parallel to $D \dd^F(\uu_\e)$ when it does not vanish. So we can write 
\begin{align*}
&\left| \e \p_t \uu_\e    -  D \dd^F     (\uu_\e  )   (\nabla\cdot\bxi ) \right|^2\\=
&\left| \e \p_t \uu_\e    -\e \Pi_{\uu_\e  } \p_t \uu_\e   \right|^2+\left| \e \Pi_{\uu_\e  } \p_t \uu_\e    -  D \dd^F    (\uu_\e  )   (\nabla\cdot\bxi ) \right|^2.
\end{align*}
Substituting this identity  into  \eqref{energy1} we find 
\begin{align}\label{energy2}
 	 \frac 2 \e \frac{d}{d t} E_\e  [ \uu_\e   | I] &+\frac 1{\e^2 }\int_\O \left[\Big(\e ^2 \left| \p_t \uu_\e    \right|^2-|\HH_\e |^2\Big)+\Big| \HH_\e -\e   |\nabla \uu_\e  |\HH \Big|^2\right]\,dx\nonumber\\
	&+ \int_\O \left|\p_t \uu_\e  -\Pi_{\uu_\e  }\p_t \uu_\e  \right|^2\,dx    \leq   \frac C \e E_\e  [ \uu_\e   | I].   \end{align}
To estimate the second term on the left-hand side,  we   use  \eqref{Ginzburg-Landau} and    \eqref{mean curvature app} to write
\[ \HH_\e =-  \e \Big( \p_t \uu_\e  -\mu \nabla\div \uu_\e \Big)\cdot\frac{\nabla \uu_\e  }{|\nabla \uu_\e  |}\quad \text{ if } \nabla \uu_\e\neq 0.\] 
Note that  the inner product  is made with  the column  vectors of $\nabla \uu_\e=(\p_1 \uu_\e,\p_2 \uu_\e,\p_3 \uu_\e)$. Using the above formula, we    expand the integrands of \eqref{energy2} and find  
\begin{align*}
&~\e ^2 \left| \p_t \uu_\e    \right|^2-|\HH_\e |^2+  \Big| \HH_\e -\e   |\nabla \uu_\e  | \HH\Big|^2\\
=&~\e ^2 \left| \p_t \uu_\e    \right|^2+\e ^2 |\HH|^2 |\nabla \uu_\e  |^2+2\e ^2 \p_t \uu_\e \cdot (\HH\cdot \nabla) \uu_\e  \\
& \qquad -2\e ^2\mu ~ \nabla (\div \uu_\e)  \cdot(\HH\cdot\nabla ) \uu_\e   \\
= &~\e ^2|\p_t \uu_\e  +(\HH \cdot\nabla) \uu_\e  |^2+\e^2  \( |\HH|^2 |\nabla \uu_\e |^2- |(\HH \cdot\nabla) \uu_\e  |^2 \)  \\
&\qquad-2\e ^2\mu ~\nabla (\div \uu_\e)\cdot (\HH\cdot\nabla ) \uu_\e.
\end{align*}
Note that the second term in the last display is non-negative due to Cauchy-Schwarz's inequality, and  this  implies that  
\begin{align*}
& \int_\O  |\p_t \uu_\e  +(\HH \cdot\nabla) \uu_\e  |^2 \, dx\nonumber\\
\leq &  \frac 1{\e ^2} \int_\O  \left[\Big(\e ^2 \left| \p_t \uu_\e    \right|^2-|\HH_\e |^2\Big)+  \Big| \HH_\e -\e   |\nabla \uu_\e| \HH  \Big |^2\right]\, dx\nonumber\\
&\qquad  +2\mu \int_\O   \nabla (\div \uu_\e)\cdot (\HH\cdot\nabla ) \uu_\e\, dx. \end{align*}
Adding the above inequality to \eqref{energy2} leads to 
\begin{align}\label{lower bound relative2}
 	2\e^{-1} \frac{d}{d t} E_\e  [ \uu_\e   | I] &+\int_\O  \Big|\p_t \uu_\e  +(\HH \cdot\nabla) \uu_\e  \Big|^2 \, dx + \int_\O \Big|\p_t \uu_\e  -\Pi_{\uu_\e  }\p_t \uu_\e  \Big|^2\,dx  \nonumber \\
	 &\leq C \e^{-1} E_\e  [ \uu_\e   | I]+2\mu \int_\O   \nabla (\div \uu_\e)\cdot (\HH\cdot\nabla ) \uu_\e \, dx.   \end{align}
To estimate the last  term,  we write    $\uu_\e=(u^\e_i)_{1\leq i\leq 3}$ and $\HH=(H_i)_{1\leq i\leq 3}$.
Using integration by parts and   \eqref{bc n and H}, we obtain 
 \[\label{control is div2}\begin{split}
& \int_\O \nabla (\div \uu_\e)\cdot (\HH\cdot\nabla ) \uu_\e\, dx \\
= & - \int_\O    (\div \uu_\e) (\HH\cdot\nabla ) \div \uu_\e  \, dx- \int_\O    (\div \uu_\e)(\p_j \HH\cdot\nabla ) u^\e_j\, dx\\
=&\frac 12 \int_\O    (\div \HH )  (\div \uu_\e)^2\, dx-\int_\O (\div \uu_\e)  \p_k H_j \p_k u^\e_j\, dx \\
&\qquad -\int_\O (\div \uu_\e) (\p_j H_k-\p_k H_j) \p_k u^\e_j \, dx.
\end{split}\]
In view of \eqref{energy bound0}, the first integral in the last  display  of  \eqref{control is div2}  is bounded by $$ \mu^{-1} \e ^{-1} \|\div \HH\|_{L^\infty_{t,x}}E_\e  [ \uu_\e   | I].$$
The second  integral  can be estimated  by decomposing  $\nabla u^\e_j$ and   by using \eqref{normal H}:
\begin{align} \label{tan est1}
&-\int_\O (\div \uu_\e)  \,\nabla H_j\cdot \nabla u^\e_j\,dx\nonumber\\
= & -\int_\O (\div \uu_\e)  \,\nabla H_j\cdot \Big((\mathbb{I}_3-\nn\otimes \nn )  \nabla u^\e_j\Big)\,dx-\int_\O (\div \uu_\e)  \, (\nn \cdot \nabla   H_j)  \( \nn  \cdot \nabla u^\e_j\)\,dx\nonumber\\
 \leq  &  \int_\O |\div \uu_\e|^2\,dx +   \int_\O |\nabla\HH|^2 \Big|(\mathbb{I}_3-\nn\otimes \nn )   {\nabla \uu_\e}\Big|^2\,dx\nonumber\\
&\quad +C \int_\O |\nabla \uu_\e|^2 \min\(d^2_I ,1\)\,dx.
\end{align}
By \eqref{def:H}  and \eqref{cut-off eta delta}, the second integral in the last display can be estimated using   \eqref{tan est of Q2}  with $\eta_1:=|\nabla\HH|^2$. The other two terms can be controlled by $ {(\mu^{-1} +1)}C\e ^{-1} E_\e  [ \uu_\e   | I]$ using \eqref{energy bound0} and \eqref{energy bound3} respectively. To summarize we deduce from    \eqref{control is div2} and \eqref{tan est1} that 
\[
\begin{split}
& \int_\O \nabla (\div \uu_\e)\cdot (\HH\cdot\nabla ) \uu_\e\, dx \\
\leq &~\mu^{-1} \e ^{-1} \|\div \HH\|_{L^\infty_{t,x}}E_\e  [ \uu_\e   | I]+ (\mu^{-1} +1)C\e ^{-1} E_\e  [ \uu_\e   | I]  \nonumber\\
&   -\int_\O (\div \uu_\e) (\p_j H_k-\p_k H_j) \p_k u^\e_j \, dx.
\end{split}\]
Combining this with \eqref{lower bound relative2}, we find 
\begin{align}\label{lower bound relative3}
 	2\e^{-1} \frac{d}{d t} E_\e  [ \uu_\e   | I] &+\int_\O  \Big|\p_t \uu_\e  +(\HH \cdot\nabla) \uu_\e  \Big|^2 \, dx + \int_\O \Big|\p_t \uu_\e  -\Pi_{\uu_\e  }\p_t \uu_\e  \Big|^2\,dx  \nonumber \\&\leq C \e^{-1} E_\e  [ \uu_\e   | I]-2\mu\int_\O (\div \uu_\e) (\p_j H_k-\p_k H_j) \p_k u^\e_j\,dx.\end{align}
	Note that due to \eqref{conventionmu} the constant $C$ above can be made  independent of $\mu$.
It remains to estimate  the last integral in \eqref{lower bound relative3}.
By    orthogonal  decompositions\footnote{For a  square matrix $A$, the decomposition $A=\frac{A+A^{\mathsf{T}}}2+\frac{A-A^{\mathsf{T}}}2$  is orthogonal under  the Frobenius inner product $A:B\triangleq \tr (A^{\mathsf{T}} B)$.},
$$  (\p_j H_k-\p_k H_j) \p_k u^\e_j=-( \rot \uu_\e ) \cdot(\rot \HH).$$  
We also need the following identity which follows by taking the wedge product of   \eqref{Ginzburg-Landau} with $\uu_\e$.
$$\mu (\nabla \div \uu_\e)  \wedge    \uu_\e= (\p_t \uu_\e- \Delta \uu_\e)  \wedge     \uu_\e.$$
Using the above two identities, we   integrate by parts  to obtain  
\begin{align*}
&-\mu\int_\O (\div \uu_\e) (\p_j H_k-\p_k H_j) \p_k u^\e_j\,dx\\
=&~\mu\int_\O (\div \uu_\e) ( \rot \uu_\e ) \cdot (\rot \HH)\,dx\\
 =&~\mu\int_\O  (\div \uu_\e)     \uu_\e \cdot (\rot  \rot \HH)\,dx-\int_\O \mu (\nabla \div \uu_\e)  \wedge    \uu_\e  \cdot (\rot \HH)\,dx\\
=&~\mu\int_\O  (\div \uu_\e)     \uu_\e  \cdot (\rot  \rot \HH)\,dx -\int_\O (\p_t \uu_\e- \Delta \uu_\e)  \wedge     \uu_\e   \cdot (\rot \HH )\,dx\\
=& ~\mu\int_\O  (\div \uu_\e)     \uu_\e \cdot  (\rot  \rot \HH)\,dx-\int_\O  \Big(\p_t \uu_\e +(\HH \cdot\nabla) \uu_\e\Big) \wedge   \uu_\e  \cdot ( \rot \HH)\,dx\\
&+\int_\O   (\HH \cdot\nabla) \uu_\e   \wedge    \uu_\e \cdot( \rot \HH)\,dx+\int_\O   \Delta \uu_\e   \wedge     \uu_\e  \cdot (\rot \HH)\,dx.
\end{align*}
 {Inserting} this identity  into \eqref{lower bound relative3}, and using  the Cauchy--Schwarz  inequality, \eqref{L infinity bound1} and   \eqref{energy bound0}, we find  
\begin{align}\label{lower bound relative4}
 	&2\e^{-1} \frac{d}{d t} E_\e  [ \uu_\e   | I] +\frac 12 \int_\O  \Big|\p_t \uu_\e  +(\HH \cdot\nabla) \uu_\e  \Big|^2 \, dx + \int_\O \Big|\p_t \uu_\e  -\Pi_{\uu_\e  }\p_t \uu_\e  \Big|^2\,dx  \nonumber \\
	 & \leq C\Big(1+ \e^{-1} E_\e  [ \uu_\e   | I]\Big)+2\int_\O   (\HH \cdot\nabla) \uu_\e   \wedge   \uu_\e \cdot  (  \rot \HH)\,dx+2\int_\O   \Delta \uu_\e   \wedge     \uu_\e \cdot  (\rot \HH) \,dx\nonumber\\
	  & =  C\Big(1+ \e^{-1} E_\e  [ \uu_\e   | I]\Big)+2\int_\O    H_k    \Big(\p_k \uu_\e-\Pi_{\uu_\e}\p_k  \uu_\e\Big)   \wedge     \uu_\e \cdot (  \rot \HH)\,dx\nonumber\\
	  &\qquad\qquad   -2\int_\O  \Big(\p_k  \uu_\e-\Pi_{\uu_\e} \p_k \uu_\e\Big)   \wedge     \uu_\e \cdot \Big(\p_k \rot \HH\Big)\,dx.
	 \end{align}
	Note that in the last step we    used  integration by parts, the identity 
	\[(\Pi_{\uu_\e}\p_k  \uu_\e)\wedge  \uu_\e  =0\label{theidentity} \]
	  which follows from \eqref{projection1}, and the identities $( \p_k \uu_\e )   \wedge (   \p_k \uu_\e)= 0$ for each fixed $k\in \{1,2,3\}$.    Finally, applying  the Cauchy--Schwarz inequality  and then \eqref{energy bound0} and  \eqref{L infinity bound1} in   the last two integrals of \eqref{lower bound relative4}, we find 
	  \begin{align}
 	2\e^{-1} &\frac{d}{d t} E_\e  [ \uu_\e   | I] +\frac 12 \int_\O  \Big|\p_t \uu_\e  +(\HH \cdot\nabla) \uu_\e  \Big|^2 \, dx +  \int_\O \Big|\p_t \uu_\e  -\Pi_{\uu_\e  }\p_t \uu_\e  \Big |^2\,dx  \nonumber \\
	  &\leq C\Big(  1+ \e^{-1} E_\e  [ \uu_\e   | I]\Big).  	 \end{align}
This combined with \eqref{initial} and    Gr\"{o}nwall's inequality leads to   \eqref{energy bound4}.   \end{proof}
 Using \eqref{energy bound3} and \eqref{energy bound4}, we readily   obtain the following  {corollary.}
 \begin{coro}\label{coro space-time der bound}
 Under the assumptions of Theorem \ref{main thm}, there exists a  constant $C>0$,  which depends only  on the geometry of the  interface \eqref{interface} and $c_1$, such that   
 \begin{subequations}
 \begin{align}
\sup_{t\in  [0,T]}\int_{\O^\pm_t\backslash B_{\delta}(I_t)}\(|\nabla \uu_\e  |^2+\frac 1{\e^2} 
F(\uu_\e  )+\frac 1\e |\nabla\psi_\e| \)\, dx  & \leq C\delta^{-2} ,\label{space der bound local}\\
\int_0^T\int_{\O^\pm_t\backslash B_{\delta}(I_t)} |\p_t \uu_\e  |^2\, dx dt &\leq C\delta^{-2},\label{time der bound local}
\end{align}
 \end{subequations}
  hold  for each fixed $\delta\in (0, \delta_0)$.
 \end{coro}
 Indeed,  \eqref{time der bound local} follows from  \eqref{space der bound local} and the inequality
  \[\int_0^T\int_\Omega   \Big|\p_t \uu_\e  +(\HH \cdot\nabla) \uu_\e  \Big|^2\, dxdt\leq C,\label{weaktransport}\]
 which is a consequence of  \eqref{energy bound4}.
Another consequence of \eqref{energy bound4} is the following lemma concerning  
\begin{align}\label{def uhat}
\widehat{\uu}_\e:=\begin{cases}
\frac{\uu_\e}{|\uu_\e|}&\text{ if } \uu_\e\neq 0,\\
0& \text{otherwise}.
\end{cases}
\end{align}

 \begin{lemma}
 Under the assumptions of Theorem \ref{main thm}, there exists a  constant $C>0$, which depends only  on the geometry of the  interface \eqref{interface} and $c_1$, such that 
\begin{subequations}
\begin{align}
&\sup_{t\in [0,T]}\int_\O |\uu_\e|^2 \left| \nabla \widehat{\uu}_\e\right|^2\, dx+ \sup_{t\in [0,T]}\int_\O \Big|  \widehat{\uu}_\e\cdot \nabla |\uu_\e| \Big|^2\, dx\leq  (1+\mu^{-1})C   \label{bound degree+orien},\\
&\sup_{t\in [0,T]}\int_\O \(\widehat{\uu}_\e  \cdot \nn_\e\)^2 |\nabla \psi_\e|\, dx\leq  (1+\mu^{-1})(1+\sqrt{\mu+1}) C\e.  \label{cos law}
\end{align}
\end{subequations}

\end{lemma}
\begin{proof}
We first  deduce   from \eqref{energy bound4} and \eqref{energy bound0} that 
 \[\sup_{t\in [0,T]} \int_\O \Big(  \mu |\div \uu_\e |^2+\left|\nabla \uu_\e  -\Pi_{\uu_\e  }\nabla \uu_\e  \right|^2   \Big)\, d x\leq C.\label{energy bound7}\]
By \eqref{def uhat}  we have the identity    $\uu_\e=|\uu_\e|\widehat{\uu}_\e$. Using this and   \eqref{projection1}, we can write
\[\nabla \uu_\e  -\Pi_{\uu_\e  }\nabla \uu_\e=|\uu_\e| \nabla\widehat{\uu}_\e~  \text{ if }\uu_\e\neq 0.\] 
Substituting this formula into \eqref{energy bound7}, we obtain  the estimate of the first   integral on the left-hand side of \eqref{bound degree+orien}. To control  the second one,   we use the following formula which follows from  \eqref{projection1}:
  \[\tr \nabla \uu_\e-\tr \(\Pi_{\uu_\e  }\nabla \uu_\e\)=\div \uu_\e-\widehat{\uu}_\e\cdot \nabla |\uu_\e|~  \text{ if }\uu_\e\neq 0.\] 
   Note that on the set  $\{ x\mid  |\uu_\e|=0\}$, we have  $\nabla |\uu_\e|=0$ a.e., and thus the above formula is  still valid.
This   and   \eqref{energy bound7} yield the estimate of $\widehat{\uu}_\e\cdot \nabla |\uu_\e|$ and     \eqref{bound degree+orien} is proved.

Regarding  \eqref{cos law},  it suffices to estimate over the set  $$\{ x\mid \nabla\psi_\e\neq 0\}=:U_\e$$  because  the integral over its complement vanishes. By \eqref{df def} and  \eqref{ADM chain rule}, we have 
$U_\e\subset \{ x\mid |\uu_\e|\notin  \{0,1\}\}$ where  $g(|\uu_\e|)=|D \dd^F|(\uu_\e)>0$. This combined with   \eqref{ADM chain rule1} and   \eqref{normal diff} implies that   
$$\nn_\e=\frac{\nabla\psi_\e}{|\nabla\psi_\e|}=\frac{\nabla |\uu_\e|}{|\nabla |\uu_\e||}\quad \text{ on } ~U_\e.$$
On the other hand, by the polar decomposition $\uu_\e=|\uu_\e|\widehat{\uu}_\e$ and orthogonality
    $\hat{\uu}_\e\perp \p_{x_j} \hat{\uu}_\e$, we have     
\[|\nabla \uu_\e|^2=|\nabla |\uu_\e| |^2+|\uu_\e|^2|\nabla\hat{\uu}_\e|^2\geq |\nabla |\uu_\e| |^2\quad \text{ on } ~U_\e.\label{refereestar1}\]
Setting   $\widehat{\uu}_\e  \cdot \nn_\e=:\cos\theta_\e$,    we have 
  \[\mu \int_{U_\e} \cos^2 \theta_\e \big|\nabla |\uu_\e|\big|^2\, dx=\mu \int_{U_\e} \left|  \widehat{\uu}_\e \cdot \nn_\e \right|^2 \big|\nabla |\uu_\e|\big|^2\, dx\overset{\eqref{bound degree+orien}}\leq  (1+\mu) C. \]
 This inequality,  \eqref{energy bound-1} and \eqref{refereestar1} together imply  that 
\begin{align*}
 {(1+\mu)}C  &\geq  \int_{U_\e} \frac \mu 2 \cos^2 \theta_\e \big|\nabla |\uu_\e|\big|^2\, dx+\int_{U_\e} \(\frac{1}{2} \Big|\nabla |\uu_\e|  \Big|^2+\frac{1}{\e^2 } F  (\uu_\e  )-\frac 1{\e}|\nabla \psi_\e | \) \, d x \\
&\geq \frac 1{\e} \int_{U_\e}  \(\sqrt{ \mu\cos^2 \theta_\e+1}   \,\,\Big|\nabla |\uu_\e|\Big|\sqrt{2F  (\uu_\e  )}- |\nabla \psi_\e | \)  \, d x \\
&= \frac 1{\e} \int_{U_\e}  \(\sqrt{\mu \cos^2 \theta_\e+1}  -1\)  |\nabla \psi_\e |   \, d x.
\end{align*}
Note that in the last step we have used the identity   $\big|\nabla |\uu_\e|\big|\sqrt{2F  (\uu_\e  )}=|\nabla\psi_\e|$, which holds on $U_\e$. 
So   \eqref{cos law} follows from  conjugation.
\end{proof}

\section{Estimates of level sets}\label{sec level}
Recalling \eqref{normalization}, the main result of this section is the following $L^1$-estimate of $\psi_\e$.
\begin{theorem}\label{sharp l1}
Under the assumptions of Theorem \ref{main thm},
there exists $C>0$ independent of  $\e$ such  that 
\begin{align}
&\sup_{t\in [0,T]}B  [\uu_\e  | I](t)  \leq C\e,\label{sharpnew1}\\
&\sup_{t\in [0,T]}\int_{\O}| \psi_\e-\1_{\O_t^+}| \, dx \leq C\e^{1/4}.\label{volume convergence}
\end{align}
\end{theorem}
 
\begin{proof}
We shall   denote  the positive and  negative parts   of a function $h$ by $h^+$ and $h^-$ respectively.  For simplicity we shall suppress $\, dx$ in a volume integral.
 By \cite[pp. 153]{MR3409135},  for any $h\in W^{1,1}(\O)$, we have 
\[\p_i (h(x))^+= (\p_i h(x)) \1_{\{x\mid h(x)>0\}}(x)\quad \text{for } a.e.~~x\in \O.\label{der positive part}\]

Our goal is to estimate $2\psi_\e-1-\chi$ where $\chi(x,t)=\pm 1$ in $\O_t^\pm$.
Using the formula $h=h^+-h^-$, we can write 
 \[2\psi_\e-1=2(\psi_\e-1)^++ \(1 -2(\psi_\e-1)^-\), \label{psi deco}\]
 and we shall estimate its difference with $\chi$. This will be done by 
 establishing   differential inequalities   for  the following  energies which add  up to \eqref{gronwall2new}:
\begin{subequations} 
\begin{align}
g_\e(t):=&\int_\O  ( \psi_\e-1)^+\zeta\circ d_I ,\label{gronwall1}\\
h_\e(t):=&\int_\O   \Big(\chi-[1- 2(\psi_\e-1)^-] \Big)\eta\circ d_I ,\label{gronwall2}
\end{align}
\end{subequations}
where $\eta(z)$ is defined by \eqref{truncation eta} 
and   $ |\eta|(z)=:\zeta(z)$. It is obvious that  the integrand of   \eqref{gronwall1}    is  non-negative.
Since  $\psi_\e\geq 0$, we have $(\psi_\e-1)^-\in [0,1]$ and thus $[1 -2(\psi_\e-1)^-]$ ranges in $[-1,1]$.  
 Using   the identity $(\eta \circ d_I)~\chi=|\eta\circ d_I|$, we deduce that   the integrand of \eqref{gronwall2} is also  non-negative and 
 \[h_\e(t)=\int_\O \Big|1 -2(\psi_\e-1)^--\chi\Big| ~\zeta\circ d_I .\label{gronwall4}\]
 Finally, we deduce from  \eqref{initial} that 
    \begin{align}\label{gronwallinitial1}
      g_\e(0)+h_\e(0)\leq c_1\e.
    \end{align}

{\it Step 1: estimates of weighted errors.}
Using \eqref{psi} and \eqref{bulk}, we have
\begin{align}\label{volume evo1}
\p_t \psi_\e
=&\Big(\p_t \uu_\e+(\HH\cdot\nabla)\uu_\e\Big)\cdot\frac{\uu_\e}{|\uu_\e|} \sqrt{2F(\uu_\e)}-\HH\cdot\nabla\psi_\e.
\end{align}
Using this and \eqref{der positive part}  we can calculate 
\begin{align*}
g_\e'(t)
 =&\int_{\{ \psi_\e > 1\}} (\p_t \uu_\e+(\HH\cdot\nabla)\uu_\e)\cdot\frac{\uu_\e}{|\uu_\e|} \sqrt{2F(\uu_\e)}~\zeta\circ d_I \\
&-\int_{\{ \psi_\e > 1\}} \HH\cdot \nabla\psi_\e ~\zeta\circ d_I +\int_\O  ( \psi_\e-1)^+ \p_t(\zeta\circ d_I) \\
=&\int_{\{ \psi_\e > 1\}} (\p_t \uu_\e+(\HH\cdot\nabla)\uu_\e)\cdot\frac{\uu_\e}{|\uu_\e|} \sqrt{2F(\uu_\e)}~\zeta\circ d_I \\
&-\int_\O  \HH\cdot \nabla ( \psi_\e-1)^+ ~\zeta\circ d_I -\int_\O   ( \psi_\e-1)^+ \HH\cdot\nabla (\zeta\circ d_I) \\
&+\int_\O  \Big(\p_t  (\zeta\circ d_I)+\HH\cdot\nabla (\zeta\circ d_I)\Big)  ( \psi_\e-1)^+.\end{align*}
 By  \eqref{mcf}, the integrand of the last integral vanishes on $B_{\delta_0}(I_t)$. Moreover, we can      combine  the second and the third integrals in the last display  using  integration  by parts. Using also that $\|\div \HH \|_{L^\infty_{x,t}}\leq C$ and \eqref{energy bound3}, we find  
 \begin{align}\label{referee1}
g_\e'(t)
\leq &\int_{\{ \psi_\e > 1\}} (\p_t \uu_\e+(\HH\cdot\nabla)\uu_\e)\cdot\frac{\uu_\e}{|\uu_\e|} \sqrt{2F(\uu_\e)}~\zeta\circ d_I\nonumber \\
&+\int_\O  (\div \HH) ( \psi_\e-1)^+ ~\zeta\circ d_I  + C\int_{\O\backslash B_{\delta_0}(I_t)}  ( \psi_\e-1)^+\nonumber\\
 \leq & \int_\O \e \Big|\p_t \uu_\e+(\HH\cdot\nabla)\uu_\e\Big|^2+ \left(\int_\O   \frac 1{\e}{F(\uu_\e)}\zeta^2\circ d_I \right) +Cg_\e \nonumber \\
 \leq & C E_\e[\uu_\e |I]  +Cg_\e+ \int_\O \e \Big|\p_t \uu_\e+(\HH\cdot\nabla)\uu_\e\Big|^2.
\end{align}
Now using   \eqref{gronwallinitial1},  \eqref{weaktransport} and \eqref{energy bound4}, we can apply the Gr\"{o}nwall lemma and obtain $\sup_{t\in [0,T]}g_\e(t)\leq C \e$ for some $C$ which is independent of $\e$.  
Concerning   $h_\e$, for simplicity we introduce  $w_\e:=\chi-[1- 2(\psi_\e-1)^-]$. Using the identity $(\p_i \chi )~\eta \circ d_I \equiv  0$ (in the sense of  distribution), we find 
  \[(\p_i w_\e) ~\eta\circ d_I=(2\p_i \psi_\e)~ \1_{\{\psi_\e< 1\}}~\eta\circ d_I.\] So by the same  calculation for $g_\e$  we obtain
\begin{align*}
h_\e'(t)
=&\int_{\{ \psi_\e < 1\}} 2(\p_t \uu_\e+(\HH\cdot\nabla)\uu_\e)\cdot\frac{\uu_\e}{|\uu_\e|} \sqrt{2F(\uu_\e)}~\eta\circ d_I \\
&+\int_\O  (\div \HH)w_\e~\eta\circ d_I  +\int_\O  \Big(\p_t (\eta\circ d_I)+(\HH\cdot\nabla) \eta\circ d_I\Big)  w_\e\\
\leq & C E_\e[\uu_\e |I]  +Ch_\e(t)+\int_\O \e \Big|\p_t \uu_\e+(\HH\cdot\nabla)\uu_\e\Big|^2.\end{align*}
Using \eqref{gronwallinitial1} and   {\eqref{weaktransport}}, we can apply the Gr\"{o}nwall lemma and obtain  $\sup_{t\in [0,T]}h_\e(t)\leq C\e$. Finally, by \eqref{psi deco} and \eqref{gronwall4}, we find 
\begin{align}
&\int_\O |2\psi_\e-1-\chi|\zeta\circ d_I \nonumber\\
&  \leq   \int_\O 2(\psi_\e-1)^+\zeta\circ d_I  + \int_\O \Big|1 -2(\psi_\e-1)^--\chi\Big| \zeta\circ d_I \nonumber\\
 & =  2g_\e(t)+h_\e(t)\leq C\e\quad \text{ for all }t\in [0,T],\label{gronwall3}
\end{align}
 and this proves \eqref{sharpnew1}.
 
 {\it Step 2: remove the weight.}
 First note that \eqref{gronwall3} implies   $\eqref{volume convergence}$ with $\O$   replaced by  $\O\backslash B_{\delta_0}(I_t)$. So we shall focus on the estimate on  $B_{\delta_0}(I_t)$.
We    set $\chi_\e:=2\psi_\e-1$ and abbreviate $\delta_0$ by $\delta$. 
For fixed $t\in [0,T]$ and $p\in I_t$ with normal vector $\nn=\nn(p)$, applying H\"{o}lder's inequality and Lemma \ref{Lemma cubic est}  below  with $f(r,p,t)=\left|\chi\left(p + r \nn, t\right)-\chi_\e (p+ r \nn, t )\right|$, we find
\begin{align*}
& \(\int_{B_{\delta}(I_t)}\left|\chi(x,t)-\chi_\e(x, t)\right| \, dx\)^{4/3}  \\
=&    \(  \int_{I_t } \int_{-\delta}^{\delta}f(r,p,t) \,  {dr}  \,  d \mathcal{H}^{d-1}(p)\)^{4/3}  \\
 \leq  &  C   \int_{I_t } \( \int_{-\delta}^{\delta}f(r,p,t) \,  {dr}   \)^{4/3} d \mathcal{H}^{d-1}(p)  \\
\overset{\eqref{cubic est}}\leq & C \int_{I_t }\|f(\cdot,p,t)\|_{L^{3/2}(-\delta,\delta)}\(\int_{-\delta}^{\delta}f(r,p,t) |r| \, dr\)^{1/3} \, d\mathcal{H}^{d-1}(p) \\
 =~& C\|f(\cdot,t)\|_{L^{3/2}(B_{\delta}(I_t))}\(\int_{I_t }\int_{-\delta}^{\delta}f(r,p,t) |r| \, dr \, d\mathcal{H}^{d-1}(p))\)^{1/3}. 
\end{align*}
In view of \eqref{psi} and  \eqref{bc of omega},  we have \blue{$\psi_\e=0$ on $\p\O$}. So by   Sobolev's embedding $W^{1,1}\hookrightarrow L^{3/2}$ we obtain
\begin{align*}
& \(\int_{B_\delta(I_t)}\left|\chi(x,t)-\chi_\e(x, t)\right| \, dx\)^4  \\
 \leq  \,& \,C \( \|\chi\|^3_{L^{3/2}(\O)}+\|\chi_\e\|^3_{L^{3/2}(\O)}\) \int_{\O}\zeta\circ d_I~ |\chi_\e-\chi |   \, dx  \\
  \leq  \, & \,C (1+ \|\nabla \psi_\e\|^3_{L^1(\O)}) \int_{\O} \zeta\circ d_I ~|\chi_\e-\chi |   \, dx \leq C \e. 
\end{align*}
 Note that in the last step we employed \eqref{nablapsiest} and \eqref{gronwall3}.
This    gives the desired  estimate   in $B_{\delta_0}(I_t)$ and thus the proof of \eqref{volume convergence} is finished.  \end{proof}
\begin{lemma}\label{Lemma cubic est}
For any integrable function $f:[-\delta,\delta]\to \R_{\geq 0}$, we have 
\[\(\int_{-\delta}^\delta f(r)\, dr \)^4 \leq  {6}\|f\|^{3}_{L^{3/2}(-\delta,\delta)} \int_{-\delta}^\delta |r|   f(r)\, dr.\label{cubic est}
\]
\end{lemma}
\begin{proof} 
 We write  $x=(x_1,x_2,x_3),y=(y_1,y_2,y_3)$ and $F(x)=f(x_1)f(x_2)f(x_3)$.
 By symmetry and the H\"{o}lder inequality, we find 
\begin{align*}
\| f\|^6_{L^1(0,\delta)} 
&=
\int_{[0,\delta]^6} F(x)F(y)\, dxdy\\
&=2\int_{[0,\delta]^6\cap \{(x,y)~\mid x_1+x_2+x_3\leq y_1+y_2+y_3\}}  F(x)F(y)\, dxdy\\
&=2 \int_{[0,\delta]^3} \(\int_{[0,\delta]^3\cap \{x~\mid x_1+x_2+x_3\leq y_1+y_2+y_3\}} 1\cdot  F(x)\, dx\)F(y)\,dy\\
&\leq 2 \int_{[0,\delta]^3} (y_1+y_2+y_3)   \(\int_{[0,\delta]^3} F^{3/2}(x)\, dx \)^{2/3}F(y)\,dy\\
&=6 \| f\|^3_{L^{3/2}(0,\delta)}\| f\|^2_{L^1(0,\delta)} \int_0^\delta r   f(r)\, dr.
\end{align*}

\end{proof}
Now we turn to the  study of the    level sets of $\psi_\e$. The main tool is the following estimate, \blue{which is a    consequence} of \eqref{normal diff}, \eqref{energy bound1} and   \eqref{energy bound4}. 
\begin{align}\label{localcalibrationest1}
&\sup_{t\in [0,T]}\int_U \Big(|\nabla\psi_\e|- \bxi \cdot \nabla\psi_\e\Big)\, dx  \nonumber\\
=& \sup_{t\in [0,T]}   \int_U \Big(|\nabla\psi_\e|- \bxi \cdot \nn_\e |\nabla\psi_\e|\Big)\, dx \leq C\e, \quad \forall U \text{ measurable in }\O.
\end{align}

\begin{lemma}\label{sharp est of bdy}
For each  $t\in [0,T]$ there exists a null set $\mathcal{N}_t^\e\subset (0,1/8)$ such that the following holds: 
for every   $\alpha  \in (0,1/8)\backslash \mathcal{N}_t^\e$, there exist \[ b_{\e,\alpha}(t)\in [1/2-\alpha  ,1/2+\alpha  ]\quad\text{ and }\quad q_{\e,\alpha}(t)\in [  2-\alpha  ,  2+\alpha  ]\]  such that the sets \[\{x\mid \psi_\e(x,t) >b_{\e,\alpha}(t)\}\text{ and } \{x\mid \psi_\e(x,t) <q_{\e,\alpha}(t) \}\] are of  finite perimeter and 
\begin{subequations}
\begin{align}\label{areacompare1}
\Big|\mathcal{H}^{d-1}(\{x\mid \psi_\e(x,t) =b_{\e,\alpha}(t)\})-\mathcal{H}^{d-1} (I_t)\Big|\leq  C  \e^{1/4}\alpha  ^{-1}, \\
\label{areacompare2}
  \mathcal{H}^{d-1}(\{x\mid  \psi_\e(x,t) =q_{\e,\alpha}(t)\}) \leq  C  \e^{1/4}\alpha  ^{-1},
\end{align}
\end{subequations}
 where $C>0$ is independent of $t, \e$ and $\alpha  $.
\end{lemma}

\begin{proof}
To prove \eqref{areacompare1}, we consider the set 
\[S_t^{\e,\alpha  } =\{x\in \O\mid   |2\psi_\e (x,t)- 1|\leq  2\alpha   \}, \qquad \forall \alpha   \in (0,1/8).\label{local set}\]
It follows from  the co-area formula of BV function \cite[section 5.5]{MR3409135}   that $ S_t^{\e,\alpha  }$ has finite perimeter for   every $\alpha\in  (0,1/8)\backslash \tilde{\mathcal{N}}_t^\e $ for some null set $\tilde{\mathcal{N}}_t^\e\subset (0,1/8)$.  Moreover, by \eqref{localcalibrationest1}, we have for   every $\alpha     \in  (0,1/8)\backslash \tilde{\mathcal{N}}_t^\e$ that 
\begin{align}\label{localcalibration1}
C\e &\geq    \int_{S_t^{\e,\alpha  }} \Big(|\nabla\psi_\e|- \bxi \cdot \nabla\psi_\e\Big)\, dx \nonumber\\
&=\int_{\frac 12 -\alpha  }^{\frac 12 +\alpha  }   \mathcal{H}^{d-1}\(\{x\mid\psi_\e  =s\}\)\, ds-\int_{\p S_t^{\e,\alpha  }} \bxi \cdot \bnu \psi_\e  \, d\mathcal{H}^{d-1}+\int_{S_t^{\e,\alpha  }} (\div \bxi ) \psi_\e  \, dx,
\end{align}  
where $\bnu$ is the outward   normal of the set $  S_t^{\e,\alpha  }$, defined on its (measure-theoretic) boundary.
Since $|\bxi|\leq 1$ on $\O$  and  $\psi_\e\leq 1$ on $S_t^{\e,\alpha  }$, we have
$$\left|\int_{S_t^{\e,\alpha  }} (\div \bxi ) \psi_\e  \, dx\right| \leq C |S_t^{\e,\alpha  }|,$$ 
where $|A|=\mathcal{L}^d (A)$ is the $d$-Lebesgue measure of a set $A$. Combining this with \eqref{localcalibration1}, we find  
\[\left|\int_{\frac 12 -\alpha  }^{\frac 12 +\alpha  }    \mathcal{H}^{d-1}\(\{x\mid\psi_\e  =s\}\)\, ds-  \int_{\p S_t^{\e,\alpha  }} \bxi \cdot \bnu  \psi_\e  \, d\mathcal{H}^{d-1}\right|\leq C( \e  +  |S_t^{\e,\alpha  }|).\label{volume est1}\]
By the divergence theorem,   we have 
\begin{align*}
\int_{\p S_t^{\e,\alpha  }} \bxi \cdot \bnu  \psi_\e  \, d\mathcal{H}^{d-1}& = -\(\tfrac 12-\alpha   \) \int_{\{x\mid \psi_\e< \frac 12-\alpha  \}} (\div \bxi)     \, dx- \(\tfrac 12 +\alpha   \)\int_{\{x\mid \psi_\e> \frac 12+\alpha  \}} (\div \bxi)     \, dx,\\
-2\alpha   \mathcal{H}^{d-1} (I_t)&\overset{\eqref{def:xi}}=\quad \(\tfrac 12-\alpha   \)\int_{\O_t^-}(\div \bxi) \, dx\qquad\quad  +\(\tfrac 12+\alpha   \) \int_{\O_t^+}(\div \bxi) \, dx.
\end{align*}
Inserting these  two equations into  \eqref{volume est1}, we find
\begin{align}\label{volume est3}
&\left|  \int_{\frac 12 -\alpha  }^{\frac 12 +\alpha  }   \mathcal{H}^{d-1}\(\{x\mid\psi_\e  =s\}\)\, ds-2\alpha    \mathcal{H}^{d-1} (I_t)\right|\nonumber\\
&\leq C\left( \e  +  |S_t^{\e,\alpha  }| +  \Big| \O_t^- \triangle     {\{x\mid \psi_\e <\tfrac 1 2-\alpha  \}} \Big|+  \Big| \O_t^+\triangle {\{x\mid \psi_\e>\tfrac1 2+\alpha  \}}\Big|\right),
\end{align}
where $A\triangle B:=(A-B)\cup (B-A)$ is the symmetric difference of two sets $A$ and $B$.  \blue{We first estimate $r_\e^+:=
~\Big| \O_t^+\triangle {\{x\mid \psi_\e>\tfrac1 2+\alpha  \}}\Big|$.}
\begin{align*}
r_\e^+=&~\Big| \O_t^+- {\{x\mid \psi_\e>\tfrac1 2+\alpha  \}}\Big|+\Big|  {\{x\mid \psi_\e>\tfrac1 2+\alpha  \}}- \O_t^+\Big|\\
=&~\Big| \(\O_t^+- \{x\in \O_t^+\mid \psi_\e>\tfrac1 2+\alpha  \} \)- \{x\in \O_t^-\mid \psi_\e>\tfrac1 2+\alpha  \}\Big|\\
&  +\Big|  {\{x\in \O_t^-\mid \psi_\e>\tfrac1 2+\alpha  \}}\Big|\\
\leq &~\Big|   {\{x\in \O_t^+\mid \psi_\e\leq \tfrac1 2+\alpha  \}}\Big|+ \Big| {\{x\in \O_t^-\mid\psi_\e>\tfrac1 2+\alpha  \}}\Big|.
\end{align*}
Now using  Chebyshev's inequality and \eqref{volume convergence}, we find    $r_\e^+\leq C \e^{1/4} $. Similar estimates apply to  $ |S_t^{\e,\alpha  }|$ and $r_\e^-:= | \O_t^- \triangle     {\{x\mid \psi_\e <\tfrac 1 2-\alpha  \}} |$. Substituting these estimates into \eqref{volume est3}, we find   
\begin{align}\label{volume est4}
&\left| \frac 1{ 2\alpha  }  \int_{\frac 12 -\alpha  }^{\frac 12 +\alpha  }   \Big(\mathcal{H}^{d-1}\(\{x\mid\psi_\e  =s\}\)-\mathcal{H}^{d-1} (I_t)\Big)\, ds\right|\leq  C\e^{1/4}\alpha  ^{-1}.
\end{align}
So  \eqref{areacompare1} follows from  Fubini's theorem.
 
To prove  \eqref{areacompare2}, we consider the set 
\[Q_t^{\e,\alpha  } =\{x\in \O\mid    |\psi_\e (x,t)- 2|\leq  \alpha   \}, \qquad \forall \alpha   \in (0,1/8),\label{local set1}\]
 Using \eqref{localcalibrationest1}  and the co-area formula, we have for   every $\alpha  \in (0,1/8)\backslash \mathcal{N}_t^\e$ that 
\begin{align*}
C\e & \geq  \int_{Q_t^{\e,\alpha  }} \(|\nabla\psi_\e|- \bxi \cdot \nabla\psi_\e\)\, dx \\
&=\int_{ 2 -\alpha  }^{ 2 +\alpha  }   \mathcal{H}^{d-1}\(\{x\mid \psi_\e  =s\}\)\, ds-\int_{\p Q_t^{\e,\alpha  }} \bxi \cdot \bnu \psi_\e  \, d\mathcal{H}^{d-1}+\int_{Q_t^{\e,\alpha  }} (\div \bxi ) \psi_\e  \, dx,
\end{align*}  
where  $ \mathcal{N}_t^\e\supset  \tilde{\mathcal{N}}_t^\e$ is a null set in $(0,1/8)$ and $\bnu$ is the outward   normal of  $\p Q_t^{\e,\alpha  }$.  Since $\psi_\e\leq 3$ on $Q_t^{\e,\alpha  }$, we have
$\left|\int_{Q_t^{\e,\alpha  }} (\div \bxi ) \psi_\e  \, dx\right| \leq C |Q_t^{\e,\alpha  }|,$
and thus  
\[\int_{ 2 -\alpha  }^{ 2 +\alpha  }    \mathcal{H}^{d-1}\(\{x\mid \psi_\e  =s\}\)\, ds\leq \left|  \int_{\p Q_t^{\e,\alpha  }} \bxi \cdot \bnu  \psi_\e  \, d\mathcal{H}^{d-1}\right|+C \e  + C |Q_t^{\e,\alpha  }|.\label{volume est1new}\]
Using  \eqref{bc n and H}, we  have   $\int_{\O}(\div \bxi)\, dx=0$, and thus  
\begin{align*}
&\int_{\p Q_t^{\e,\alpha  }} \bxi \cdot \bnu  \psi_\e  \, d\mathcal{H}^{d-1}\\
 = & \( 2-\alpha   \) \int_{\{x\mid  \psi_\e\geq  2-\alpha  \}} (\div \bxi )    \, dx+ \( 2 +\alpha   \)\int_{\{x\mid \psi_\e\leq   2+\alpha  \}} (\div \bxi  )   \, dx\\
 =&   \( 2-\alpha   \) \int_{\{x: \psi_\e\geq  2-\alpha  \}} (\div \bxi )   \, dx- \( 2 +\alpha   \)\int_{\{x\mid  \psi_\e>  2+\alpha  \}} (\div \bxi  )   \, dx.
\end{align*}
This combined with   Chebyshev's inequality and \eqref{volume convergence} implies that 
$$|Q_t^{\e,\alpha  }|+\left|  \int_{\p Q_t^{\e,\alpha  }} \bxi \cdot \bnu  \psi_\e  \, d\mathcal{H}^{d-1}\right|\leq C \e^{1/4}. $$
Substituting this in  \eqref{volume est1new} leads to 
\[\frac1 {2\alpha}\int_{ 2 -\alpha  }^{ 2 +\alpha  }    \mathcal{H}^{d-1}\(\{x\mid \psi_\e  =s\}\)\, ds\leq  C\e^{1/4}\alpha^{-1}.\]
So    \eqref{areacompare2} follows from  Fubini's theorem.
\end{proof}
 We end this section with the following   result concerning the convergence of $\uu_{\e}$.
\begin{prop}\label{global control prop}
 For every sequence   $\e_k\downarrow 0$  there exists a subsequence, which we will not relabel,   such that $\uu_k:=\uu_{\e_k}$ satisfies 
\begin{subequations} 
\begin{align}\label{global control1}
 \p_t \uu_k\wedge\uu_k    \xrightarrow{k\to \infty}& ~\p_t \uu\wedge \uu~\text{weakly in}~  L^2(0,T;L^{6/5}(\Omega)),\\
 \p_i \uu_k\wedge \uu_k    \xrightarrow{k\to \infty}&  ~ \p_i \uu\wedge \uu~\text{weakly-star in}~  L^\infty(0,T;L^{6/5}(\Omega)), ~1\leq i\leq 3,
\label{global control t}
\end{align}
\end{subequations}
  where     $\uu=\uu(x,t)$ satisfies  
\begin{subequations}\label{revision5.2}
\begin{align}
\uu &\in L^\infty(0,T;   W^{1,2}_{loc}\cap W^{1,6/5}(\O^+_t;\mathbb{S}^2)),\label{orientation integrable1} \\
\p_t \uu&\in  L^2(0,T; L^2_{loc}\cap L^{6/5}(\O^+_t)),\label{orientation integrable2}\\
 \uu(x,t)&=0 ~ \text{ for  every  }~ t\in [0,T]\text{ and for a.e. } x\in \O_t^-.\label{u equal 0 region}
\end{align}
\end{subequations}
  Furthermore,
\begin{subequations}\label{weak strong convergence}
\begin{align}
\p_t \uu_k\xrightarrow{ k\to\infty } \p_t \uu   &~\text{weakly in}~  L^2(0,T;L^2_{loc}(\O^\pm_t)),\label{deri con2}\\
\nabla \uu_k\xrightarrow{k\to\infty }  \nabla \uu   &~\text{weakly-star in}~  L^\infty(0,T;L^2_{loc}(\O^\pm_t)),\label{deri con}\\
  \uu_k\xrightarrow{k\to\infty }    \uu   & ~\text{strongly in}~  C([0,T];L^2_{loc}(\O^\pm_t)).\label{deri con1}
\end{align}
\end{subequations}
\end{prop}
 Before proving this result, we state  the Aubin--Lions--Simon lemma. See  \cite[Theorem 8.62, Exercise 8.63]{MR3726909}  or \cite[Corollary 8]{Simon1987} for the  proof.
\begin{lemma}\label{aubinlions}
Let $I \subset \mathbb{R}$ be an open bounded interval, let $\left(Y_0,\|\cdot\|_{Y_0}\right),\left(Y_1,\|\cdot\|_{Y_1}\right)$, and   $\left(Y_2,\|\cdot\|_{Y_2}\right)$ be Banach spaces with $Y_0 \hookrightarrow Y_1 \hookrightarrow Y_2$. Assume that the embedding $Y_0 \hookrightarrow Y_1$ is compact. Let   $\mathcal{V}$ be the Banach space of all functions $u \in L^\infty\left(I ; Y_0\right)$ whose distributional derivative $\p_t u$ belongs to $L^2\left(I ; Y_2\right)$ endowed with the norm
$$
\|u\|_{\mathcal{V}}:=\|u\|_{L^\infty\left(I ; Y_0\right)}+ \|\p_t u   \|_{L^2\left(I ; Y_2\right)} .
$$
Then the embedding $\mathcal{V} \hookrightarrow C\left(\bar{I} ; Y_1\right)$ is compact.
\end{lemma}

\begin{proof}[Proof of Proposition \ref{global control prop}]
Define  $\O^\pm:=\bigcup_{t\in (0,T]}\O_t^\pm\times \{t\}$. 
We first deduce from \eqref{energy bound4} and \eqref{energy bound0} that 
\[\| \p_t \uu_\e  -\Pi_{\uu_\e  }\p_t \uu_\e   \|_{L^2(0,T;L^2(\O))} 
 +\|  \nabla  \uu_\e  -\Pi_{\uu_\e  }\nabla \uu_\e \| _{L^\infty(0,T;L^2(\O))}\leq C\label{modulatedlocalapp1}
 \]
 for some $C$ independent of $\e $. On the other hand,  by   \eqref{projection1} we find 
\begin{equation}
    \Pi_{\uu_\e  }\p_i \uu_\e  (x,t)\wedge \uu_\e  (x,t)=0\quad\forall (x,t)\in \O\times (0,T)\label{wedge trick}
\end{equation}
for $0\leq i\leq 3$ where $\p_0:=\p_t$. Combining \eqref{wedge trick} and \eqref{modulatedlocalapp1} with       \eqref{L infinity bound1},
we deduce that 
\begin{align}\label{uniform bound wedge}
&\|  \p_t \uu_\e  \wedge\uu_\e   \|_{L^2(0,T;L^{6/5}(\O))}+\| \nabla  \uu_\e  \wedge \uu_\e   \|_{L^\infty(0,T;L^{6/5}(\O))}\nonumber\\
 =&\| (\p_t \uu_\e  -\Pi_{\uu_\e  }\p_t \uu_\e)\wedge  \uu_\e   \|_{L^2(0,T;L^{6/5}(\O))}\nonumber\\
&+\| (\nabla  \uu_\e  -\Pi_{\uu_\e  }\nabla \uu_\e ) \wedge \uu_\e    \|_{L^\infty(0,T;L^{6/5}(\O))}\leq C.
\end{align}
  So it  follows from  the Banach--Alaoglu theorem (cf. \cite[A.5.]{MR3726909}) that 
\blue{\[\label{convergenceg}
  \begin{split} 
 \p_t \uu_k\wedge\uu_k    \xrightarrow{k\to \infty}& ~\mathbf{g}_0~\text{weakly in}~  L^2(0,T;L^{6/5}(\Omega)),\\
 \p_i \uu_k\wedge \uu_k    \xrightarrow{k\to \infty}&  ~ \mathbf{g}_i~\text{weakly-star in}~  L^\infty(0,T;L^{6/5}(\Omega))
\end{split}
\]}
 where 
\[\label{gilimit1}\mathbf{g}_0\in L^2(0,T;L^{6/5}(\Omega))\text{ and }\{\mathbf{g}_i\}_{1\leq i\leq 3}\subset L^\infty(0,T;L^{6/5}(\Omega)).
\]
It follows from \eqref{L infinity bound1}, \eqref{space der bound local} and \eqref{time der bound local}  that,  for any fixed $ \delta\in (0,\delta_0)$,  up to extraction of   subsequences there exists   $\e_k=\e_k(\delta)\xrightarrow{k\to\infty}0$ such that
\begin{subequations}\label{udelta conv}
\begin{align}
  \uu_{\e_k}\xrightarrow{k\to\infty }  \uu&~\text{weakly-star in}~  L^\infty(0,T ;L^3(\O))\label{convergence L4},\\
   \uu_{\e_k}\xrightarrow{k\to\infty }  \bar{\uu}_\delta&~\text{weakly-star in}~  L^\infty(0,T ;L^3(\O^\pm_t\backslash B_{\delta}(I_t))),\label{convergence L4*}\\
\p_t \uu_{\e_k}\xrightarrow{ k\to\infty } \p_t \bar{\uu}_\delta&~\text{weakly in}~  L^2\Big(0,T;L^2(\O^\pm_t\backslash B_{\delta}(I_t))\Big),\label{convergence weak time der}\\
\nabla \uu_{\e_k}\xrightarrow{k\to\infty } \nabla \bar{\uu}_\delta&~\text{weakly-star in}~  L^\infty\Big(0,T;L^2(\O^\pm_t\backslash B_{ \delta }(I_t))\Big).\label{convergence weak gradient}
  \end{align}
\end{subequations}
 By \eqref{convergence L4} and \eqref{convergence L4*}, we have     $\uu=\bar{\uu}_\delta$ a.e. in $U^\pm(\delta):=\cup_{t\in [0,T]} \(\O_t^\pm\backslash B_{\delta}(I_t)\)\times \{t\}$. This combined with \eqref{convergence weak time der} and \eqref{convergence weak gradient} leads to 
\begin{equation}\label{regular limit}
\uu\in L^\infty(0,T;W^{1,2}_{loc}(\O^\pm_t))   ~\text{with}~\p_t \uu\in L^2(0,T;L^2_{loc}(\O^\pm_t)).
\end{equation}
 Furthermore, employing  \eqref{convergence L4*}-\eqref{convergence weak gradient} and   Lemma \ref{aubinlions}, we   obtain 
 \[ \uu_{\e_k}\xrightarrow{k\to\infty }  \bar{\uu}_\delta=\uu ~\text{strongly in}~  C([0,T];L^2(\O^\pm_t\backslash B_{\delta}(I_t))).\label{convergence strong conti}\]
 By passing to a sequential limit  $\delta=\delta_\ell \xrightarrow{\ell\to 0} 0$ and  by a diagonal argument  we obtain \eqref{weak strong convergence}  up to extraction of   subsequences.
 
 Now we turn to the proof of \eqref{revision5.2}.  Using \eqref{space der bound local}, \eqref{convergence strong conti} and Fatou's lemma, we deduce that 
$$f(|\uu|)=F(\uu)=F(\bar{\uu}_\delta)=0\text{ a.e. in }U^\pm(\delta)$$   for any fixed $\delta\in (0,\delta_0)$. This together  with       \eqref{bulk assump2} implies    that  
  $ \uu$ ranges in $\{0\}\cup \mathbb{S}^2$ a.e. in $\O\times (0,T)$. This combined with \eqref{volume convergence} and \eqref{regular limit} yields \eqref{u equal 0 region} and   
\begin{align}
\uu\in L^\infty (0,T;W^{1,2}_{loc}(\O^+_t;\mathbb{S}^2)) \text{  with  }\p_t \uu \in L^2(0,T;L^2_{loc}(\O^+_t)).\label{orientation}
\end{align}
Now we  show the integrability of $\nabla_{x,t} \uu$ up to the boundary. To this aim, we    choose a sequence of functions 
\begin{equation}\label{cut-off}
\{\eta_k(\cdot,t)\}_{k\geq 1}\subset  C_c^\infty(\O^+_t)~\text{ with }~\eta_k(\cdot,t)\xrightarrow{k\to\infty} \mathbf{1}_{\O^+_t}\text{ in }L^\infty(\O).
\end{equation}
By  \eqref{convergenceg} and \eqref{weak strong convergence}, we deduce that    for $0\leq i\leq 3$,
\begin{equation}\label{precise S}
\eta_k \mathbf{g}_i= \eta_k \p_i \uu\wedge \uu ~\text{   a.e. in }   \O\times (0,T).
\end{equation}
By  \eqref{gilimit1}  and  the dominated convergence theorem, we can take $k\to\infty$  and  get 
\[\mathbf{g}_i=\p_i \uu\wedge \uu ~~ \text{ a.e. in }~\Omega\times (0,T),~0\leq i\leq 3.\]
This and  \eqref{convergenceg} lead to \eqref{global control1} and \eqref{global control t}.
Since  $\uu$ maps $\O^+$ into $\mathbb{S}^2$, we have 
\[  |\p_i \uu|^2=|\p_i \uu\wedge \uu|^2=|\mathbf{g}_i|^2~~a.e.~\text{ in }~\O^+,~0\leq i\leq 3.\]
 This and  \eqref{gilimit1}   improve \eqref{orientation} and yield     \eqref{orientation integrable1} and \eqref{orientation integrable2}.
\end{proof}

 
 \section{Proof of Theorem \ref{main thm}: Anchoring boundary condition}\label{anchoring}
 
 The inequalities  \eqref{initial preserve} and \eqref{con psi} have been proved in Theorem \ref{thm close energy with div} and in Theorem \ref{sharp l1}. The assertions  \eqref{reg limit}, \eqref{reguptobdy1} and \eqref{u equal 0 region1} have been proved in Proposition  \ref{global control prop} (cf. \eqref{deri con1} and \eqref{revision5.2}).
It remains to  verify    \eqref{anchoring bc}, and this will be done by applying Lemma \ref{sharp est of bdy} for  every $t\in [0,T]$ and by choosing an appropriate $\alpha$ outside  the null set $\mathcal{N}_t^{\e_k}\subset (0,1/8)$. 
For simplicity we shall   abbreviate $\psi_{\e_k}$ and $\uu_{\e_k}$ by $\psi_k$ and $\uu_k$ respectively.   For any $k\geq 1$ we can choose  $\beta_k\in [1/2,1]$ such that  $\alpha=\alpha_k:= \beta_k{\e_k^{1/8}}\notin  \mathcal{N}_t^{\e_k}$. Then by  Lemma \ref{sharp est of bdy} there exist      
  \[  b_{\e_k,\alpha _k}(t)=:b_k\in [\tfrac 1 2-\alpha _k,\tfrac 1 2+\alpha _k],\quad  q_{\e_k,\alpha _k}(t)=:q_k\in [2-\alpha _k,2+\alpha _k]\] such that 
  \[(b_k,q_k)\xrightarrow{k\to \infty}(\tfrac 12 ,2),\label{bkqk122}\] 
  and such that  the set
  \[\O_t^k:=\{x\in\O\mid  b_k<\psi_k(x,t) < q_k\}\text{ has finite perimeter}.\label{bk def}\]
 Moreover,  there exists   $C>0$ which is   independent of $t$ and the particular choice of the subsequence $\e_k$ such that  
\begin{subequations}
\begin{align}
 &\mathcal{H}^{d-1}(\{x\mid  \psi_k(x,t) =q_k\})  \leq  C {\e_k^{1/8}},\label{shrinking length}\\
& \Big|\mathcal{H}^{d-1}(\p\O_t^k)-\mathcal{H}^{d-1} (I_t)\Big|  \leq  2 C  {\e_k^{1/8}}. \label{sharp bdy est2}
\end{align}
\end{subequations}

 Using these level sets, we can prove the following  proposition which improves \eqref{weak strong convergence} to   the convergence of $\uu_k$  up to the boundary $I_t$.
 \begin{prop}\label{prop con sbv}
 Let $\uu$ be the limit vector field   in Proposition \ref{global control prop}.
 For a.e.   $t\in [0,T]$,  up to   extraction of  subsequences  which we will not relabel, we have 
\begin{subequations}
\begin{align}
 \mathbf{1}_{   \bar{\O}_t^k} \widehat{\uu}_k &\xrightarrow{k\to\infty }    \1_{\O_t^+}\uu  ~\text{ weakly-star in } BV(\O),\label{strong convergence of u1}\\
\mathbf{1}_{  \bar{\O}_t^k} \nabla \widehat{\uu}_k &\xrightarrow{k\to\infty }  \1_{\O_t^+}\nabla \uu  ~\text{ weakly  in } L^1(\O),\label{convergence diffusive}\\
 \mathbf{1}_{  \bar{\O}_t^k} \widehat{\uu}_k &\xrightarrow{k\to\infty }    \1_{\O_t^+}\uu  ~\text{ strongly  in } L^p(\O),\text{ for any fixed } p\in [1,\infty),\label{strong convergence of u3}
\end{align}
\end{subequations}
 where $\widehat{\uu}_k=\widehat{\uu}_{\e_k}$ is defined in \eqref{def uhat}.
\end{prop}
\begin{proof}
 We first claim that  there exists a positive constant $C_3$ depending only on $f$ (cf. \eqref{bulk assump})   such that the following  statement holds  for 
 any  $\delta\in (0,1/8)$:
\begin{align}
  |\uu_\e(x,t)|\geq C_3\delta\qquad &\qquad \forall x\in \{x:\psi_\e \geq \delta\}.\label{push away}
\end{align}
Indeed, by \eqref{f increase}, $f$ (and also $g$) is  increasing on $(0,s_0)$.
    If $|\uu_\e |\geq s_0$, we are done. Otherwise,
$$\delta \leq  \psi_\e=\int^{|\uu_\e |}_0g(s) \, ds\leq  |\uu_\e |  g(s_0),$$
which implies \eqref{push away}. This combined with \eqref{bound degree+orien} and \eqref{bk def} implies    
\[\sup_{t\in [0,T]}\int_{\O_t^k}   \left| \nabla \widehat{\uu}_k\right|^2\, dx\leq C\label{local gradient} \]
for $k$ sufficiently large. This and \eqref{bk def} imply   that the distributional derivatives  of  $ \vv_k(\cdot,t):= \mathbf{1}_{  \bar{\O}_t^k} \widehat{\uu}_k(\cdot,t)$  have no Cantor parts, and  the   absolute continuous parts $\{\mathbf{1}_{  \bar{\O}_t^k} \nabla \widehat{\uu}_k\}_{k\geq 1}$ is bounded in $L^2(\O)$. Moreover, their   jump parts enjoy the estimate
$$\int_{\p \O_t^k} |\vv_k(\cdot,t)-0|^2\,d\mathcal{H}^{d-1}\overset{\eqref{sharp bdy est2}}\leq C,$$
and $\{\vv_k(\cdot,t)\}_{k\geq 1}$ is bounded in $L^\infty(\O)$. 
With these properties,    it follows from   \cite{ambrosio1995new} (or \cite[Section 4.1]{MR1857292}) that   $\{\vv_k(\cdot,t)\}_{k\geq 1}$ is compact in $SBV(\O)$, the class of  special functions of bounded variation on $\O$. More precisely, there exists $\vv(\cdot,t)\in SBV(\O)$ s.t. $\vv_k\to\vv$ weakly-star in  $BV(\O)$ as $k\to \infty$,
and the absolute continuous  part of the gradient $\nabla^{a} \vv_k=\mathbf{1}_{  \bar{\O}_t^k} \nabla \widehat{\uu}_k$ converges weakly in $L^1(\O)$ to $\nabla^{a} \vv$.
To identify $\vv$, we use \eqref{volume convergence} to deduce that  $\1_{\bar{\O}_t^k}\to \1_{\O_t^+}$ in $L^1(\O)$ as $k\to\infty$. This  and \eqref{deri con1} yield  $\vv(\cdot,t)=\1_{ \O_t^+}~\uu(\cdot,t)$ a.e. in $\O$, and thus   \eqref{strong convergence of u1} and \eqref{convergence diffusive} are proved. Finally  by \eqref{strong convergence of u1}, the  compact embedding of $BV$ functions and   the $L^\infty$ bound we get \eqref{strong convergence of u3}. 
 \end{proof}
To proceed we define the following measures  {for} Borel sets $A\subset \O$:
\begin{subequations}\label{def measures}
\begin{align}
\theta(A)&=\mathcal{H}^{d-1} (A\cap I_t),\label{thetameasure}\\
   \theta_k(A)&=\int_{A\cap \O_t^k} |\nabla \psi_k|\, dx.\label{def measure mue1}
\end{align}
\end{subequations}
\begin{lemma}\label{referee2}
For a.e.  $t\in [0,T]$,  
\begin{align}
\theta_k \xrightarrow{k\to \infty} \frac 12 \theta \quad & \text{ weakly-star as Radon measures.} \label{weak radom}
\end{align}
\end{lemma}
\begin{proof}
We define  truncation functions
\begin{align}\label{truncation h}
T_k(s)=\left\{
\begin{array}{rl}
0\qquad \text{when } &s\leq b_k,\\
 s-b_k \qquad \text{when } &b_k\leq s\leq q_k,\\
 q_k-b_k\qquad \text{when } &s\geq q_k,
\end{array}
\right.\\
 T(s)=\left\{
\begin{array}{rl}
0\qquad \text{when } &s\leq   1/2,\\
 s-1/2 \qquad \text{when } &  1/2\leq s\leq 2,\\
3/2 \qquad \text{when } &s\geq 2.
\end{array}
\right.
\end{align}
By \eqref{bkqk122}, we have   $T_k\xrightarrow{k\to \infty}T$  uniformly on $\R$.  Moreover,  
\begin{subequations}
\begin{align}
& \nabla (T_k\circ\psi_k)=\nabla\psi_k ~\1_{\O_t^k}\qquad  \text{a.e. in }\O,\label{psi-bc}\\
&T_k\circ\psi_k\xrightarrow{k\to \infty} \tfrac 12 \1_{\O_t^+}\text{ strongly in }L^p(\O)\qquad\text{for any fixed}~p\in [1,\infty).\label{psi-b}
\end{align}
\end{subequations}
Indeed, by \eqref{regularityofsolution} and \eqref{df def} we know that $\psi_k(\cdot,t)\in C^1(\O)$. Also by \eqref{bk def} we have $T_k'\circ\psi_k=\1_{\O_t^k}$ for a.e. $x\in\O$.
Therefore,   \eqref{psi-bc} follows from  the  chain rule (cf. \cite[Proposition 3.24]{MR3099262}), while    \eqref{psi-b} follows from \eqref{volume convergence} and the dominated convergence theorem.  
By \eqref{localcalibrationest1}    we have for any $g\in C^1_c(\O)$ that 
 \begin{align*}
 \int_\O g\, d\theta_k\overset{\eqref{def measure mue1}}=\int_{  \O_t^k} g |\nabla \psi_k|\, dx&\overset{\eqref{localcalibrationest1}}=O(\e_k)+  \int_{  \O_t^k} g \, \bxi\cdot \nabla  \psi_k \, dx\\
 &\overset{\eqref{psi-bc}}=O(\e_k)+  \int_{\O} g \, \bxi\cdot \nabla (T_k\circ\psi_k)\, dx \\
 &~=O(\e_k)-  \int_{\O} \div (g   \bxi)  ~T_k\circ \psi_k\, dx.
 \end{align*}
 Recalling that $\bxi$ is the   inward  normal of $I_t$,  we  use \eqref{psi-b} to   pass to the limit in the above equations and  obtain
  \begin{align*}
 \lim_{k\to \infty}\int_\O g\, d\theta_k\overset{\eqref{psi-b}}= -\frac 12  \int_{\O_t^+} \div (g   \bxi) \, dx =\frac 12\int_{I_t} g\, d\mathcal{H}^{d-1}\overset{\eqref{thetameasure}}=\frac 12 \int_\O g\, d\theta, 
 \end{align*}
 for any $g \in C^1_c(\O)$.
 By approximation, one can  pass from $C^1_c(\O)$ to $C^0_c(\O)$, 
 and this proves  \eqref{weak radom}.   \end{proof}
Now we finish the proof of Theorem \ref{main thm} by verifying    \eqref{anchoring bc}.
The proof here is inspired by the blow-up  argument in  \cite{lin2020isotropic}. See also  \cite{MR1218685} for  the applications of such a method  in the study of quasi-convex functionals. 
\begin{proof}[Proof of \eqref{anchoring bc}]
For any $x_0 \in I_t$ and any $R>0$, it follows from     \eqref{strong convergence of u3}, \eqref{psi-b} and the dominated convergence theorem that 
\begin{align*}
 \lim_{k\to \infty} \int_{ B_R(x_0)  } \1_{\hat{\O}_t^k}\widehat{\uu}_k \cdot \frac{x-x_0}{|x-x_0|} T_k\circ\psi_k\,dx
=     \frac 12 \int_{ B_R(x_0)  } \1_{\O^+_t}\uu \cdot \frac{x-x_0}{|x-x_0|}  \,dx.
 \end{align*}
 We can   use spherical coordinate to rewrite the above two integrals in the form of  $\int_0^R\int_{\p B_r(x_0)}(\cdot) \,d\mathcal{H}^{d-1}dr$,  and then apply  Fubini's theorem. Therefore,    there exists   $r_{j} \downarrow 0$ such that for each $j$ we have 
\begin{align}
\lim _{k\to\infty} \int_{\partial B_{r_{j}}(x_0)  \cap \O_t^k } \widehat{\uu}_k \cdot \bnu  ~T_k\circ\psi_k\,d\mathcal{H}^{d-1} 
= \frac 12 \int_{\partial B_{r_{j}}(x_0) \cap \O^+_t} \uu \cdot \bnu\, d \mathcal{H}^{d-1}\label{boundary concentration}
\end{align}
where  $\bnu$ is   the  outward normal of $\p B_{r_j}(x_0)$.
 Moreover,  we can arrange $r_j$ such  that  $\theta(\p B_{r_{j}}(x_0))=0$. This combined with \eqref{weak radom} implies    that    
 \[
   \lim_{k\to \infty}\theta_k(B_{r_{j}}(x_0)) =  \frac 12 \theta(B_{r_{j}}(x_0)).  \label{weak radom1}\]
To proceed, we   use convexity to write,  for some $a_m,c_m\in \R$, that 
 \[ s^2=\sup_{m\in\mathbb{N}^+}\, (a_m s+c_m),\quad \forall s\in \R.\label{radom convex}\]
 (cf. \cite[Proposition 2.31]{MR1857292}).
For $\theta-a.e. ~x_0\in {\rm supp}(\theta)=I_t$, we have for each $m\geq 1$ that 
\begin{align}
0&\overset{  \eqref{cos law}}{=}    \lim_{k\to \infty} \int_{B_{r_{j}}(x_0)} \(\widehat{\uu}_k\cdot\nn_k \)^2\, d\theta_k\nonumber\\
&\overset{\eqref{radom convex}}{\geq}  \lim_{k\to \infty} \int_{B_{r_{j}}(x_0)} \(a_m \widehat{\uu}_k\cdot\nn_k +c_m\) \, d\theta_k\nonumber\\
&\overset{\eqref{normal diff}}=      \lim_{k\to \infty}  a_m \int_{B_{r_{j}}(x_0)} \1_{\O_t^k}  \widehat{\uu}_k\cdot \nabla\psi_k\, dx +c_m \theta_k (B_{r_{j}}(x_0))    \nonumber\\
&\overset{\eqref{psi-bc}}=       a_m\lim_{k\to \infty}\int_{B_{r_{j}}(x_0)\cap \O_t^k}    \widehat{\uu}_k\cdot\nabla   (T_k\circ\psi_k)\, dx +  \theta\left(B_{r_{j}}(x_0)\right)\frac {c_m}2.\label{recover bk}
\end{align}
Note that in the last step we also used  \eqref{weak radom1}.
It remains to compute the    integral in the last display of  \eqref{recover bk} under the limit $k\to \infty$ for fixed $j,m$. To this aim, we use    \eqref{psi-bc} and   integration  by parts    to find 
\begin{align}\label{lower last1}
   &\int_{B_{r_{j}}(x_0)\cap \O_t^k} \widehat{\uu}_k \cdot \nabla   (T_k\circ \psi_k) \, dx\\
 =& \int_{\p \(B_{r_{j}}(x_0)\cap \O_t^k\)} (\widehat{\uu}_k \cdot \bnu) ~  T_k\circ \psi_k \, d\mathcal{H}^{d-1}- \int_{B_{r_{j}}(x_0)}  \1_{\bar{\O}_t^k}(\div \widehat{\uu}_k)  ~ T_k\circ \psi_k \, dx\nonumber\\
 =&:A_k-B_k. \nonumber
\end{align}
Note that 
the  integrand of $A_k$  is  uniformly bounded in $L^\infty$.
To  compute the limit of $A_k$, we first deduce from   \eqref{truncation h}    that $T_k\circ \psi_k =0$ on the set  $\{x\in \O\mid  \psi_k(x,t)=b_k\}$ which has finite perimeter (cf. \eqref{sharp bdy est2}).    So we employ    \eqref{bk def}  to find 
\begin{align}
A_k= & \int_{\p B_{r_{j}}(x_0)\cap \O_t^k} (\widehat{\uu}_k \cdot \bnu )~ T_k\circ\psi_k \, d\mathcal{H}^{d-1}+\int_{ B_{r_{j}}(x_0)\cap \{x|\psi_k=q_k\} } (\widehat{\uu}_k \cdot \bnu) ~ T_k\circ \psi_k  \, d\mathcal{H}^{d-1}.
 \end{align}
 The limit of the first integral   is given  in    \eqref{boundary concentration}, and  that of the second   vanishes in the limit $k\to\infty$  by \eqref{shrinking length}. So we conclude that 
 \begin{align}\label{con ak}
\lim_{k\to \infty}A_k=  \frac 12 \int_{\partial B_{r_{j}}(x_0) \cap \O^+_t} \uu \cdot \bnu\, d \mathcal{H}^{d-1}.
\end{align}
Concerning the   integral $B_k$, by \eqref{convergence diffusive}    the sequence 
 $\{\1_{\bar{\O}_t^k}\div \widehat{\uu}_k\}_{k\geq 1}$ converges weakly in $L^1(\O)$. Moreover, $\{T_k\circ \psi_k\}_{k\geq 1}$ is uniformly bounded in $L^\infty$, and converges a.e. in $\O$ to $ \tfrac 12 \1_{\O_t^+}$, due to \eqref{psi-b}. Therefore,  applying      the Product Limit Theorem (cf. \cite{MR1014927} or \cite[pp. 169]{MR2683475}), we obtain 
\begin{align}
\lim_{k\to \infty}B_k=\frac 12 \int_{ B_{r_{j}}(x_0) \cap  \O^{+}_t} (\div \uu)\, dx.\label{lower last}
\end{align}
Using  \eqref{con ak} and  \eqref{lower last}, we can compute   the limit in  \eqref{lower last1} and find
\begin{align}
  &\lim_{k\to\infty}\int_{B_{r_{j}}(x_0)\cap \O_t^k} \widehat{\uu}_k \cdot \nabla   (T_k\circ \psi_k) \, dx \nonumber\\
 =& \frac 12 \int_{\partial B_{r_{j}}(x_0) \cap \O^+_t} \uu  \cdot \bnu \, d \mathcal{H}^{d-1} -\frac 12 \int_{\p\( B_{r_{j}}(x_0) \cap  \O^{+}_t\) }   \uu\cdot\bnu\, d \mathcal{H}^{d-1}\nonumber\\
=& \frac 12 \int_{  B_{r_{j}}(x_0) \cap \p \O^+_t} \uu  \cdot \bxi\, d \mathcal{H}^{d-1} \label{lower last2}
\end{align}
where    in the last step we used      $\bxi=-\bnu$ on  $\p \O^+_t$. Note that $\bxi$ is the   {inward} normal  of $I_t$ according to  \eqref{def:xi}, and $\O_t^+$ is the region  {enclosed by} $I_t$ with  {outward}   normal $\bnu$. 
Substituting \eqref{lower last2} into 
 \eqref{recover bk}  and then dividing the resulting inequality by $\theta\left(B_{r_{j}}(x_0)\right)$ and taking $j\to \infty$, we find 
\begin{align}
0&\geq   \lim _{j \rightarrow \infty}  \frac{a_m}{\theta\left(B_{r_{j}}(x_0)\right)}\frac 12 \int_{ B_{r_{j}}(x_0) \cap I_t} \uu \cdot\bxi\, d \mathcal{H}^{d-1} +\frac{c_m}2\nonumber\\
&\overset{\eqref{thetameasure}}= \frac {a_m}2 (\uu\cdot\bxi )(x_0)+\frac{c_m}2,\qquad \forall m\in  \mathbb{N}^+.
\end{align}
 This together with \eqref{radom convex}   implies that $(\uu\cdot\bxi)^2(x_0)=0$ for $\mathcal{H}^{d-1}$-a.e. $x_0\in I_t$.
 \end{proof}


\section{Proof of Theorem \ref{main thm oseen frank}: Oseen--Frank limit in the bulk}\label{sec har}
The method here is  inspired by \cite{du2020weak,MR4163316}, which has a 2D nature.
  We set $ \btau_\e:=\p_t \uu_\e$ and   write  \eqref{Ginzburg-Landau}   as 
\begin{align}\label{Ginzburg-Landau stat}
\btau_\e  &=\mu \nabla (\div \uu_\e)  +\Delta \uu_\e   -  \e^{-2} D F (\uu_\e  )\,\,\,\text{in}~ \Omega\times (0,T).
\end{align}
By Corollary \ref{coro space-time der bound} and Proposition \ref{global control prop} (cf. \eqref{deri con1}), for a.e.   $t_0\in (0,T)$  and for any compact set $K\subset \subset \O_{t_0}^+$, we have 
 \begin{subequations}
\begin{align}
 \int_K |\btau_\e|^2\, dx+ &\int_{K}\(\frac 12 |\nabla \uu_\e  |^2+\frac 1{\e^2} 
F( \uu_\e   )  \)\, dx   \leq \hat{c}^2\quad \text{ at }~t=t_0,\label{kortum bound big ball}\\
&\uu_{\e_k}(\cdot,t_0)\xrightarrow{k\to\infty }    \uu(\cdot,t_0)~\text{ strongly  in } L^2(K),\label{strongcovuu2}
\end{align}
\end{subequations}
 where  $\hat{c}=\hat{c}(K,t_0)>1$   is independent of    $\mu$ and  $\e$.
\begin{prop}\label{partial regularity}
\blue{Let $K$ be a compact set of $  \O_{t_0}^+$ and assume that   \eqref{kortum bound big ball} and \eqref{strongcovuu2} hold.}
There exists an   absolute constant  $\Lambda\in (0,1)$ with the following property: under the assumptions
\begin{subequations}
\begin{align}
\hat{\epsilon}<  \Lambda/ \hat{c}^2 ~\text{ and }~\mu<\Lambda,\label{defofepsilon0}\\
B_{2r}(x_0)\subset K~ \text{ with }~ r<1~ \text{ and  }~\\
\label{small energy of GL}
 \int_{B_{2r}(x_0)}\(\frac 12 |\nabla \uu_\e  |^2+\frac 1{\e^2} 
F( \uu_\e  )  \)\, dx  & \leq  \hat{\epsilon}^2~\text{ at }~t=t_0,
\end{align}
\end{subequations}
   there exists a subsequence $\e_k\downarrow 0$, which we will not relabel, such that 
\[\nabla \uu_{\e_k}(\cdot,t_0)\xrightarrow{k\to\infty} \nabla \uu(\cdot,t_0)\text{ strongly in }L^2(B_{r/2}(x_0)).\label{strongcovuu}\]

\end{prop}
 
We shall need the following inequality due to the special choice of $f$ in \eqref{bulk2}:
\[|f'(s)|^2\leq C_4 f(s),\qquad \forall s\geq 0,\label{bulk trick1}\]  
for some $C_4>1$. 
In the sequel  $C_4>1$ will also be used as  a generic constant that   might change  from line to line due to the  use of   the Sobolev embeddings and  elliptic  estimates. Note that $C_4$ is independent of $\mu$ and $r$.
\begin{lemma}\label{refereelemma1}
 Under the assumptions \eqref{kortum bound big ball} and      \eqref{small energy of GL}    with   a sufficiently small $\hat{\epsilon}$ (defined in \eqref{defepsilon1} below),  we have
\[3/4\leq  |\uu_\e(\cdot,t_0)| \leq 5/4 \text{ on } B_r(x_0)\quad \text{ for } \e\leq r/4.\label{rho near 11}\]
\end{lemma}

\begin{proof}
Without loss of generality, we assume $x_0=0$. For brevity we write $B_r(0)$  {as} $B_r$. Since all    arguments are  made at $t=t_0$, we shall suppress the time dependence of $\uu_\e$.  

{\bf Step 1.}
There exists  $\hat{C}>1$   depending  on $\hat{c}$ such that   for any $x_1 \in B_r$ we have
\[  |\uu_\e(x)-\uu_\e(y)|\leq \hat{C} \sqrt{\frac{|x-y|}\e},\quad \forall x,y\in   B_{\e}(x_1).\label{kortum holder}\]
To prove \eqref{kortum holder}, let  $\hat{\uu}_\e(z)=\uu_\e(x_1+\e z): B_2\to \R^3$. Then  we can write \eqref{Ginzburg-Landau stat}  {as} 
\[\mu \nabla\div \hat{\uu}_\e(z)  +\Delta \hat{\uu}_\e(z)=\e^2\btau_\e(x_1+\e z ) + D F (\hat{\uu}_\e(z)  ),  \quad z\in B_2.\label{kortum est p f} \]
It follows from \eqref{kortum bound big ball} and a change of variable that   $\{\e^2\btau_\e(x_1+\e \cdot )\}_{\e>0}$  is uniformly  bounded in $L^2(B_2)$.  Using \eqref{bulk trick1}, we can estimate 
\[\big \| D F (\hat{\uu}_\e  )\big\|^2_{L^2(B_2)}\overset{\eqref{bulk}}=\e^{-2}\big\|f'(|\uu_\e|  )\big\|^2_{L^2(B_{2\e}(x_1))} \leq \e^{-2}C_4 \int_{B_{2\e}(x_1)}    
F( \uu_\e  )\, dx\overset{\eqref{kortum bound big ball}}\leq \hat{c}^2 C_4. \]
Altogether, we prove that the terms on the   right-hand side of \eqref{kortum est p f} is bounded   in $L^2(B_2)$.
 Invoking  the interior estimate for  elliptic system  (cf.  \cite[Theorem 4.9]{MR3099262}), we obtain
\[\|\hat{\uu}_\e\|_{W^{2,2}(B_1)}\leq C_4( \hat{c}+\|\hat{\uu}_\e\|_{L^2(B_2)}).\label{morreyh2}\]
 Note that $C_4$ is independent of $\mu$. Now  we    estimate the last term by 
\begin{align*}
\|\hat{\uu}_\e\|_{L^2(B_2)}^2 & \leq C_4\(   1+ \e^{-2} \int_{B_{2\e}(x_1) \cap 
\{  x\mid |\uu_\e(x)|\geq 2 \}} \( | \uu_\e|-1\)^2\)  \\
& \overset{\eqref{bulk2}}\leq  C_4\(1+ \e^{-2} \int_{B_{2\e}(x_1)} f (| \uu_\e|) \)\overset{\eqref{kortum bound big ball}}\leq  (1+\hat{c}^2)C_4. 
\end{align*}
Substituting this estimate in  \eqref{morreyh2} and using Morrey's embedding $W^{2,2}\hookrightarrow C^{1/2}$, we obtain  $\|\hat{\uu}_\e\|_{C^{1/2}(\bar{B}_1)}\leq C_4\hat{c}$. Rescaling back, we find    \eqref{kortum holder}   with  $$\hat{C}:=C_4\hat{c}.$$

{\bf Step 2:}
 We claim that   with the choice    
 \[\hat{\epsilon}<   16^{-8} C_4^{-2}\hat{c}^{-2}=16^{-8} \hat{C}^{-2},\label{defepsilon1}\]      we have  {either}  \eqref{rho near 11} or 
\begin{align} 
  |\uu_\e| \leq 1/4 \text{ on } B_r\quad \text{ for } \e\leq r/4.\label{rho near 0}
\end{align}
Indeed, if neither  of them were valid, then
   \[ \exists ~\e\in (0,r/4)\text{ and } x_1\in B_r~\text{ s.t. }~|\uu_\e(x_1)|\in (1/4,3/4)\cup (5/4,+\infty).\label{contradiction1}\]
Since $16 \hat{\epsilon}<1$, it follows from \eqref{kortum holder} that 
\[|\uu_\e(x_1)-\uu_\e(x)| <    4^{-7}\qquad   \text{ for }  x\in B_{ 16 \hat{\epsilon}\e }(x_1).\label{lipuee1}\]
Using this and \eqref{bulk2}, we deduce one of the  following two cases for $x\in B_{16  \hat{\epsilon}\e}(x_1)$:

a) If $|\uu_\e(x_1)|>3$, then $|\uu_\e(x)|>2$ and thus $f( |\uu_\e(x)|)\geq 1$. 

b) If  $|\uu_\e(x_1)|\in   (1/4,3/4)\cup (5/4,3]$, then $f(|\uu_\e(x_1)|)\geq   1/{16}$. By the third condition in \eqref{bulk2} and \eqref{lipuee1}, we have $f(|\uu_\e(x)|)>   1/{32}$.

To summarize, we have the following inequality: 
\[F(\uu_\e(x))=f(|\uu_\e(x)|)>    1/{32}\qquad \forall x\in B_{16\hat{\epsilon}\e}(x_1).\]
Integrating this inequality over $B_{16\hat{\epsilon}\e}(x_1)$ and using the assumption $\e < r/4$, we find   
$$\e^{-2}\int_{B_{ 16\hat{\epsilon}\e}(x_1)} F(\uu_\e(x))\,dx> 8\pi\hat{\epsilon}^2.$$ However, this   would contradict \eqref{small energy of GL} since $B_{ 16\hat{\epsilon}\e}(x_1)\subset B_{2r}(x_0)$. So \eqref{contradiction1} is not valid and the claim is proved.

{\bf Step 3:}
We shall  rule out \eqref{rho near 0}.

Assuming \eqref{rho near 0}, we deduce from \eqref{bulk2} that  $F(\uu_\e)=|\uu_\e|^2$. By  \eqref{Ginzburg-Landau stat} we have  
\begin{align}\label{GL equ near 0}
 \mu \nabla (\div \uu_\e)  +\Delta \uu_\e   -  2\e ^{-2} \uu_\e =\btau_\e \,\,\,\text{in}~ B_r. 
\end{align}
For $z\in B_1$, we introduce  $\tilde{\uu}_\e(z):=\uu_\e(rz)$ and $\tilde{\btau}_\e(z):=\btau_\e(rz)$.  Then  
\begin{align}\label{GL equ near 1}
 \mu \nabla (\div \tilde{\uu}_\e)  +\Delta \tilde{\uu}_\e   -  2r^2\e ^{-2} \tilde{\uu}_\e =r^2\tilde{\btau}_\e  \,\,\,\text{in}~ B_1. 
\end{align}
By the interior estimate for  elliptic system, we have 
\begin{align}\label{H^2 ue}
\| \tilde{\uu}_\e\|_{W^{2,2}(B_{1/2 })}+r^2\e^{-2}\| \tilde{\uu}_\e\|_{L^2(B_{1/2 })}\leq C_4\( \|\tilde{\btau}_\e\|_{L^2(B_1)}+\| \tilde{\uu}_\e\|_{L^2(B_1)}\).
\end{align}
Indeed, one can adapt the proof of    \cite[Theorem 4.9]{MR3099262} to gain  the term $r^2\e^{-2}\| \tilde{\uu}_\e\|_{L^2(B_{1/2 })}$.   By   \eqref{H^2 ue}, \eqref{kortum bound big ball} and the conclusion in step 2, 
we find $$r^2\e^{-2} \|\uu_\e\|_{L^2(B_{r/2})}\leq C_4\(\| \btau_\e\|_{L^2(B_r)}+\| \uu_\e\|_{L^2(B_r)}\)\leq C_4(\hat{c}+1).$$
 This        implies   that   $\uu_\e\to 0$ strongly  in $L^2(B_{r/2})$, which  contradicts \eqref{volume convergence} since  $B_{r/2}\subset K\subset\subset \O_t^+$. Therefore,   we rule out \eqref{rho near 0} and obtain \eqref{rho near 11}.
\end{proof}

By \eqref{rho near 11},  
 we have polar   decomposition  $ \uu_\e=\rho_\e \vv_\e$ where
\[  \rho_\e=|\uu_\e|,\quad \vv_\e=\uu_\e/|\uu_\e|~\text{ in } ~B_r(x_0).\label{polar decom}\] 
We set    
\[\ww_\e:=(\vv_\e, \rho_\e),\label{defwvector}\]
and  define   the   projection   
\[\mathbf{a}_{\|}:=(\mathbb{I}_3-\vv_\e\otimes \vv_\e)\mathbf{a}\label{projectioninv}\]
 for  a vector field $\mathbf{a}$.

    \begin{lemma}\label{equivalencepolar}
    Under the assumptions $\e\leq r/4$,  \eqref{kortum bound big ball} and      \eqref{small energy of GL}   for   $\hat{\epsilon}$  defined in \eqref{defepsilon1},    $\rho_\e$  satisfies the following equation  in $B_r(x_0)$.
\begin{align} \label{polar decompose equ1}
 \Delta \rho_\e    -\e^{-2} f'(\rho_\e )&+ \mu  \nabla^2\rho_\e :  (\vv_\e  \otimes \vv_\e) +\mu\rho_\e (\vv_\e\cdot\nabla) \div \vv_\e \nonumber\\
  &=\btau_\e \cdot \vv_\e+ \mathcal{N}_{1,\e}(\nabla\ww_\e,\nabla\ww_\e),
\end{align}
 where $\mathcal{N}_{1,\e}(\cdot,\cdot):\R^{4\times 3}\times \R^{4\times 3}\mapsto \R$ is  bilinear  with uniformly bounded coefficients. Also, $\vv_\e$ satisfies the following equation  in $B_r(x_0)$. 
\begin{align}\label{polar decompose equ2}
    \rho_\e \Delta \vv_\e &+\mu~ ((\nabla^2\rho_\e)  \vv_\e)_{\|}  +\mu \rho_\e  \(\nabla (\div \vv_\e)\)_{\|}\nonumber \\
     &= (\btau_\e)_{\|} +  \mathcal{N}_{2,\e}(\nabla\ww_\e,\nabla\ww_\e),
\end{align}
where    $\mathcal{N}_{2,\e}(\cdot,\cdot):\R^{4\times 3}\times \R^{4\times 3}\mapsto \R^3$  is bilinear with  uniformly bounded coefficients.

\end{lemma}
\begin{proof}
To simplify the presentation we will suppress  the subscript $\e$.
By \eqref{polar decom} we have    $|\vv|^2\equiv 1$ and thus 
 \[\Delta \vv\cdot\vv=-|\nabla \vv|^2.\label{bocher}\] 
Substituting  \eqref{polar decom}  into \eqref{Ginzburg-Landau stat}, we find  
\begin{align}\label{polar decompose equ}
\btau 
=&~(\Delta \rho)  \vv +2(\nabla\rho \cdot \nabla) \vv +\rho \Delta \vv -\e^{-2} f'(\rho ) \vv \nonumber\\
&+\mu~ (\nabla^2\rho ) \vv  +\mu \rho  \nabla (\div \vv) +\mu \(\nabla \rho \cdot \p_i \vv \)_{1\leq i\leq 3}   +\mu\nabla\rho  (\div \vv).
\end{align}
Testing  \eqref{polar decompose equ}  with   $\vv $ and using \eqref{bocher}, we obtain  
  \begin{align}\label{referee4}
&- \Delta \rho     +\e^{-2} f'(\rho )\nonumber\\
=& - \btau \cdot \vv+\mu  \nabla^2\rho :  (\vv  \otimes \vv) +\mu\rho (\vv\cdot\nabla) \div \vv\nonumber \\
&+\mu(\nabla\rho\cdot\p_i \vv)v_i+\mu(\vv\cdot\nabla\rho)\div \vv-\rho  |\nabla \vv |^2.
\end{align}
The terms in the last line  are bilinear   with respect to $\nabla\ww=(\nabla\vv,\nabla\rho)$, and we  denote their sum by $-\mathcal{N}_{1,\e}(\nabla\ww,\nabla\ww)$. By \eqref{rho near 11}, it has bounded coefficients and thus   \eqref{polar decompose equ1} is proved.

To derive \eqref{polar decompose equ2}, we shall use the following identities.
\[\vv_{\|}=0~\text{ and }~ (\p_i \vv)_{\|}=\p_i \vv.\label{wellknown1}\]
These combined with \eqref{bocher} lead to 
\[(\Delta\vv)_{\|}=\Delta \vv+|\nabla\vv|^2\vv.\label{wellknown2}\]
Now applying  $(\cdot)_{\|}$  to the equation  in  \eqref{polar decompose equ}, and using \eqref{wellknown1} and \eqref{wellknown2}, we obtain
  \begin{align}\label{referee5}
\btau_{\|} 
=&~ \rho \Delta \vv +\mu~ ((\nabla^2\rho)  \vv)_{\|}  +\mu \rho  \(\nabla (\div \vv)\)_{\|}  \nonumber\\
&+ 
2((\nabla\rho \cdot \nabla) \vv)_{\|} +\rho \vv|\nabla\vv|^2+\mu \Big(\(\nabla \rho \cdot \p_i \vv \)_{1\leq i\leq 3}\Big)_{\|}   +\mu(\nabla\rho)_{\|}  (\div \vv).
\end{align}
The terms in the second  line of the above equation are  bilinear  with respect to $\nabla\ww$,  and we  denote their sum by $-\mathcal{N}_{2,\e}(\nabla\ww,\nabla\ww)$. By \eqref{rho near 11}, it has bounded coefficients, and thus \eqref{polar decompose equ2} is proved.

\end{proof}

\begin{proof}[Proof of Proposition \ref{partial regularity}]
We first show that, by choosing $\hat{\epsilon}$ and $\mu$  sufficiently small, we have 
\begin{align}\label{CZ est3}
\|\nabla^2 (\vv_\e, \rho_\e)\|_{L^{4/3}(B_{r/2}(x_0))} 
\leq  \blue{2} C_4r^{-2}.
\end{align}
Recalling \eqref{defwvector}, 
we deduce from \eqref{small energy of GL} and \eqref{rho near 11} that   
\[\|\nabla \ww_\e\|_{L^2( B_r(x_0))}\leq  4\hat{\epsilon}~  \text{ on } B_r(x_0)~ \text{ for } \e\leq r/4.\label{OS flow1} \]
Recalling  that $r<1$, let $\chi$ be a  $C^2$ cut-off function such that 
\[\label{differentiatechi}
\chi\equiv \begin{cases}
1\text{ in }B_{r/2}(x_0)\\
0\text{ in } {B_1(x_0)}\backslash B_r(x_0)
\end{cases} \text{ and } |\nabla^\ell \chi|\leq 8r^{-\ell} \text{ in } B_1(x_0)\text{ for } \ell\in \{1,2\}.
\] and let    
$\bar{\ww}_\e:=(\bar{\vv}_\e,\bar{\rho}_\e)$ with 
\[\bar{\rho}_\e=\chi (\rho_\e-1)\quad \text{ and }\quad  \bar{\vv}_\e=\chi \vv_\e.\label{bar def}\]
Multiplying   \eqref{polar decompose equ2} by  $\chi$ and using the linearity of $\mathbf{a}_{\|}$ about $\mathbf{a}$ (cf.   \eqref{projectioninv}), we find   
\begin{align}\label{CZ est}
\rho_\e \Delta\bar{\vv}_\e+&\rho_\e [\chi,\Delta] \vv_\e + \mu \Big([\chi,\nabla^2] (\rho_\e-1)  \vv_\e\Big)_{\|}+
\mu\Big( \nabla^2 \bar{\rho}_\e \vv_\e\Big)_{\|} \nonumber\\&+\mu  \rho_\e~ \Big([\chi,\nabla\div] \vv_\e\Big)_{\|} +\mu\rho_\e~\Big( \nabla (\div \bar{\vv}_\e)\Big)_{\|}\nonumber \\
     &=    (\btau_\e)_{\|} \chi+  \mathcal{N}_{2,\e}(\chi \nabla\ww_\e,\nabla\ww_\e)\qquad  \text{ in } B_1(x_0),
\end{align}
and $\bar{\vv}_\e |_{\p B_1(x_0)}=0$.
  For  brevity we denote      $L^p(B_1(x_0))$ by $L^p$.
Note that the commutators in \eqref{CZ est} involve at most first order derivatives of $\ww_\e=(\vv_\e,\rho_\e)$, which satisfies   \eqref{OS flow1}.  Now applying  the $L^p$-estimate for  elliptic  equation    \cite[pp. 109]{MR2435520} (componentwise) in \eqref{CZ est}, and invoking  \eqref{OS flow1} and  \eqref{rho near 11}, we have 
 \begin{align}\label{polar decompose equ4}
\|\nabla^2 \bar{\vv}_\e\|_{L^{4/3}} 
\leq  ~&C_4\Big(  r^{-2} +r^{-1}+ \mu \|\nabla^2  \bar{\ww}_\e\|_{L^{4/3}}+\left\| \mathcal{N}_{2,\e}(\chi \nabla\ww_\e,\nabla\ww_\e)\right\|_{L^{4/3}}\Big).
\end{align}
Note that the prefactors $r^{-1}$ and $r^{-2}$ are due to the differentiation of $\chi$  (cf.  \eqref{differentiatechi}), and that $C_4$     is independent of $r$.
To estimate the last term, we employ the bi-linearity of $\mathcal{N}_{2,\e}$, \eqref{OS flow1} and \eqref{rho near 11}:
\[\label{polar decompose equ6}
\begin{split}
&\left\| \mathcal{N}_{2,\e}(\chi \nabla\ww_\e,\nabla\ww_\e)\right\|_{L^{4/3}}\\
\leq ~& \left\| \mathcal{N}_{2,\e}( \nabla\bar{\ww}_\e,\nabla\ww_\e)\right\|_{L^{4/3}}+\left\| \mathcal{N}_{2,\e}( \nabla\chi\otimes\ww_\e,\nabla\ww_\e)\right\|_{L^{4/3}}+C_4 r^{-1}\\
\leq  ~& C_4\(\|\nabla\bar{\ww}_\e\|_{L^4}\|\nabla\ww_\e\|_{L^2}+r^{-1}\)\\
\leq  ~& C_4\(\|\nabla^2\bar{\ww}_\e\|_{L^{4/3}}\|\nabla\ww_\e\|_{L^2}+r^{-1}\).
\end{split}
\]
Note that in the last step we used the Sobolev's embedding  $W^{1,4/3}(B_1)\subset  L^4(B_1)$.
Combining \eqref{polar decompose equ6} with \eqref{polar decompose equ4}, we obtain 
 \begin{align}\label{CZ est1} 
\|\nabla^2 \bar{\vv}_\e\|_{L^{4/3}} 
\leq  ~&C_4\(r^{-2}+ \mu \|\nabla^2 (\bar{\vv}_\e,\bar{\rho}_\e)\|_{L^{4/3}}+  { \|\nabla^2\bar{\ww}_\e\|_{L^{4/3}}}\|\nabla\ww_\e\|_{L^2}\).
\end{align}
Now we turn to the estimate of $\rho_\e$.
 Using \eqref{rho near 11} and  \eqref{bulk2},  we have $f'(\rho_\e)=2(\rho_\e-1)$ in $B_r(x_0)$. Now multiplying   \eqref{polar decompose equ1} by  $\chi$ and using the linearity of \eqref{projectioninv}, we find 
\begin{align*}
&-2\e^{-2} \bar{\rho}_\e+\Delta \bar{\rho}_\e+[\chi,\Delta] (\rho_\e-1)      +\mu (\vv_\e \otimes \vv_\e):  \nabla^2\bar{\rho}_\e\\
&+\mu (\vv_\e\otimes \vv_\e):  [\chi,\nabla^2](\rho_\e-1)  +\mu\rho_\e  \vv_\e \cdot (\nabla \div \bar{\vv}_\e) +\mu\rho_\e   \vv_\e \cdot \( [\chi,\nabla\div]\vv_\e\) \nonumber\\
  &=\chi \btau_\e \cdot \vv_\e+ \mathcal{N}_{1,\e}(\chi \nabla\ww_\e,\nabla\ww_\e).
\end{align*}
In the same way as we did for  \eqref{CZ est1}, we find 
 \begin{align*}
\|\nabla^2 \bar{\rho}_\e\|_{L^{4/3}} 
\leq  ~& C_4\(r^{-2}+ \mu \|\nabla^2 (\bar{\vv}_\e,\bar{\rho}_\e)\|_{L^{4/3}}+\left\| \mathcal{N}_{1,\e}(\chi \nabla\ww_\e,\nabla\ww_\e)\right\|_{L^{4/3}}\)\\
\leq  ~& C_4\(r^{-2}+\mu \|\nabla^2 (\bar{\vv}_\e,\bar{\rho}_\e)\|_{L^{4/3}}+ \|\nabla^2\bar{\ww}_\e\|_{L^{4/3}}\|\nabla\ww_\e\|_{L^2}\).
\end{align*}
Combining this with  \eqref{CZ est1} and \eqref{OS flow1} we discover  
\[\label{vrhoimproved}
\|\nabla^2 (\bar{\vv}_\e,\bar{\rho}_\e)\|_{L^{4/3}(B_r(x_0))} 
\leq C_4\(r^{-2}+ \max\{ \hat{\epsilon}, \mu\}\|\nabla^2(\bar{\vv}_\e,\bar{\rho}_\e)\|_{L^{4/3}(B_r(x_0))}\).
\]
\blue{
Note that before Lemma \ref{refereelemma1}, we have assumed that  $C_4>1$ and $\hat{c}>1$.  Recall also the choice of   $\hat{\epsilon}$ in  \eqref{defepsilon1}.  By choosing $$\Lambda= 16^{-8} C_4^{-2}\text{   in \eqref{defofepsilon0}},$$
we find $C_4 \max\{ \hat{\epsilon}, \mu\}<1/2$. This combined with   \eqref{vrhoimproved} yields  
$$\|\nabla^2 (\bar{\vv}_\e,\bar{\rho}_\e)\|_{L^{4/3}(B_r(x_0))} 
\leq 2C_4 r^{-2}.
$$
 In view of \eqref{differentiatechi} and \eqref{bar def}, this  implies   \eqref{CZ est3}.}

Now using \eqref{strongcovuu2}, we have $\rho_{\e_k}(\cdot,t_0)\xrightarrow{k\to\infty }    |\uu|(\cdot,t_0)=1$  strongly  in $L^2(B_r(x_0))$. Thus, using  \eqref{rho near 11} we find  
$$\|\vv_{\e_k}-\uu\|^2_{L^2(B_r(x_0))}\leq 2\|\uu_{\e_k}-\uu \rho_{\e_k}\|^2_{L^2(B_r(x_0))}\xrightarrow{k\to\infty } 0.$$
These  together  with \eqref{CZ est3} and the Gagliardo–Nirenberg interpolation inequality yield 
\begin{align}\label{strongvvrho}
\( \vv_{\e_k}, \rho_{\e_k} \) \xrightarrow{k\to \infty} \(  \uu, 1\)  \text{ strongly  in } W^{1,2}(B_{r/2}(x_0)).
\end{align}
Finally, using \eqref{rho near 11}  and \eqref{strongvvrho} we find 
\begin{align*}
\nabla \uu_{\e_k}= \rho_{\e_k} \nabla \vv_{\e_k}+\vv_{\e_k} \nabla\rho_{\e_k} \xrightarrow{k\to \infty} \nabla \uu\text{ strongly in }L^2(B_{r/2}(x_0)),
\end{align*}
and  finish the proof of \eqref{strongcovuu}.

 \end{proof}

\begin{proof}[Proof of Theorem \ref{main thm oseen frank}]
We employ   the  covering argument in  \cite{MR990191}.  For  any test function $\bPsi\in C_c^1(\O_t^+;\R^3)$, we choose $K=\overline{\mathrm{supp} (\bPsi)}\subset\subset \O_t^+$, and  we      define the singular set  at time $t \in(0, T]$ by
\[\label{discretesingularset}
\Sigma_t:=\bigcap_{0<r<1}\left\{x \in K\mid B_{2r}(x)\subset K, \varliminf _{k \rightarrow \infty} \int_{B_{2r} (x )}\( \frac{1}{2}\left|  \nabla \uu_{\e_k}\right|^{2}+\frac{F(  \uu_{\e_k})}{\e_{k}^2} \)\, dx>\frac{\hat{\epsilon}^2} 2\right\}.
\]
We claim    that $\Sigma_t$ is discrete.  Indeed, choose  an arbitrary finite set  $\{y_j\}_{j=1}^J\subset \Sigma_t$ with mutually disjoint balls $\{B_{2r_j}(y_j)\}_{j=1}^J$ inside $K$ with $r_j<1/2$. Since $J$ is finite, there exists $k_J>0$ such that for any $k\geq k_J$ we have 
\[\int_{B_{2r_j} (y_j )}\( \frac{1}{2}\left|  \nabla \uu_{\e_k}\right|^{2}+\frac{F(  \uu_{\e_k})}{\e_{k}^2} \)\, dx>\frac{\hat{\epsilon}^2} 4\quad \text{ for all }j\in \{1,\cdots, J\}.\]
Combined with \eqref{kortum bound big ball}, this implies   
\[\hat{c}^2\geq \int_{\bigsqcup_{j=1}^J B_{2r_j} (y_j )}\( \frac{1}{2}\left|  \nabla \uu_{\e_k}\right|^{2}+\frac{F(  \uu_{\e_k})}{\e_{k}^2} \)\, dx>\frac{\hat{\epsilon}^2} 4J.\]
As a result, $J\leq 4\hat{c}^2\hat{\epsilon}^{-2}$  and thus   $\Sigma_t$ is discrete.
Therefore w.l.o.g. we can  assume that $\Sigma_t=\{x_0\}$  and $B_{2r}(x_0)\subset K$. 
Let $\eta\in C_c^1(B_2(0))$ be a cut-off function which $\equiv 1$ in $B_1(0)$. Then    
\[\bPsi_\delta(x):=\bPsi(x)\(1-\eta(\tfrac{ x-x_0}\delta)\)\xrightarrow{\delta\to 0} \bPsi(x)\text{ for any }x\neq x_0.\label{varphi a.e. cov}\] 
 It is obvious that $\bPsi_\delta=0$   in $B_{\delta}(x_0)$. By \eqref{discretesingularset} and Proposition \ref{partial regularity},  we have    
 \[\nabla \uu_{\e_k}\xrightarrow{k\to\infty} \nabla \uu\text{ strongly in }L^2(K\backslash B_\delta(x_0)).\]
Using these properties,  we can apply  $ \wedge \uu_{\e_k} \cdot \bPsi_\delta$ to both sides of   \eqref{Ginzburg-Landau}, integrate by parts and then send  $k\to \infty$:
\begin{align}\label{varphi conv}
&\int_\O \p_t \uu \wedge \uu  \cdot \bPsi_\delta\,dx+\mu  \int_\O \(\div \uu\) (\rot   \uu) \cdot\bPsi_\delta\,dx\nonumber\\
&+\int_\O (\nabla \uu \wedge \uu)\cdot \nabla\bPsi_\delta\,dx-\mu  \int_\O \(\div \uu \)  ( \rot\bPsi_\delta)\cdot \uu\,dx=0.
\end{align}
 Note that we have also used $\p_t \uu_{\e_k} \wedge \uu_{\e_k}\xrightarrow{k\to\infty} \p_t \uu \wedge \uu$ weakly in $L^2(0,T;L^{6/5}(\Omega))$, which is  due to  Proposition \ref{global control prop}. 
By \eqref{varphi a.e. cov} and  the regularity of $\uu$ (cf. \eqref{orientation integrable1} and \eqref{orientation integrable2}),  we can  send  $\delta\to 0$ in the first and the second integrals in \eqref{varphi conv} using the dominated convergence theorem. Concerning the third one, we have 
\begin{align}
 \int_\O (\nabla \uu \wedge \uu)\cdot \nabla\bPsi_\delta\,dx&= \int_\O \(1-\eta(\tfrac{ x-x_0}\delta)\)(\nabla \uu \wedge \uu)\cdot \nabla\bPsi\,dx\nonumber\\
&- \sum_{i=1}^3\int_{B_{2\delta}(x_0)} \frac {1}\delta  (\p_i\eta)(\tfrac{ x-x_0}\delta)  ~\p_i \uu \wedge \uu\cdot\bPsi\,dx.\label{passtolimit1}
\end{align}
  We claim that the second integral on the right-hand side  vanishes as $\delta\to 0$. Indeed, by  the  Cauchy--Schwarz  inequality we have  
\begin{align}
  &\left|\sum_{i=1}^3\int_{B_{2\delta}(x_0)} \frac {1}\delta  (\p_i\eta)(\tfrac{ x-x_0}\delta)  ~\p_i \uu \wedge \uu\cdot\bPsi\,dx\right|  \nonumber\\
  & \leq  C\|\bPsi\|_{L^\infty}  \|\nabla\eta\|_{L^2(B_2)}\| \nabla \uu\|_{L^2(B_{2\delta}(x_0))}\xrightarrow{\delta\to 0}0.
   \end{align}
Now using  $\lim_{\delta \to 0}\eta(\tfrac{ x-x_0}\delta)= 0$ for any $x\neq x_0$,  we can send $\delta \to 0$  in  \eqref{passtolimit1}  and obtain
$$\int_\O (\nabla \uu \wedge \uu)\cdot \nabla\bPsi_\delta\,dx\xrightarrow{\delta\to 0}\int_\O  (\nabla \uu \wedge \uu)\cdot \nabla\bPsi\,dx.$$
By the same  argument we can compute  the fourth integral in \eqref{varphi conv} and find
\[   \int_\O \(\div \uu \)  ( \rot\bPsi_\delta)\cdot \uu\,dx \xrightarrow{\delta\to 0}\int_\O \(\div \uu \)  ( \rot\bPsi)\cdot \uu\,dx.\]
 Using  the above two formulas, we can send  $\delta\to 0$ in  \eqref{varphi conv} and obtain  \eqref{weak OF flow}.
 \end{proof}

 \noindent{\it Acknowledgements}.   
   Y. Liu is partially supported by NSF of China under Grant  11971314.
We thank an   anonymous referee for meticulous feedback and invaluable comments.
\appendix

\section{Proof of Proposition \ref{prop initial data}}\label{app initial}

 \begin{proof}[Proof of Proposition \ref{prop initial data}]
 
We first recall that  $\sigma=1$ (cf.  \eqref{normalization}), $I_0\subset \O$ is the initial interface and   $\eta_0$ is the cut-off function in \eqref{cut-off eta delta}. Then we define
\begin{equation}\label{cut-off initial}
s_\e (x):= \eta_0\(  x \)\theta\(\frac{d_{I_0}(x)}\e \)+\Big(1-\eta_0\(  x \)\Big)\1_{\O^+_0},
\end{equation}
where  $\theta(z)$   is  the solution of the ODE
\begin{align}
-\theta''(z)+f'(\theta)=0,\quad \theta(-\infty)=0,~\theta(+\infty)=1.\label{travelling wave}
\end{align}
We note that $d_{I_0}$ is Lipschitz continuous in  $\O$,  and thus by Rademacher's theorem we have   $|\nabla d_{I_0}|\leq 1$ a.e. in $\O$.
Recalling   \eqref{initialvectorfield}, we define
 \begin{equation}\label{sharp initial}
\uu_\e^{in} (x):=s_\e (x) \uu^{in}(x).
\end{equation}
One can verify that   $\uu_\e ^{in}\in W^{1,2}_0(\O)\cap L^\infty(\O)$, $\|\uu_\e ^{in}\|_{L^\infty(\O)}\leq 1$, and
\begin{align}\label{transition initial data}
\uu_\e ^{in}=\left\{
\begin{array}{rl}
 \uu^{in} & \quad \text{if}~ x\in \O^+_0\backslash B_{2\delta_0}(I_0),\\
\theta\(\frac{d_{I_0}}\e \) \uu^{in} & \quad \text{if}~ x\in B_{\delta_0}(I_0),\\
0&\quad \text{if}~ x\in \O^-_0\backslash B_{2\delta_0}(I_0).
\end{array}
\right.
\end{align}
So the condition \eqref{bdd0} is verified.
To verify   the others, we first  compute the modulated energy  in \eqref{entropy} for   the initial  datum  $\uu_\e^{in} $.  We  write \eqref{cut-off initial}  as
\begin{equation}\label{shat}
s_\e(x)= \theta\(\frac{d_{I_0}(x)}\e \)  +\hat{s}_\e(x),
\end{equation}
where
$
\hat{s}_\e(x):=\(1-\eta_0\( x \)\)\(\1_{\Omega^+_0} -\theta\(\frac{d_{I_0}(x)}\e \)\)
$.
Invoking   \eqref{cut-off eta delta} and the exponential convergence  of $\theta(z)$ as $z\to\pm\infty$ (cf.  \eqref{travelling wave}), we deduce  that
\begin{equation}\label{hat S decay}
\|\hat{s}_\e\|_{L^\infty(\O)}+\|\nabla \hat{s}_\e\|_{L^\infty(\O)}\leq Ce^{-\frac C\e },
\end{equation}
 for some constant $C>0$ that only depends on $I_0$. By a Taylor's expansion, we find 
$$F(\uu_\e^{in} )=f(\theta +\hat{s}_\e)=f(\theta)+ O(e^{-C/\e }).$$
Combining  \eqref{sharp initial}, \eqref{shat} with  \eqref{hat S decay}, we obtain  
$$
|\nabla \uu_\e^{in} |^2= \e^{-2}\theta'^2 +\theta^2 |\nabla\uu^{in}|^2+O(e^{-C/\e })(|\nabla\uu^{in}|^2+1).
$$
Note that we have also employed the identities  $\p_{x_i} \uu^{in}\cdot\uu^{in}=0$ a.e. in $\O$.
Recalling  \eqref{psi},    we have   
\[\psi_\e\Big|_{t=0}=\int_0^{\theta+\hat{s}_\e}\sqrt{2f(s)}\, ds: \O\mapsto [0,1].\label{psiint}\]
So   we can compute 
\begin{align}\label{initial cal}
&\frac{\varepsilon}2 \left|\nabla \uu_\e^{in}\right|^2+\frac{1}{\varepsilon} F(\uu_\e^{in})-\bxi\cdot\nabla \psi_\e\Big|_{t=0}\nonumber \\
=&\frac 1{2\e} \theta'^2 +\frac{1}{\varepsilon} f(\theta  )-\e^{-1} \bxi\cdot \nn_{I_0} \theta' \sqrt{2f(\theta  )}+\frac \e 2\theta^2 |\nabla\uu^{in}|^2+O(e^{-C/\e })(|\nabla\uu^{in}|^2+1).
  \end{align}
 It follows from   \eqref{def:xi} that  $1-\bxi\cdot \nn_{I_0}=O(d_I^2)$.  So we have
 $$\e^{-1} \bxi\cdot \nn_{I_0} \theta' \sqrt{2f(\theta  )}=\e^{-1} \theta' \sqrt{2f(\theta)}+O(e^{-C/\e })+ \e^{-1} O(d_{I_0}^2)  \theta' \sqrt{2f(\theta)}.$$
 Note that the last term can be written as $$\e^{-1} O(d_{I_0}^2)  \theta' \sqrt{2f(\theta)}= O(\e)z^2  \theta'(z) \sqrt{2f(\theta(z))}\big|_{z= \frac{d_{I_0}(x)}{\e}  }.$$
 Substituting the above two equations into  \eqref{initial cal}, we find  
 \begin{align}\label{initial cal1}
&\int_\O\(\frac{\varepsilon}2 \left|\nabla \uu_\e^{in}\right|^2+\frac{1}{\varepsilon} F(\uu_\e^{in})-\bxi\cdot\nabla \psi_\e\)\, dx\nonumber \\
=&\int_\O \underbrace{\(\frac 1{2\e} \theta'^2 +\frac{1}{\varepsilon} f(\theta  )-\frac 1{\e}   \theta' \sqrt{2f(\theta)  }\)}_{=0}\,dx\nonumber\\
&+\int_\O \frac \e 2\theta^2 |\nabla\uu^{in}|^2\,dx+O(e^{-C/\e })\int_\O (|\nabla\uu^{in}|^2+1)\,dx +O(\e)\qquad \text{ at } t=0.
  \end{align}
  Note that the integrand of the first integral on the right-hand side of \eqref{initial cal1} vanishes due to  the identity  $\theta'^2(z)=2f(\theta(z))$, which follows from    \eqref{travelling wave}.
  Now we turn to the first term in \eqref{entropy}. Using \eqref{hat S decay} we can estimate 
 \[ |\div \uu_\e^{in}|^2\leq 2|\nabla\theta\cdot\uu^{in}|^2+2\theta^2|\div \uu^{in}|^2+O(e^{-C/\e })(1+|\div \uu^{in}|^2).\label{div cal}\]
 By  the   exponential decay of $\theta'(z)$ as $z\to\pm\infty$,  we deduce that   
  \[|\nabla\theta\cdot\uu^{in}|=\begin{cases}
   \Big|\frac{d_{I_0} }\e \theta'\(\frac{d_{I_0} }\e\)\frac{\uu^{in}\cdot\nn_{I_0}}{d_{I_0} }\Big|\leq C\Big| \frac{\uu^{in}\cdot\nn_{I_0}}{d_{I_0}}\Big|\quad \text{ in } B_{\delta_0}(I_0)\backslash I_0,\\
   \\
   \Big|\frac1\e \theta'\(\frac{d_{I_0} }\e\)\uu^{in}\cdot\nn_{I_0}\Big|\leq e^{-\frac C\e}\quad\qquad \text{ in } \O\backslash B_{\delta_0}(I_0).
  \end{cases}
\]
Using this,  \eqref{initialvectorfield}  and    Hardy's inequality (cf. \cite{MR1655516}), we find 
\begin{align}
\int_\O |\nabla\theta\cdot\uu^{in}|^2\, dx & \leq C \int_{I_0}\int_{-\delta_0}^{\delta_0}  \Big| \frac{\uu^{in}\cdot\nn_{I_0}}{d_{I_0}}\Big|^2\, dr\,d\mathcal{H}^{d-1}+C\nonumber\\
& \leq C\( \int_\O|\nabla \uu^{in}|^2\, dx+1\).
\end{align}
Combining this with \eqref{div cal} and \eqref{initial cal1} we derive  $E_\e  [\uu_\e^{in}  | I_0]\leq C\e$. Recalling  \eqref{earlycal1}, we have also obtained \eqref{bdd1}.
To  verify   \eqref{initial}, we shall compute \eqref{gronwall2new} at $t=0$. By \eqref{psiint}, we see that 
$$B  [\uu_\e^{in}  | I_0]= 2\int_\O   \Big( \tfrac{\chi+1}2-\psi_\e  \Big)\eta\circ d_I  \, dx.$$
We shall only give the estimate in $B_{\delta_0}(I_0)\cap \O_0^+$ because   the one in  $B_{\delta_0}(I_0)\cap \O_0^-$ follows in  the same way, and the one in $\O\backslash   B_{\delta_0}(I_0)$ is due to \eqref{hat S decay} and  the exponential convergence  of $\theta(z)$ at $\pm\infty$. 
\begin{align}
&\int_{B_{\delta_0}(I_0)\cap \O_0^+} |\psi_\e-1| d_{I}(x)\, dx\Big|_{t=0}\nonumber\\
&\overset{\eqref{psiint}}= \int_{B_{\delta_0}(I_0)\cap \O_0^+}\(\int_{s_\e(x)}^1 \sqrt{2f(s)}\, ds\) d_{I}(x)\, dx\Big|_{t=0}\nonumber\\
&\overset{ \eqref{hat S decay}}=\e\int_{B_{\delta_0}(I_0)\cap \O_0^+}\(\int_{\theta(\frac{d_{I}(x)}\e)}^1 \sqrt{2f(s)}\, ds\) \frac{d_{I}(x)}\e\, dx\Big|_{t=0}+O(e^{-C/\e })\leq C \e^2,
\end{align}
where the last step is due to  the exponential decay of $Q(z):=z\int_{\theta(z)}^1\sqrt{2f(s)}\, ds$  as $z\uparrow \infty$.
\end{proof}

\section{Proof of Proposition \ref{gronwallprop}}\label{appendix}

\begin{lemma}\label{lemma:expansion 1} The following identity holds:
 	\begin{align}
	&\int \nabla \HH: (\bxi  \otimes\nn_\e)\left|\nabla \psi_\e\right| \, d x
	-\int (\nabla \cdot \HH) \, \bxi   \cdot \nabla \psi_\e \, d x\nonumber\\
	=&\int \nabla \HH: (\bxi -\nn_\e ) \otimes\nn_\e\left|\nabla \psi_\e\right|\, d x+\int \HH_\e\cdot\HH |\nabla \uu_\e |\, d x \nonumber\\
	&+\int\nabla\cdot\HH  \( \frac{\e }2 |\nabla \uu_\e |^2  +\frac{1}\e  F (\uu_\e ) -|\nabla \psi_\e| \)\, d x +\int\nabla\cdot\HH ( |\nabla\psi_\e|-\bxi \cdot\nabla\psi_\e)\, d x\nonumber\\
	&-\int    (\nabla \HH)_{ij} \,\e  \(\p_i \uu_\e  \cdot \p_j \uu_\e    \)\, d x +\int \nabla \HH: (\nn_\e  \otimes\nn_\e)\left|\nabla \psi_\e\right|\, d x.\label{expansion1}
	\end{align}
\end{lemma}
\begin{proof}
	We
	introduce the  stress tensor $	(T_\e)_{ij}:= \( \frac{\e }2 |\nabla \uu_\e |^2 +\frac{1}{\e } F (\uu_\e )  \) \delta_{ij} - \e  \p_i \uu_\e \cdot \p_j \uu_\e.$
By \eqref{mean curvature app},
	we have the identity
	$\nabla \cdot T_\e
	   =\HH_\e |\nabla \uu_\e |.$
	Testing this identity with  $\HH$, integrating by parts and using \eqref{bc n and H}, we obtain
	\begin{equation*}
	\begin{split}
	&\int \HH_\e\cdot\HH |\nabla \uu_\e |\,d x  =- \int \nabla\HH \colon T_\e \,d x\\
	&=-  \int\nabla\cdot\HH  \( \frac{\e }2 |\nabla \uu_\e |^2 +\frac{1}{\e } F (\uu_\e )  \)\, dx+ \int  (\nabla\HH)_{ij} \, \e \(\p_i \uu_\e \cdot \p_j \uu_\e \) d x.
		\end{split}
	\end{equation*}
	So adding zero leads to
	\begin{equation*}
	\begin{split}
	&\int \nabla \HH: \nn_\e  \otimes\nn_\e\left|\nabla \psi_\e\right|d x\\
	&=\int \HH_\e\cdot\HH |\nabla \uu_\e |\, dx+\int\nabla\cdot\HH  \( \frac{\e }2 |\nabla \uu_\e |^2 +\frac{1}{\e } F (\uu_\e ) -|\nabla \psi_\e| \)\,d x+\int\nabla\cdot\HH |\nabla\psi_\e|\,d x\\
	&-\int  (\nabla\HH)_{ij}\,\e \(\p_i \uu_\e \cdot \p_j \uu_\e    \) d x+\int (\nabla \HH): (\nn_\e  \otimes\nn_\e)\left|\nabla \psi_\e\right|d x,
	\end{split}
	\end{equation*}
	which yields   \eqref{expansion1}.
\end{proof}  
 \begin{lemma}\label{lemma exact dt relative entropy}
	Under the assumptions of Theorem \ref{main thm}, the following identity holds:
\begin{subequations}\label{time deri 4}
		\begin{align}
		\frac{d}{d t} E&\left[\uu_\e   | I\right]
		+\frac 1{2\e }\int \(\e ^2 \left| \p_t \uu_\e    \right|^2-|\HH_\e |^2\)\,dx\nonumber\\
		&+\frac 1{2\e }\int \Big| \e  \p_t \uu_\e    -(\nabla \cdot \bxi )  D \dd^F    (\uu_\e  )  \Big|^2d x
		+\frac 1{2\e }\int \Big| \HH_\e -\e |\nabla \uu_\e  | \HH \Big|^2\,d x\nonumber \\
		=&\frac 1{2\e } \int \Big| (\nabla \cdot \bxi ) | D \dd^F    (\uu_\e  )|\nn_\e  +\e  |\Pi_{\uu_\e  } \nabla \uu_\e  | \HH\Big|^2\,d x\label{tail1}
		\\&+\frac \e {2} \int |\HH|^2\(|\nabla \uu_\e  |^2-|\Pi_{\uu_\e  }\nabla \uu_\e  |^2\)\,d x
		-\int \nabla \HH\cdot (\bxi -\nn_\e )^{\otimes 2}\left|\nabla \psi_\e\right|\,d x\label{tail2}\\
		&   +\int \(\nabla\cdot\HH\)  \( \frac{\e }2 |\nabla \uu_\e  |^2 +\frac{1}\e  F  (\uu_\e  ) -|\nabla \psi_\e | \)\,d x\label{tail4}\\
		&+\int\(\nabla\cdot\HH\)  \(1-\bxi \cdot \nn_\e \)|\nabla\psi_\e |\, d x+ \int (J_\e^1+J_\e^2)\, d x,\label{tail3}
		\end{align}
	\end{subequations}
	where  
	\begin{align}
	  J_\e^1
	:=&  \nabla \HH: \nn_\e  \otimes\nn_\e\(|\nabla \psi_\e |-\e  |\nabla \uu_\e  |^2\)\nonumber\\
	&+\e    \nabla \HH:(\nn_\e \otimes \nn_\e )\(
	|\nabla \uu_\e  |^2-|\Pi_{\uu_\e  } \nabla \uu_\e  |^2\)  \nonumber\\
	&-\sum_{ij}\e    (\nabla \HH)_{ij}   \Big((\p_i \uu_\e  -\Pi_{\uu_\e  } \p_i \uu_\e  )\cdot(\p_j \uu_\e  -\Pi_{\uu_\e  } \p_j \uu_\e  )\Big), \label{J1}\\
	J_\e^2:= &-  \(\p_t  \bxi +\left(\HH \cdot \nabla\right) \bxi +\left(\nabla \HH\right)^{\mathsf{T}} \bxi \)\cdot \nabla \psi_\e. \label{J2}
	\end{align}
\end{lemma}
\begin{proof}
We shall   employ the  Einstein summation convention by summing over repeated   indices.
	Using the energy dissipation law in   \eqref{dissipation} and adding zero, we find
		\begin{align}
	&\frac{d}{d t} E_\e  [ \uu_\e   | I]
	+\e \int |\p_t \uu_\e  |^2\,d x-\int (\nabla \cdot \bxi )     {D \dd^F }   (\uu_\e  )\cdot \p_t \uu_\e   \,d x\nonumber\\
	=&\int   \left(\HH  \cdot \nabla\right) \bxi \cdot\nabla \psi_\e \,d x
	+\int \left(\nabla \HH \right)^{\mathsf{T}} \bxi  \cdot\nabla \psi_\e \,d x+ \int J_\e^2\, d x.
	 \label{time deri 1}
	\end{align}
	By the symmetry of   $\nabla^2\psi_\e $  and the boundary conditions in  \eqref{bc n and H}, we have
	\begin{align*}
	\int \nabla \cdot (\bxi  \otimes \HH ) \cdot \nabla \psi_\e  \, d x =  \int \nabla \cdot (\HH \otimes \bxi ) \cdot \nabla \psi_\e  \, d x.
	\end{align*}
	Hence, the first integral on the right-hand side  of \eqref{time deri 1} can be rewritten as
	\begin{align*}
	&\int\left(\HH  \cdot \nabla\right) \,\bxi  \cdot \nabla \psi_\e \, d x\nonumber \\
	&=\int \nabla \cdot (\bxi  \otimes \HH ) \cdot \nabla \psi_\e  \, d x
	-\int (\nabla \cdot \HH ) \,\bxi  \cdot \nabla \psi_\e \, d x\nonumber \\
	&= \int(\nabla \cdot \bxi ) \,\HH  \cdot \nabla \psi_\e \, d x
	+\int(\bxi  \cdot \nabla) \,\HH  \cdot \nabla \psi_\e \, d x
	-\int (\nabla \cdot \HH ) \,\bxi   \cdot \nabla \psi_\e \,d x.
	\end{align*}
	Therefore,
	\begin{equation*}
	\begin{split}
	&\frac{d}{d t} E_\e  [ \uu_\e   | I]
	+\e \int |\p_t \uu_\e  |^2\,d x-\int (\nabla \cdot \bxi )     {D \dd^F }   (\uu_\e  )\cdot \p_t \uu_\e  \,d x \\
	=& \int (\nabla \cdot \bxi )\, \HH  \cdot \nabla \psi_\e  d x+\int (\bxi  \cdot \nabla) \,\HH  \cdot \nabla \psi_\e \,d x
	-\int (\nabla \cdot \HH )\, \bxi   \cdot \nabla \psi_\e \,d x\\
	&   +\int \nabla \HH : \(\bxi  \otimes\nn_\e\)\left|\nabla \psi_\e\right| d x+ \int J_\e^2\, d x.
	\end{split}
	\end{equation*}
	Now using   \eqref{expansion1}   to replace the third  and the fourth  integrals on the right-hand side of the above equation, we find 	\begin{align}\label{time deri 3-}
&\frac{d}{d t} E_\e  [ \uu_\e   | I]
	+\e \int |\p_t \uu_\e  |^2\,d x-\int  (\nabla \cdot \bxi )    {D \dd^F }   (\uu_\e  )\cdot \p_t \uu_\e  \,d x
	\\   =& \int (\nabla \cdot \bxi ) \,\HH   \cdot \nabla \psi_\e\,d x
	+\int (\bxi  \cdot \nabla)\, \HH  \cdot \nabla \psi_\e\,d x
	+\int \nabla \HH : (\bxi -\nn_\e  ) \otimes\nn_\e\left|\nabla \psi_\e\right|\,d x\nonumber\\
	&   +\int \HH_\e \cdot \HH |\nabla \uu_\e  |\,d x
	+ \int \nabla\cdot\HH  \( \frac{\e }2 |\nabla \uu_\e  |^2 +\frac{1}\e  F  (\uu_\e  ) -|\nabla \psi_\e | \)\,d x\nonumber\\&
	+\int \nabla\cdot\HH \(|\nabla\psi_\e |-\bxi \cdot \nabla\psi_\e \)\,d x-\int     (\nabla \HH)_{ij} \,\e  \(\p_i \uu_\e  \cdot \p_j \uu_\e    \)\, d x\nonumber\\
	&
	+\int \nabla \HH : \nn_\e   \otimes\nn_\e\left|\nabla \psi_\e\right|\,d x   +\int J_\e^2\, dx.\nonumber
	\end{align}
We shall show  that   $J_\e^1$ arises from the second  and the third  to last integrals by proving the following identity:
\[\Pi_{\uu_\e  } \p_i \uu_\e   \cdot \Pi_{\uu_\e  } \p_j \uu_\e=n_\e^i n_\e^j |\Pi_{\uu_\e  } \nabla \uu_\e|^2 ~\text{ a.e. in }\O,\label{cancel identity}\]
where $(n_\e^\ell)_{1\leq\ell \leq 3}=\nn_\e$.
Such an identity holds obviously on the set  $\{ x\mid \uu_\e=0\}$ by \eqref{projection1}. It also holds on   $\{ x \mid g(|\uu_\e|)>0\}$ due to the following identity which follows from \eqref{projection1} and \eqref{projectionnorm}:
\begin{align*}
\Pi_{\uu_\e  } \p_i \uu_\e   \cdot \Pi_{\uu_\e  } \p_j \uu_\e |  {D \dd^F }(\uu_\e)|^2=\p_i \psi_\e ~ \p_j \psi_\e =n_\e^i ~ n_\e^j |\Pi_{\uu_\e  } \nabla \uu_\e|^2|  {D \dd^F }(\uu_\e)|^2.
\end{align*}
On  the open set $\{ x\mid |\uu_\e|>0\}$ which includes $\{ x\mid |\uu_\e|=1\}$, we deduce from   \eqref{projection1} and \eqref{ADM chain rule} that  $\Pi_{\uu_\e  } \p_j \uu_\e=(\p_j |\uu_\e|) ~\uu_\e$. By   \cite[Theorem 4.4]{MR3409135} we have  $\p_j |\uu_\e|=0$ a.e. on $\{ x\mid |\uu_\e|=1\}$.   We thus complete  the proof of  \eqref{cancel identity}.

Now by \eqref{cancel identity} and   adding zero, we find  
	\begin{equation*}
	\begin{split}
	&   \nabla \HH : \nn_\e   \otimes\nn_\e\left|\nabla \psi_\e\right|-   (\nabla \HH)_{ij} \,\e  \(\p_i \uu_\e  \cdot \p_j \uu_\e    \)\\
	\overset{\eqref{projection1}}=&  \nabla \HH : \nn_\e   \otimes\nn_\e\left|\nabla \psi_\e\right|-   \e (\nabla \HH)_{ij}(\Pi_{\uu_\e  } \p_i \uu_\e   \cdot \Pi_{\uu_\e  } \p_j \uu_\e  ) \\
	& -   (\nabla \HH)_{ij} \,\e  \Big((\p_i \uu_\e  -\Pi_{\uu_\e  } \p_i \uu_\e  )\cdot(\p_j \uu_\e  -\Pi_{\uu_\e  } \p_j \uu_\e  )\Big) \overset{\eqref{J1}}=  J_\e^1~\text{ a.e. in }\O.
	\end{split}
	\end{equation*}
 Using  the identities  $\nabla \psi_\e=\nn_\e |\nabla \psi_\e |$  and  $  \nabla \HH   :(\bxi \otimes \bxi )=0$ (due to \eqref{normal H2}), we   merge  the second  and  the third   integrals on the right-hand side  of \eqref{time deri 3-}:
	\begin{align}\label{time deri 3}
	&\frac{d}{d t} E_\e  [ \uu_\e   | I]
	=-\e \int |\p_t \uu_\e  |^2\, dx+\int  (\nabla \cdot \bxi )    {D \dd^F }   (\uu_\e  )\cdot \p_t \uu_\e  \, dx\nonumber\\
	&+ \int (\nabla\cdot \bxi ) \,\HH   \cdot \nabla \psi_\e\, dx+\int \HH_\e \cdot \HH |\nabla \uu_\e  | \, dx -\int \nabla \HH : (\bxi -\nn_\e  )^{\otimes 2}\left|\nabla \psi_\e\right|\, dx\nonumber\\
	&   +\int \(\nabla\cdot\HH\)  \Big( \frac{\e }2 |\nabla \uu_\e  |^2 +\frac{1}\e  F  (\uu_\e  ) -|\nabla \psi_\e | \Big)\, dx\nonumber\\&+\int (\nabla\cdot\HH) \(1-\bxi \cdot \nn_\e  \)|\nabla\psi_\e |\, dx+ \int (J_\e^1+J_\e^2) \, dx.
	\end{align}
 Now we complete squares for the first four terms on the right-hand side  of \eqref{time deri 3}.
	Reordering terms, we have
	\begin{align*}
	\notag-&\e  |\p_t \uu_\e  |^2+   (\nabla \cdot \bxi )   D \dd^F    (\uu_\e  )\cdot \p_t \uu_\e  
	+ (\nabla \cdot \bxi ) \HH   \cdot \nabla \psi_\e+ \HH_\e \cdot \HH |\nabla \uu_\e  |
	\\\notag&= -\frac1{2\e } \Big(  |\e  \p_t \uu_\e  |^2 -2(\nabla \cdot \bxi )   D \dd^F    (\uu_\e  )\cdot \e  \p_t \uu_\e  
	+(\nabla \cdot \bxi )^2 |  D \dd^F    (\uu_\e  )|^2 \Big)
	\\\notag&\quad - \frac1{2\e } |\e  \p_t \uu_\e  |^2 + \frac1{2\e }(\nabla \cdot \bxi )^2 |  D \dd^F    (\uu_\e  )|^2
	+ (\nabla \cdot \bxi ) \HH   \cdot \nabla \psi_\e
	\\\notag&\quad - \frac1{2\e } \Big( |\HH_\e |^2 - 2 \e  |\nabla \uu_\e  | \HH_\e \cdot  \HH + \e ^2 |\nabla \uu_\e  |^2 |\HH|^2\Big)
	+ \frac1{2\e } \Big( |\HH_\e |^2 + \e ^2 |\nabla \uu_\e  |^2 |\HH|^2\Big)
	\\&\notag =  -\frac1{2\e } \Big|\e  \p_t \uu_\e  - (\nabla \cdot \bxi )  D \dd^F    (\uu_\e  ) \Big|^2
	- \frac1{2\e } \Big|\HH_\e  - \e  |\nabla \uu_\e  | \HH \Big|^2
	- \frac1{2\e }  |\e  \p_t \uu_\e  |^2 +\frac1{2\e }  |\HH_\e |^2
	\\&\quad + \frac1{2\e } \Big( (\nabla \cdot \bxi )^2  | D \dd^F    (\uu_\e  )|^2 + 2\e (\nabla \cdot \bxi )  \nabla \psi_\e  \cdot \HH + |\e  \Pi_{\uu_\e  }\nabla \uu_\e  |^2 |\HH|^2 \Big)
	\\\notag&\quad +\frac\e {2} \left(|\nabla \uu_\e  |^2- | \Pi_{\uu_\e  }\nabla \uu_\e  |^2\right)  |\HH|^2.
	\end{align*}
	Substituting this identity   into \eqref{time deri 3}, we arrive at \eqref{time deri 4}.
	 
\end{proof}

\begin{proof}[Proof of Proposition \ref{gronwallprop}]
	The proof here  is   the same as the case $\mu=0$, done in \cite[Lemma 4.4]{MR4284534}. This is because    the form of the energy dissipation law  \eqref{dissipation} remains unchanged  in  the presence of  the divergence term in \eqref{Ginzburg-Landau}. 
	
	
	We first  estimate the right-hand side  of \eqref{time deri 4} by $E_\e  [\uu_\e   | I]$ up to a constant that only depends on $I_t$. Concerning \eqref{tail1}, it follows from the  triangle inequality that
	\begin{equation*}
	\begin{split}
	 \int& \left|\frac1{\sqrt{\e }} (\nabla \cdot \bxi ) | D \dd^F    (\uu_\e  )|\nn_\e  +\sqrt{\e } |\Pi_{\uu_\e  } \nabla \uu_\e  | \HH\right|^2d x
\\	\leq &\int \left|(\nabla\cdot \bxi )
	\left(	   \frac{1}{\sqrt{\e }} | D \dd^F    (\uu_\e  )|-\sqrt{\e } |\Pi_{\uu_\e  } \nabla \uu_\e  |
	\right)\nn_\e \right|^2\,d x\\&+\int \Big|(\nabla\cdot \bxi )\sqrt{\e }  |\Pi_{\uu_\e  } \nabla \uu_\e  |(\nn_\e -\bxi )\Big|^2\,	d x	\\&+ \int \left|
	\big((\nabla \cdot \bxi ) \bxi  +\HH \big)  \sqrt{\e } |\Pi_{\uu_\e  } \nabla \uu_\e  | \right|^2\,d x.
	\end{split}
	\end{equation*}
	The first  integral on the right-hand side  of the above inequality is controlled using  \eqref{energy bound2}.  Due to the elementary inequality  $|\bxi  - \nn_\e |^2 \leq 2 (1-\nn_\e \cdot\bxi )$,   the second integral is controlled by \eqref{energy bound1}. The third integral can be treated using  the relation $\HH=(\HH\cdot\bxi ) \bxi +O(d_I(x,t))$ and \eqref{div xi H}. So it can be  controlled by \eqref{energy bound3}.

	The integrals in \eqref{tail2} can be controlled using \eqref{energy bound2} and \eqref{energy bound1}. The one  in \eqref{tail4} is controlled by \eqref{energy bound-1}.
The first term in \eqref{tail3} can be controlled using \eqref{energy bound1}.  It remains to estimate   \eqref{J1} and \eqref{J2}.  The integrals of the  last two terms defining $J_\e^1$ can be controlled  by    \eqref{energy bound0}. Therefore,
	\begin{align*}
	\int J_\e^1\, dx
	\overset{\eqref{energy bound0}}\leq &\int \nabla \HH: \( \nn_\e  \otimes (\nn_\e -\bxi )\) \(|\nabla \psi_\e |-\e  |\nabla \uu_\e  |^2\)\, dx\nonumber\\
	&+\int (\bxi \cdot \nabla) \HH\cdot \nn_\e \(|\nabla \psi_\e |-\e  |\nabla \uu_\e  |^2\)\, dx+ C E_\e  [\uu_\e   | I]\nonumber\\
	\overset{\eqref{normal H2}}\leq  &C\Big( \int  |\nn_\e -\bxi | \Big(\e  |\nabla \uu_\e  |^2-\e  |\Pi_{\uu_\e  }\nabla \uu_\e  |^2\Big)\, dx\\
	&+\int  |\nn_\e -\bxi | \left|\e  |\Pi_{\uu_\e  }\nabla \uu_\e  |^2-|\nabla \psi_\e |\right| \, dx\nonumber\\
	&+ \int\min \(d^2_I ,1\)  \(|\nabla \psi_\e |+\e  |\nabla \uu_\e  |^2\)\, dx+   E_\e  [\uu_\e   | I]\Big).
	\end{align*}
	  The first and the third integrals in the last display can be estimated using  \eqref{energy bound0} and \eqref{energy bound3} respectively. 
	Then  we employ \eqref{projectionnorm} to find 
	\begin{align*}
	\int J_\e^1\, dx
	 \leq  &\,C \Big( \int  |\nn_\e -\bxi | \left| \e  |\Pi_{\uu_\e  }\nabla \uu_\e  |^2-|\nabla \psi_\e |\right|\, dx+   E_\e  [\uu_\e   | I]\Big)\nonumber\\
	\overset{\eqref{projectionnorm}}= &\, C\Big(  \int  |\nn_\e -\bxi | \sqrt{\e } |\Pi_{\uu_\e  }\nabla \uu_\e  | \left| \sqrt{\e } |\Pi_{\uu_\e  }\nabla \uu_\e  |-\frac{1}{\sqrt{\e }}|   D \dd^F     (\uu_\e  )|\right|\, dx+  E_\e  [\uu_\e   | I]\Big).
	\end{align*}
	Finally applying the Cauchy-Schwarz inequality and then \eqref{energy bound2} and \eqref{energy bound1}, we obtain
$\int J_\e^1 \,dx\leq C  E_\e  [\uu_\e   | I].$
	As for $J_\e^2$ \eqref{J2}, we employ \eqref{xi der1} and  \eqref{energy bound3}  to obtain  $\int J_\e^2\,dx  \leq C E_\e  [\uu_\e  | I].$
  All in all,  we have proved that the right-hand side  of \eqref{time deri 4} is bounded by $E_\e  [\uu_\e   | I]$ up to a multiplicative constant which only depends on $I_t$.
\end{proof}


\begin{thebibliography}{10}

\bibitem{MR1308851}
N.~D. Alikakos, P.~W. Bates, and X.~Chen.
\newblock Convergence of the {C}ahn-{H}illiard equation to the {H}ele-{S}haw
  model.
\newblock {\em Arch. Rational Mech. Anal.}, 128(2):165--205, 1994.

\bibitem{ambrosio1995new}
L.~Ambrosio.
\newblock A new proof of the {SBV} compactness theorem.
\newblock {\em Calc. Var. Partial Differential Equations}, 3(1):127--137, 1995.

\bibitem{MR1857292}
L.~Ambrosio, N.~Fusco, and D.~Pallara.
\newblock {\em Functions of bounded variation and free discontinuity problems}.
\newblock Oxford Mathematical Monographs. The Clarendon Press, Oxford
  University Press, New York, 2000.

\bibitem{MR2401600}
L.~Ambrosio, N.~Gigli, and G.~Savar\'{e}.
\newblock {\em Gradient flows in metric spaces and in the space of probability
  measures}.
\newblock Lectures in Mathematics ETH Z\"{u}rich. Birkh\"{a}user Verlag, Basel,
  second edition, 2008.

\bibitem{ball2017mathematics}
J.~M. Ball.
\newblock Mathematics and liquid crystals.
\newblock {\em Mol. Cryst. Liq. Cryst.}, 647(1):1--27, 2017.

\bibitem{MR1205984}
G.~Barles, H.~M. Soner, and P.~E. Souganidis.
\newblock Front propagation and phase field theory.
\newblock {\em SIAM J. Control Optim.}, 31(2):439--469, 1993.

\bibitem{MR1655516}
H.~Brezis and M.~Marcus.
\newblock Hardy's inequalities revisited.
\newblock volume~25, pages 217--237 (1998). 1997.
\newblock Dedicated to Ennio De Giorgi.

\bibitem{MR1443865}
L.~Bronsard and B.~Stoth.
\newblock The singular limit of a vector-valued reaction-diffusion process.
\newblock {\em Trans. Amer. Math. Soc.}, 350(12):4931--4953, 1998.

\bibitem{MR1324400}
L.~A. Caffarelli and Y.~S. Yang.
\newblock Vortex condensation in the {C}hern-{S}imons {H}iggs model: an
  existence theorem.
\newblock {\em Comm. Math. Phys.}, 168(2):321--336, 1995.

\bibitem{MR1313011}
S.~J. Chapman, Q.~Du, and M.~D. Gunzburger.
\newblock On the {L}awrence-{D}oniach and anisotropic {G}inzburg-{L}andau
  models for layered superconductors.
\newblock {\em SIAM J. Appl. Math.}, 55(1):156--174, 1995.

\bibitem{MR1425577}
X.~Chen.
\newblock Global asymptotic limit of solutions of the {C}ahn-{H}illiard
  equation.
\newblock {\em J. Differential Geom.}, 44(2):262--311, 1996.

\bibitem{MR2754215}
X.~Chen, D.~Hilhorst, and E.~Logak.
\newblock Mass conserving {A}llen-{C}ahn equation and volume preserving mean
  curvature flow.
\newblock {\em Interfaces Free Bound.}, 12(4):527--549, 2010.

\bibitem{MR1100211}
Y.~G. Chen, Y.~Giga, and S.~Goto.
\newblock Uniqueness and existence of viscosity solutions of generalized mean
  curvature flow equations.
\newblock {\em J. Differential Geom.}, 33(3):749--786, 1991.

\bibitem{MR990191}
Y.~M. Chen and M.~Struwe.
\newblock Existence and partial regularity results for the heat flow for
  harmonic maps.
\newblock {\em Math. Z.}, 201(1):83--103, 1989.

\bibitem{MR1672406}
P.~{D}e Mottoni and M.~Schatzman.
\newblock Geometrical evolution of developed interfaces.
\newblock {\em Trans. Amer. Math. Soc.}, 347(5):1533--1589, 1995.

\bibitem{MR1014927}
R.~J. DiPerna and P.-L. Lions.
\newblock On the {C}auchy problem for {B}oltzmann equations: global existence
  and weak stability.
\newblock {\em Ann. of Math. (2)}, 130(2):321--366, 1989.

\bibitem{du2020weak}
H.~Du, T.~Huang, and C.~Wang.
\newblock Weak compactness of simplified nematic liquid flows in 2d.
\newblock {\em arXiv preprint arXiv:2006.04210}, 2020.

\bibitem{MR3409135}
L.~C. Evans and R.~F. Gariepy.
\newblock {\em Measure theory and fine properties of functions}.
\newblock Textbooks in Mathematics. CRC Press, Boca Raton, FL, revised edition,
  2015.

\bibitem{MR1177477}
L.~C. Evans, H.~M. Soner, and P.~E. Souganidis.
\newblock Phase transitions and generalized motion by mean curvature.
\newblock {\em Comm. Pure Appl. Math.}, 45(9):1097--1123, 1992.

\bibitem{MR1100206}
L.~C. Evans and J.~Spruck.
\newblock Motion of level sets by mean curvature. {I}.
\newblock {\em J. Differential Geom.}, 33(3):635--681, 1991.

\bibitem{Fei2023aa}
M.~Fei, F.~Lin, W.~Wang, and Z.~Zhang.
\newblock Matrix-valued {A}llen-{C}ahn equation and the
  {K}eller-{R}ubinstein-{S}ternberg problem.
\newblock {\em Invent. Math.}, 233(1):1--80, 2023.

\bibitem{MR4059996}
M.~Fei, W.~Wang, P.~Zhang, and Z.~Zhang.
\newblock On the {I}sotropic-{N}ematic phase transition for the liquid crystal.
\newblock {\em Peking Math. J.}, 1(2):141--219, 2018.

\bibitem{MR4072686}
J.~Fischer and S.~Hensel.
\newblock Weak-strong uniqueness for the {N}avier-{S}tokes equation for two
  fluids with surface tension.
\newblock {\em Arch. Ration. Mech. Anal.}, 236(2):967--1087, 2020.

\bibitem{fischer2020convergence}
J.~Fischer, T.~Laux, and T.~M. Simon.
\newblock Convergence rates of the {A}llen-{C}ahn equation to mean curvature
  flow: a short proof based on relative entropies.
\newblock {\em SIAM J. Math. Anal.}, 52(6):6222--6233, 2020.

\bibitem{MR1218685}
I.~Fonseca and S.~M\"{u}ller.
\newblock Relaxation of quasiconvex functionals in {${\rm BV}(\Omega,{\bf
  R}^p)$} for integrands {$f(x,u,\nabla u)$}.
\newblock {\em Arch. Rational Mech. Anal.}, 123(1):1--49, 1993.

\bibitem{MR3099262}
M.~Giaquinta and L.~Martinazzi.
\newblock {\em An introduction to the regularity theory for elliptic systems,
  harmonic maps and minimal graphs}, volume~11 of {\em Appunti. Scuola Normale
  Superiore di Pisa (Nuova Serie) [Lecture Notes. Scuola Normale Superiore di
  Pisa (New Series)]}.
\newblock Edizioni della Normale, Pisa, second edition, 2012.

\bibitem{MR4076075}
D.~Golovaty, M.~Novack, P.~Sternberg, and R.~Venkatraman.
\newblock A {M}odel {P}roblem for {N}ematic-{I}sotropic {T}ransitions with
  {H}ighly {D}isparate {E}lastic {C}onstants.
\newblock {\em Arch. Ration. Mech. Anal.}, 236(3):1739--1805, 2020.

\bibitem{MR3910590}
D.~Golovaty, P.~Sternberg, and R.~Venkatraman.
\newblock A {G}inzburg-{L}andau-type problem for highly anisotropic nematic
  liquid crystals.
\newblock {\em SIAM J. Math. Anal.}, 51(1):276--320, 2019.

\bibitem{HLWZZ}
J.~Han, Y.~Luo, W.~Wang, P.~Zhang, and Z.~Zhang.
\newblock From microscopic theory to macroscopic theory: a systematic study on
  modeling for liquid crystals.
\newblock {\em Arch. Ration. Mech. Anal.}, 215(3):741--809, 2015.

\bibitem{HardtKinderlehrerLin1986}
R.~Hardt, D.~Kinderlehrer, and F.-H. Lin.
\newblock Existence and partial regularity of static liquid crystal
  configurations.
\newblock {\em Comm. Math. Phys.}, 105(4):547--570, 1986.

\bibitem{MR1050529}
J.~Hong, Y.~Kim, and P.~Y. Pac.
\newblock Multivortex solutions of the abelian {C}hern-{S}imons-{H}iggs theory.
\newblock {\em Phys. Rev. Lett.}, 64(19):2230--2233, 1990.

\bibitem{MR1030675}
G.~Huisken.
\newblock Asymptotic behavior for singularities of the mean curvature flow.
\newblock {\em J. Differential Geom.}, 31(1):285--299, 1990.

\bibitem{MR1803974}
J.~E. Hutchinson and Y.~Tonegawa.
\newblock Convergence of phase interfaces in the van der
  {W}aals-{C}ahn-{H}illiard theory.
\newblock {\em Calc. Var. Partial Differential Equations}, 10(1):49--84, 2000.

\bibitem{MR1237490}
T.~Ilmanen.
\newblock Convergence of the {A}llen-{C}ahn equation to {B}rakke's motion by
  mean curvature.
\newblock {\em J. Differential Geom.}, 38(2):417--461, 1993.

\bibitem{IyerXuZarnescu2015}
G.~Iyer, X.~Xu, and A.~D. Zarnescu.
\newblock Dynamic cubic instability in a 2{D} {$Q$}-tensor model for liquid
  crystals.
\newblock {\em Math. Models Methods Appl. Sci.}, 25(8):1477--1517, 2015.

\bibitem{MR1050530}
R.~Jackiw and E.~J. Weinberg.
\newblock Self-dual {C}hern-{S}imons vortices.
\newblock {\em Phys. Rev. Lett.}, 64(19):2234--2237, 1990.

\bibitem{MR3353807}
R.~L. Jerrard and D.~Smets.
\newblock On the motion of a curve by its binormal curvature.
\newblock {\em J. Eur. Math. Soc.}, 17(6):1487--1515, 2015.

\bibitem{MR4163316}
J.~Kortum.
\newblock Concentration-cancellation in the {E}ricksen-{L}eslie model.
\newblock {\em Calc. Var. Partial Differential Equations}, 59(6):Paper No. 189,
  16, 2020.

\bibitem{MR2435520}
N.~V. Krylov.
\newblock {\em Lectures on elliptic and parabolic equations in {S}obolev
  spaces}, volume~96 of {\em Graduate Studies in Mathematics}.
\newblock American Mathematical Society, Providence, RI, 2008.

\bibitem{MR4284534}
T.~Laux and Y.~Liu.
\newblock Nematic-isotropic phase transition in liquid crystals: a variational
  derivation of effective geometric motions.
\newblock {\em Arch. Ration. Mech. Anal.}, 241(3):1785--1814, 2021.

\bibitem{MR3726909}
G.~Leoni.
\newblock {\em A first course in {S}obolev spaces}, volume 181 of {\em Graduate
  Studies in Mathematics}.
\newblock American Mathematical Society, Providence, RI, second edition, 2017.

\bibitem{LinWang2008}
F.~Lin and C.~Wang.
\newblock {\em The analysis of harmonic maps and their heat flows}.
\newblock World Scientific Publishing Co. Pte. Ltd., Hackensack, NJ, 2008.

\bibitem{lin2020isotropic}
F.~Lin and C.~Wang.
\newblock Isotropic-{N}ematic phase transition and liquid crystal droplets.
\newblock {\em Comm. Pure Appl. Math.}, 76(9):1728--1792, 2023.

\bibitem{MR1294333}
F.-H. Lin and C.~Poon.
\newblock On {E}ricksen's model for liquid crystals.
\newblock {\em J. Geom. Anal.}, 4(3):379--392, 1994.

\bibitem{MR4272911}
Y.~Liu, X.~Y. Lu, and X.~Xu.
\newblock Regularity of a gradient flow generated by the anisotropic
  {L}andau-de {G}ennes energy with a singular potential.
\newblock {\em SIAM J. Math. Anal.}, 53(3):3338--3365, 2021.

\bibitem{MR803255}
L.~Modica.
\newblock A gradient bound and a {L}iouville theorem for nonlinear {P}oisson
  equations.
\newblock {\em Comm. Pure Appl. Math.}, 38(5):679--684, 1985.

\bibitem{MR3495430}
A.~Pisante and F.~Punzo.
\newblock Allen-{C}ahn approximation of mean curvature flow in {R}iemannian
  manifolds {I}, uniform estimates.
\newblock {\em Ann. Sc. Norm. Super. Pisa Cl. Sci. (5)}, 15:309--341, 2016.

\bibitem{MR2253464}
M.~R\"{o}ger and R.~Sch\"{a}tzle.
\newblock On a modified conjecture of {D}e {G}iorgi.
\newblock {\em Math. Z.}, 254(4):675--714, 2006.

\bibitem{MR2683475}
L.~Saint-Raymond.
\newblock {\em Hydrodynamic limits of the {B}oltzmann equation}, volume 1971 of
  {\em Lecture Notes in Mathematics}.
\newblock Springer-Verlag, Berlin, 2009.

\bibitem{MR2440879}
N.~Sato.
\newblock A simple proof of convergence of the {A}llen-{C}ahn equation to
  {B}rakke's motion by mean curvature.
\newblock {\em Indiana Univ. Math. J.}, 57(4):1743--1751, 2008.

\bibitem{MR1399194}
W.~Schlag.
\newblock Schauder and {$L^p$} estimates for parabolic systems via {C}ampanato
  spaces.
\newblock {\em Comm. Partial Differential Equations}, 21(7-8):1141--1175, 1996.

\bibitem{Simon1987}
J.~Simon.
\newblock {Compact sets in the space {$L^p(0,T;B)$}}.
\newblock {\em Ann. Mat. Pura Appl. (4)}, 146:65--96, 1987.

\bibitem{MR1674799}
H.~M. Soner.
\newblock Ginzburg-{L}andau equation and motion by mean curvature. {I}.
  {C}onvergence.
\newblock {\em J. Geom. Anal.}, 7(3):437--475, 1997.

\bibitem{MR2040901}
Y.~Tonegawa.
\newblock Integrality of varifolds in the singular limit of reaction-diffusion
  equations.
\newblock {\em Hiroshima Math. J.}, 33(3):323--341, 2003.

\end{thebibliography}
\end{document}